\documentclass[oneside, reqno, pdf]{amsart}

\usepackage[T1]{fontenc}
\usepackage[utf8]{inputenc}

\usepackage{mathtools}
\usepackage{bm}

\usepackage[top=1.5in, bottom=1.5in, left=1.25in, right=1.25in]{geometry}
\usepackage[all]{xy}
\usepackage{amsmath, amsfonts}
\usepackage{amssymb}
\usepackage{amsthm}
\usepackage{amscd}
\usepackage{dsfont}
\usepackage{stmaryrd}
\usepackage{enumitem}
\usepackage{mathrsfs}
\usepackage{bm}

\usepackage{scalerel,stackengine}
\stackMath
\newcommand\reallywidehat[1]{%
\savestack{\tmpbox}{\stretchto{%
  \scaleto{%
    \scalerel*[\widthof{\ensuremath{#1}}]{\kern-.7pt\bigwedge\kern-.7pt}%
    {\rule[-\textheight/2]{1.5ex}{\textheight}}
  }{\textheight}%
}{0.9ex}}%
\stackon[2.25pt]{#1}{\tmpbox}%
}

\newcommand{\KK}{\mathbb{K}}

\newcommand{\I}{\mathbb  I}
\newcommand{\RR}{\mathbb  R}

\newcommand{\NN}{\mathbb  N}
\newcommand{\Z}{\mathbb  Z}

\newcommand{\C}{\mathbb  C}

\newcommand{\bra}[1]{{\llbracket {#1} \rrbracket}} 

\newcommand{\ind}[1]{\mathds{1}_{{#1}}}
\newcommand{\Tra}[1]{{\rm T}_{{#1}}}

\numberwithin{equation}{section}

\newtheorem{theorem}[equation]{Theorem}
\newtheorem{definition}[equation]{Definition}
\newtheorem{proposition}[equation]{Proposition}
\newtheorem{cor}[equation]{Corollary}
\newtheorem{lemma}[equation]{Lemma}

\newtheorem{claim}[equation]{Claim}

\theoremstyle{definition}
\newtheorem{remark}[equation]{Remark}

\DeclareMathOperator{\supp}{supp}

\DeclareMathOperator{\Rea}{Re}

\DeclareMathOperator{\N}{\mathbb{N}}

\DeclareMathOperator{\coe}{c}
\DeclareMathAlphabet{\mathpzc}{OT1}{pzc}{m}{it}

\usepackage{hyperref} 
\hypersetup{
    colorlinks=true,       
    linkcolor=blue,          
    citecolor=magenta,        
    filecolor=magenta,      
    urlcolor=cyan           
}

\usepackage{amssymb,amsmath,amsthm,verbatim,IEEEtrantools, mathtools}

\usepackage{bbm}

\usepackage{comment}

\usepackage{latexsym}

\usepackage{graphicx}

\def\be#1{\begin{equation}\label{#1}}
\def\bas{\begin{align*}}
\def\eas{\end{align*}}
\def\bi{\begin{itemize}}
\def\ei{\end{itemize}}

\usepackage{pdfpages}

\begin{document}
\title[Polynomial progressions in topological fields]{Polynomial progressions in topological fields}

\author{Ben Krause}
\address[Ben Krause]{
School of Mathematics,
University of Bristol, Bristol, BS8 1UG, England}
\email{ben.krause@bristol.ac.uk}

\author{Mariusz Mirek }
\address[Mariusz Mirek]{
Department of Mathematics,
Rutgers University,
Piscataway, NJ 08854-8019, USA 
\&
Instytut Matematyczny,
Uniwersytet Wroc{\l}awski,
Plac Grunwaldzki 2/4,
50-384 Wroc{\l}aw,
Poland}
\email{mariusz.mirek@rutgers.edu}

\author{Sarah Peluse}
\address[Sarah Peluse]{
Department of Mathematics,
Stanford University\\
450 Serra Mall \\
Building 380 \\
Stanford, CA 94305, USA}
\email{speluse@stanford.edu}

\author{James Wright}
\address[James Wright]{Maxwell Institute of Mathematical Sciences and the School of Mathematics, University of
Edinburgh, JCMB, The King's Buildings, Peter Guthrie Tait Road, Edinburgh, EH9 3FD, Scotland}
\email{j.r.wright@ed.ac.uk}

\newcommand{\ben}[1]{{\color{red}{#1}}}
\newcommand{\mariusz}[1]{{\color{blue}{#1}}}
\newcommand{\sarah}[1]{{\color{green}{#1}}}

\maketitle

\begin{abstract} Let $P_1, \ldots, P_m \in \KK[{\rm y}]$ be polynomials with distinct degrees, no
constant terms and coefficients in a general local field $\KK$. We give a
quantitative count of the number of polynomial progressions $x, x+P_1(y), \ldots, x + P_m(y)$
lying in a set $S\subseteq \KK$ of positive density. The proof relies on a general $L^{\infty}$
inverse theorem which is of independent interest. This inverse theorem implies a Sobolev improving
estimate for multilinear polynomial averaging operators which in turn implies our
quantitative estimate for polynomial progressions. This general Sobolev inequality has the potential to be
applied in a number of problems in real, complex and $p$-adic analysis.
\end{abstract}

\section{Introduction}\label{intro}
Szemer\'edi's famous theorem \cite{S} states that any set $S$ of integers with positive
(upper) density must necessarily contain arbitrarily long arithmetic
progressions. Quantitative versions have been obtained by several authors, first by
Roth~\cite{Roth} for three-term arithmetic progressions and by Gowers~\cite{G} in general,
with the current best bounds due to Bloom and Sisask~\cite{BloomSisask}, Kelley and Meke~\cite{KM}
in the three-term case, and Leng, Sah and Sawhney~\cite{LSS} for longer progressions (see also
Green and
Tao~\cite{GT}, and Gowers~\cite{G}). More generally, one can consider polynomial
progressions $x, x+P_1(y), \ldots, x + P_m(y)$ for $x,y \in {\mathbb Z}$ with $y\not= 0$
where $P_j \in {\mathbb Z}[{\rm y}]$ is a sequence of polynomials with integer
coefficients and no constant terms (the case of arithmetic progressions corresponding to
linear polynomials).  Bergelson and Leibman \cite{BL}, extending earlier work of
Bergelson, Furstenberg and Weiss \cite{BFW}, generalised Szemer\'edi's theorem to
polynomial progressions. Obtaining quantitative versions of Bergelson and Leibman's result
has been a challenging problem and no progress (outside a few results on 2-term
progressions) has been made until very recently.

Inspired by the earlier work of Bergelson, Furstenberg and Weiss, Bourgain obtained a
quantitative lower bound on the count of 3-term polynomial progressions in the setting of
the real field ${\mathbb R}$. He accomplished this by coupling a technique he developed in
his work on arithmetic progressions \cite{BO}, together with fourier-analytic methods.

\begin{theorem}[Bourgain \cite{B}]\label{bourgain} Given $\varepsilon>0$, there exists a $\delta(\varepsilon)>0$ such
that for any $N\ge 1$ and measurable set $S \subseteq [0,N]$ satisfying $|S\cap [0,N]|\ge \varepsilon N$, we have
\begin{align}\label{bourgain-3-term}
\bigl|\{(x,y) \in [0,N]\times[0,N^{1/d}] : x, x+y, x + y^d \in S \}\bigr| \ \ge \ \delta N^{1+1/d}.
\end{align}
In particular we have the existence of a triple $x, x+y$ and $x+y^d$ belonging to $S$ with $y$ satisfying
the gap condition $y \ge \delta N^{1/d}$.
\end{theorem}

The bound \eqref{bourgain-3-term} implies a quantitative multiple
recurrence result. Only recently have there been extensions to more
general 3-term progressions $x, x +P_1(y), x + P_2(y)$; see the work
of Durcik, Guo, and Roos \cite{D+} when $P_1(y) = y$ and general $P_2$
and of Chen, Guo, and Li \cite{CGL} for general
$P_1, P_2 \in {\mathbb R}[{\rm y}]$ with distinct degrees.  The
methods in these papers, using delicate oscillatory integral operator
bounds, seem limited to 3-term progressions. 

In another direction, Bourgain and Chang \cite{BC} gave quantitative
bounds for 3-term progressions of the form $x, x+y, x+y^2$ in the
setting of finite fields ${\mathbb F}_q$. This result was extended to
more general 3-term polynomial progressions by Peluse \cite{PEL-field}
and Dong, Li, and Sawin \cite{DLS}. The techinques in these papers,
using a Fourier-analytic approach which relies on sophisticated
exponential sum bounds over finite fields, also seem limited to 3-term
progressions.

By using new ideas in additive combinatorics, by-passing the need of inverse theorems for Gowers' uniformity norms
of degree greater than 2,
Peluse \cite{PEL0} recently made a significant advance, giving quantitative bounds for general polynomial progressions 
$x, x+P_1(y), \ldots, x+P_m(y)$ in ${\mathbb F}_q$
where $\{P_1,\ldots, P_m\} \subseteq {\mathbb Z}[{\rm y}]$ are linearly independent over ${\mathbb Q}$.

Inspired by this work, Peluse and Prendiville \cite{PP} obtained the first quantitative bounds for 3-term polynomal progressions
in the setting of the integers ${\mathbb Z}$. This has been extended recently to general polynomial progressions
$x, x+P_1(y), \ldots, x + P_m(y)$ with $P_j \in {\mathbb Z}[{\rm y}] $ having distinct
degrees by Peluse \cite{PEL}. So although
the first quantitative bounds for polynomial progressions were made in the setting of the real field ${\mathbb R}$,
we have seen major advances in both the finite field ${\mathbb F}_q$ and integer ${\mathbb Z}$ settings by employing
new ideas in additive combinatorics. 

One purpose of this paper is to rectify this situation for the continuous setting by establishing quantitative bounds
for general polynomial progressions in the real field ${\mathbb R}$, bringing it in line with the recent advances 
in the finite field and integer settings. 
Another purpose is to illustrate how one can marry these new ideas in additive combinatorics with other ideas,
notably from the work of Krause, Mirek and Tao \cite{K+}, to obtain compactness results for general multilinear polynomial
averaging operators which have implications for problems in euclidean harmonic analysis. These ideas and arguments
are robust enough to allow us to obtain quantitative bounds for polynomial progressions in a general local field.

\begin{theorem}\label{thm:main} Let $\KK$ be a local field with Haar measure $\mu$.
Let ${\mathcal P} = \{P_1, \ldots, P_m\}$ be a sequence of polynomials in $\KK[{\rm y}]$ with distinct degrees and
no constant terms and let $d$ denote the largest degree among the polynomials in ${\mathcal P}$. When $\KK$
has positive characteristic, we asssume the characteristic is larger than $d$.

For any $\varepsilon>0$, there exists a $\delta(\varepsilon, {\mathcal P}) > 0$ and $N(\varepsilon, {\mathcal P}) \ge 1$ such that
for any $N \ge N(\varepsilon, {\mathcal P})$ and measurable set $S \subseteq \KK$ satisfying $\mu(S \cap B_N) \ge \varepsilon N$, we have
\begin{align}\label{mult-recurrence}
\mu\bigl(\{(x,y) \in B_N \times B_{N^{1/d}} : x, x+P_1(y),\ldots x + P_m(y) \in S \}\bigr) \ \ge \ \delta N^{1+1/d}.
\end{align}
In particular we have the existence of a progression $x, x+P_1(y), \ldots, x+ P_m(y)$ belonging to $S$ with $y$ satisfying
the gap condition $|y| \ge \delta N^{1/d}$. The proof will show that we can take $\delta = \varepsilon^{C\varepsilon^{-2m-2}}$
for some $C = C_{\mathcal P} > 0$ and $N(\varepsilon, {\mathcal P}) = \varepsilon^{-C' \varepsilon^{-2m-2}}$ for a slightly
larger $C' > C_{\mathcal P}$.
\end{theorem}

When $\KK = {\mathbb R}$ is the real field, Theorem \ref{thm:main} extends the work in \cite{B}, \cite{D+} and \cite{CGL}
from 3-term polynomial progressions to general polynomial progressions albeit for large $N$, depending on
$\varepsilon$. 

When $\KK={\mathbb C}$, Theorem \ref{thm:main}
represents the first known results for complex polynomial progressions. The absolute value $|\cdot|$ used in
the statement of Theorem \ref{thm:main} is normalised so that we can express the result in this generality
(see Section \ref{local-fields}). For any sequence of complex polynomials
$\{P_1,\ldots, P_m\} \subseteq {\mathbb C}[{\rm z}]$ with distinct degrees and $P_j(0) = 0$, Theorem \ref{thm:main}
has the following consequence: given $\varepsilon>0$, there is a $\delta >0$ such that for sufficiently large $N$
and any set $S$ in the complex plane satisfying $|S\cap {\mathbb D}_N| \ge \varepsilon N^2$, we can find a 
progression of the form $w, w+ P_1(z), \ldots, w + P_m(z)$ lying in $S$ such that $|z| \ge \delta N^{2/d}$.

Important in our analysis are certain properties for $m+1$ linear forms formed from our collection
${\mathcal P} = \{P_1, \ldots, P_m\} \subseteq \KK[{\rm y}]$ of $m$ polynomials with distinct degrees,
say $1\le \deg(P_1) < \ldots < \deg(P_m) =: d$. Let $N \ge 1$ and consider the form
$$
\Lambda_{\mathcal P; N}(f_0,\ldots, f_m):=
\frac{1}{N^d} \int_{\KK^2}f_0(x)\prod_{i=1}^mf_{i}(x-P_i(y))d\mu_{[N]}(y)d\mu(x).
$$
Here $d\mu_{[N]}(y) = N^{-1} \ind{B_N(0)}(y) d\mu(y)$ is normalised measure on the ball $B_N(0)$
(we will describe notation used in the paper in Section \ref{sec:2}). 
The key result in the proof of Theorem \ref{thm:main} is the following $L^{\infty}$ inverse theorem
for $\Lambda_{\mathcal P; N}$ which is of independent interest.

\begin{theorem}[Inverse theorem for $(m+1)$-linear forms]
\label{thm:inverse-informal}
With the set-up above, let
$f_0, f_1,\ldots, f_m$ be $1$-bounded functions supported
on a ball $B\subset \KK$ of measure $N^d$. 
Suppose that
\begin{align*}
|\Lambda_{\mathcal P; N}(f_0,\ldots, f_m)|\ge\delta.
\end{align*}
Then there exists $N_1\simeq \delta^{O_{\mathcal P}(1)}N^{\deg(P_1)}$ such that
\begin{align*}
N^{-d}\big\| \mu_{[N_1]}*f_1\big\|_{L^1(\KK)} \gtrsim_{\mathcal P} \delta^{O_{\mathcal P}(1)}.
\end{align*}
\end{theorem}

The main application of Theorem \ref{thm:inverse-informal} for us will be to prove a precise structural
result for multilinear polynomial operators of the form
$$
A_N^{\mathcal P}(f_1,\ldots, f_m)(x) \ = \ \int_\KK f_1(x + P_1(y)) \cdots f_m(x + P_m(y)) \, d\mu_{[N]}(y).
$$
We will use ideas in the recent work of Krause, Mirek and Tao \cite{K+}
to accomplish this and consequently, we will be able to establish the following important Sobolev estimate.

\begin{theorem}[A Sobolev inequality for $A_N^{\mathcal P}$]
\label{sobolev-informal}
Let $1<p_1,\ldots, p_m<\infty$ satisfying
$\frac{1}{p_1}+\ldots+\frac{1}{p_m}=1$ be given. Then
for $N_j\simeq \delta^{O_{\mathcal P}(1)}N^{\deg(P_j)}$, we have
\begin{align*}
\|A_N^{\mathcal P}(f_1,\ldots,f_{j-1}, (\delta_0-\varphi_{N_j})*f_j,f_{j+1}\ldots, f_m)\|_{L^1(\KK)}
\lesssim \delta^{1/8}
\prod_{i=1}^{m}
\|f_i\|_{L^{p_i}(\KK)},
\end{align*}
provided $N\gtrsim \delta^{-O_{\mathcal P}(1)}$.
Here $\varphi_{N_j}$ is a smooth cut-off function such that ${\widehat{\varphi_{N_j}}}(\xi) \equiv 1$ for $\xi \in B_{{N_j}^{-1}}(0)$.
\end{theorem}

Following an argument of Bourgain in \cite{B} we will show how
Theorem \ref{sobolev-informal} implies Theorem \ref{thm:main}. Versions of Theorem \ref{sobolev-informal} for two real
polynomials $\{P_1, P_2\} \subseteq {\mathbb R}[{\rm y}]$ were established in \cite{B}, \cite{D+} and \cite{CGL} using
delicate oscillatory integral operator bounds. Our arguments are much more elementary in nature and do not
require deep oscillatory integral/exponential sum/character sum bounds outside a standard application of van der Corput
bounds (see \cite{FatS}) when $\KK = {\mathbb R}$ or Hua's exponential sum bound \cite{H} when
$\KK = {\mathbb Q}_p$ (which extends Mordell's classical bound
from the finite field setting to complete exponenial sums over ${\mathbb Z}/p^m{\mathbb Z}$) -- these bounds extend readily to any local field ${\mathbb K}$; see Section \ref{local-fields}. Furthermore
the Sobolev inequalities in \cite{D+} and \cite{CGL} were only established for certain sparse sequences of scales $N$. The
bound in Theorem \ref{sobolev-informal} holds for all sufficiently large scales $N$.

The Sobolev bound in Theorem \ref{sobolev-informal} potentially has many other applications.
See \cite{B} for a discussion on the implications of Theorem \ref{sobolev-informal} to compactness properties of
the multilinear operator $A_N^{\mathcal P}$. 
Pointwise convergence results for multilinear polynomial averages
are common applications of such Sobolev bounds. See \cite{CGL} where the Sobolev inequality is
used to prove the existence of polynomial progressions in sets of sufficiently large Hausdorff dimension. 
See also \cite{K}, \cite{LP}, \cite{Kel1}, \cite{Kel2} and \cite{CLP}.

Our results require the scales $N$ to be large. It would be interesting, for various applications,
to establish these results for small scales as well.

\section{Structure of the paper}\label{paper} After a review of analysis in the setting of local fields, including some essential but basic oscillatory integral bounds, we set up some
notation and detail some 
tools involving the Gowers uniformity norms. In Section \ref{preliminaries} we give some preliminary results
necessary to carry out the core arguments. In Section \ref{sec:inverse} we give the proof of Theorem \ref{thm:inverse-informal}
which is based on a PET (polynomial ergodic theorem) induction scheme and a degree lowering argument
developed by the third author in earlier work.
In Section \ref{sec:sobolev} we will prove Theorem \ref{sobolev-informal}. Finally in Section \ref{appendix}, 
we show how Theorem \ref{thm:main} follows as a consequence of Theorem \ref{sobolev-informal}.

 \section{Review of basic analysis on local fields}\label{local-fields}

\centerline{A basic reference for the material reviewed in this section is \cite{N}.}

Let $\KK$ be a locally compact topological field with a nondiscrete topology. Such fields are called
{\it local fields} and have
a unique (up to a positive multiple) Haar measure $\mu$. They also carry a nontrivial absolute
value $|\cdot|$ such that the corresponding balls $B_r(x) = \{y \in \KK: |y-x|\le r\}$ generate the topology.

Recall that an absolute value on a field $\KK$ is a map $|\cdot| : \KK \to {\mathbb R}^{+}$ satisfying 
$$
(a) \ |x| = 0 \ \Leftrightarrow \ x = 0, \ \ \ (b) \ |xy| = |x||y| \ \ \ {\rm and} \ \ (c) \ |x+y| \le C (|x| + |y|)
$$
for some $C\ge 1$. It is {\it nontrivial} if there is an $x\not = 0$ such that $|x| \not = 1$. 
Two absolute values $|\cdot|_1$ and $|\cdot|_2$ are said to be {\it equivalent} if there is a $\theta> 0$
such that $|x|_2 = |x|_1^{\theta}$ for all $x\in \KK$. Equivalent absolute values give the same topology. There
is always an equivalent absolute value such that the triangle inequality $(c)$ holds with $C = 1$. If $|\cdot|$ satisfies
the stronger triangle inequality $(c') \, |x+y| \le \max(|x|,|y|)$, we say that $|\cdot|$ is non-archimedean. Note
that if $|\cdot|$ is non-archimedean, then all equivalent absolute values are non-archimedean. The field $\KK$ is
said to be {\it non-archimedean} if the underlying absolute value (and hence all equivalent ones) is non-archimedean.
Otherwise we say $\KK$ is archimedean.

When $\KK$ is archimedean, then it is isomorphic to the real ${\mathbb R}$ or complex ${\mathbb C}$ field with the
usual topology. In this case Haar measure is a multiple of Lebesgue measure. When $\KK$ is non-archimedean, then it is a finite extension of a $p$-adic field ${\mathbb Q}_p$ in the charateristic zero case,
and a function field of Laurent series over a finite field in the positive characteristic case. Furthermore, the {\it ring of integers} $o_\KK := \{x \in \KK: |x|\le 1\}$
and the unique maximal ideal $m_\KK := \{x \in \KK : |x| < 1\}$ do not depend on the choice of absolute value 
(it is invariant when we pass to an equivalent absolute value). 
For any $\KK$, we normalise Haar measure
so that $\mu(B_1(0)) = 1$.

When $\KK$ is non-archimedean, the unique maximal ideal $m_\KK = (\pi)$ is principal and we call any generating
element $\pi$ a {\it uniformizer}. Furthermore the residue field $k := o_\KK/m_\KK$ is finite, say with $q$ elements. 
For $x\in \KK$, there is a unique $n\in {\mathbb Z}$ such that $x = \pi^n u$ where $u$ is a unit. 
We can go further and expand any
$x\in \KK$ as a Laurent series in $\pi$; $x = \sum_{j\ge -L} x_j \pi^j$ where each $x_j$ belongs to the residue field $k$.
If $x_{-L} \not=0$, then $x = \pi^{-L} u$ where $u = \sum_{j\ge -L} x_j \pi^{j+L}$ is a unit.

There is a choice of (equivalent) absolute value $|\cdot|$ such that $\mu(B_r(x)) \simeq r$ for all $r>0$ and $x\in \KK$.
When $\KK= {\mathbb R}$, we have $|x| = x \, {\rm sgn}(x)$ and when $\KK = {\mathbb C}$, we have $|z| = z{\overline{z}}$.
When $\KK$ is non-archimedean, then the absolute value $|x| := q^{-m}$ where $x = \pi^m u$ and $u$ a unit
has the property that its balls satisfy 
$\mu(B_r(x)) = q^n$ where $q^n \le r < q^{n+1}$ and so $\mu(B_r(x)) \simeq r$. We choose the absolute value with
this normalisation. 

We will need a couple simple change of variable formulae which we will use again and again:
$$
\int_\KK f(x + y) \, d\mu(x) \ = \ \int_\KK f(x) \, d\mu(x) \ \ {\rm and} \ \ 
\int_\KK f(y^{-1} x) \, d\mu(x) \ = \ |y| \, \int_\KK f(x) \, d\mu(x).
$$ 
The first follows from the translation invariance of the Haar measure $\mu$. For the second formula, the
measure $E \to \mu(yE)$ defined by an element $y\in \KK$ is translation-invariant and so by the uniqueness
of Haar measure, we have $\mu(yE) = {\rm mod}_{\mu}(y) \mu(E)$ for some nonnegative number ${\rm mod}_{\mu}(y)$,
the so-called {\it modulus} of the measure $\mu$. In fact $|y| := {\rm mod}_{\mu}(y)$ defines the absolute value
with the desired normalisation whose balls $B_r(x)$ satisfy $\mu(B_r(x)) \simeq r$. This proves the second change
of variables formula. There is one additional, more sophisticated, nonlinear change of variable formula which we will
need at one point but we will justify this change of variables at the time.

The (additive) character group of $\KK$ is isomorphic to itself. Starting with any non-principal character ${\rm e}$
on $\KK$, all other characters $\chi$ can be identified with an element $y\in \KK$ via $\chi(x) = {\rm e}(yx)$. We fix
a convenient choice for ${\rm e}$; when $\KK={\mathbb R}$, we take ${\rm e}(x) = e^{2\pi i x}$. When $\KK$
is non-archimedean, we choose ${\rm e}$ so that ${\rm e} \equiv 1$ on $o_\KK$ and nontrivial on $B_q(0)$; that is,
there is a $x_0$ with $|x_0| = q$ such that ${\rm e}(x_0) \not= 1$. The choice of ${\rm e}$ on ${\mathbb C}$
does not really matter but a convenient choice is ${\rm e}(z) = e^{2\pi i \Rea{z}}$.
We define the fourier transform
$$
{\widehat{f}}(\xi) \ = \ \int_{\KK} f(x) {\rm e}(-\xi x) \, d\mu(x).
$$
Plancherel's theorem and
the fourier inversion formula hold as in the real setting.

\subsection{An oscillatory integral estimate} For $P(x) = a_d x^d + \cdots + a_1 x \in \KK[{\rm x}]$, we will use the following oscillatory integral bound:
\begin{align}\label{osc-int-est}
|I(P)| \ \le \ C_d \, [\max_j |a_j| ]^{-1/d} \ \ \ {\rm where} \ \ \ I(P) \ = \ \int_{B_1(0)} {\rm e}(P(x)) \, d\mu(x).
\end{align}
When $\KK = {\mathbb R}$, it is a simple matter to deduce the bound \eqref{osc-int-est} from general oscillatory bounds due to van der Corput (see \cite{FatS}).
When $\KK = {\mathbb Q}_p$ is the $p$-adic field, then 
$$
I(P) \ = \ p^{-s} \sum_{x=0}^{p^s -1} e^{2\pi i Q(x)/p^s} \ \ \ {\rm where} \ \ \ p^s = \max_j |a_j| \ \ {\rm and} \ \ 
Q(x) = b_d x^d + \cdots + b_1 x \in {\mathbb Z}[{\rm x}] 
$$
satisfies ${\rm gcd}(b_d, \ldots, b_1, p) = 1$; hence a classical result of Hua~\cite{H} implies $|I(P)| \le C_d p^{-s/d}$ 
which is \eqref{osc-int-est} in this case. It is natural to extend Hua's bound to other non-archimedean fields; see
for example \cite{W-jga} where character sums are treated over general Dedekind domains
which in particular establishes \eqref{osc-int-est} for any non-archimedean field $\KK$ when the characteristic of $\KK$ (if positive) is larger than $d$, a basic assumption appearing in our main result Theorem \ref{thm:main}.

It is not straightforward to apply van der Corput bounds when $\KK = {\mathbb C}$. However we can see
the bound \eqref{osc-int-est} for both $\KK = {\mathbb R}$ and $\KK = {\mathbb C}$ as a consequence of the following general bound due to
Arkhipov, Chubarikov and Karatsuba \cite{ACK}: let
$P \in {\mathbb R}[X_1, \ldots, X_n]$ be a real polynomial of degree $d$ in $n$ variables. If ${\mathbb B}^n$ denotes the unit ball
in ${\mathbb R}^n$, then
\begin{align}\label{ACK}
\Bigl| \int_{{\mathbb B}^n} e^{2\pi i P({\underline{x}})} \, d{\underline{x}} \Bigr| \ \le \ C_{d,n} \, H(P)^{-1} \ \ \ {\rm where} \ \ \
H(P) = \min_{{\underline{x}}\in {\mathbb B}^n} \max_{\alpha} |\partial^{\alpha} P({\underline x})|^{1/|\alpha|}.
\end{align}
A simple equivalence of norms argument shows that $H(P) \ge c_d [\max_{\alpha} |a_{\alpha}|]^{1/d}$ where 
$P({\underline{x}}) = \sum_{\alpha} a_{\alpha} {\underline{x}}^{\alpha}$ and $d$ is the degree of $P$. 
Hence \eqref{ACK} implies \eqref{osc-int-est} when $\KK = {\mathbb R}$. When $\KK = {\mathbb C}$ and 
$f(z) = a_d z^d + \cdots + a_1 z \in {\mathbb C}[{\rm z}]$, write $f(x+iy) = P(x,y) + i Q(x,y)$ and note that
$$
\int_{B_1(0)} {\rm e}(f(z)) \, dz \ = \ \int_{{\mathbb B}^2} e^{2\pi i P(x,y)} \, dx dy
$$
for the choice of character ${\rm e}(z) = e^{2\pi i \Rea{z}}$.
From the Cauchy--Riemann equations, we have $H(P) \simeq_d \min_{|z|\le 1} \max_k |f^{(k)}(z)|^{1/2k} 
\ge c_d [\max_j |a_j|]^{1/2d}$ (recall we are using the absolute value $|z| = z{\overline{z}}$ on ${\mathbb C}$)
and so \eqref{ACK} implies \eqref{osc-int-est} with exponent $1/2d$ in this case. 
There is an alternative
argument which establishes \eqref{osc-int-est} with the exponent $1/d$ when $\KK = {\mathbb C}$ but this is unimportant
for our purposes.

\section{Some notation and basic tools}
\label{sec:2}
By a {\it scale} $N$, we mean a positive number when $\KK$ is archimedean and when $\KK$ is non-archimedean,
it denotes a discrete value $N = q^k, \, k\in {\mathbb Z}$, a power of the cardinality of the residue field $k$.
When $N$ is a scale, we denote by
$[N] := B_N(0)$ the ball with
centre $0$ and radius $N$. In this case, we have $\mu([N]) \simeq N$ (equality in the non-archimedean case) 
by our normalisations of the
absolute value $|\cdot|$ and Haar measure $\mu$.
An {\it interval} $I$ is a ball $I = B_{r_I}(x_I)$ with some centre $x_I \in \KK$
and radius $r_I>0$.
For an interval $I$, we associate the measure 
$$
d\mu_I(x) \ = \ \frac{1}{\mu(I)} \mathbbm{1}_I(x) \, d\mu(x).
$$
For an interval $I$, we define the Fej\'er kernel $\kappa_I(x) = \mu(I)^{-2} \mathbbm{1}_I * \mathbbm{1}_{-I} (x)$ and
the corresponding measure $d\nu_I(x) = \kappa_I(x) d\mu(x)$. When $I = [N]$ for some scale $N$, we have
$-I = I$ and so $\kappa_{[N]}(x) = N^{-2} \mathbbm{1}_{[N]}*\mathbbm{1}_{[N]}(x)$. Furthermore when $\KK$ is non-archimedean,
we have $\kappa_{[N]}(x) = N^{-1} \mathbbm{1}_{[N]}(x)$ and so $d\nu_I = d\mu_I$ in this
case. When $\KK = {\mathbb R}$ and $I = [0,N]$, we have $\kappa_I(x) = N^{-1}(1 - |x|/N)$ when $|x|\le N$ and zero otherwise.

We now give precise notation which we will use throughout the paper.

\subsection{Basic notation}
As usual $\Z$ will denote the ring of rational integers.
\begin{enumerate}[label*={\arabic*}.]
\item We use $\Z_+:=\{1, 2,\ldots\}$ and $\N := \Z_+\cup\{0\}$ to
denote the sets of positive integers and non-negative integers,
respectively.  

\item For any $L\in\RR_+$ we will use the notation
\[
\bra{L}_0 := \{ \ell \in \N : \ell \le L \} \ \ \ {\rm and} \ \ \ \bra{L} := \{\ell \in \Z_+ : \ell \le L\}.
\]

\item We use $\ind{A}$ to denote the indicator function of a set $A$. If $S$ 
is a statement we write $\ind{S}$ to denote its indicator, equal to $1$
if $S$ is true and $0$ if $S$ is false. For instance $\ind{A}(x)=\ind{x\in A}$.

\end{enumerate}

\subsection{Asymptotic notation and magnitudes}
 The letters $C,c, C_0, C_1, \ldots>0$ will always denote
absolute constants, however their values may vary from occurrence to
occurrence.

\begin{enumerate}[label*={\arabic*}.]

\item For two nonnegative quantities
$A, B$ we write $A \lesssim_{\delta} B$ ($A \gtrsim_{\delta} B$) if
there is an absolute constant $C_{\delta}>0$ (which possibly depends
on $\delta>0$) such that $A\le C_{\delta}B$ ($A\ge C_{\delta}B$).  We
will write $A \simeq_{\delta} B$ when $A \lesssim_{\delta} B$ and
$A\gtrsim_{\delta} B$ hold simultaneously. We will omit the subscript
$\delta$ if irrelevant.

\item For a function $f:X\to \C$ and positive-valued
function $g:X\to (0, \infty)$, write $f = O(g)$ if there exists a
constant $C>0$ such that $|f(x)| \le C g(x)$ for all $x\in X$. We will
also write $f = O_{\delta}(g)$ if the implicit constant depends on
$\delta$.  For two functions $f, g:X\to \C$ such that $g(x)\neq0$ for
all $x\in X$ we write $f = o(g)$ if $\lim_{x\to\infty}f(x)/g(x)=0$.

\end{enumerate}

\subsection{Polynomials} Let $\KK[{\rm t}]$ denote the space of all
polynomials in one indeterminate $\rm t$ with
coefficients in $\KK$. Every polynomial $P\in \KK[{\rm t}]$ can be written as a formal power series
\begin{align}
\label{eq:29}
P(t)=\sum_{j=0}^{\infty}c_jt^j,
\end{align}
where all but finitely many coefficients $c_j\in \KK$ vanish.

\begin{enumerate}[label*={\arabic*}.]

\item We
define the degree of $P\in \KK[{\rm t}]$ by
\begin{align*}
\deg(P):=&\max\{j\in \Z_+: c_j\neq0\}.
\end{align*}

\item A finite  collection
$\mathcal P\subset \KK[{\rm t}]$ has degree $d\in\N$, if $d=\max\{\deg(P): P\in\mathcal P\}$.

\item For a polynomial  $P\in \KK[{\rm t}]$ and $j\in \N$ let
$\coe_j(P)$ denote $j$-th coefficient of $P$.  We also let $\ell(P)$
denote the leading coefficient of
$P$;  that is, for $P$ as in \eqref{eq:29} we have $\coe_j(P)=c_j$ for
$j\in \N$ and 
$\ell(P)=c_{d}$ where $d = \deg{P}$.

\end{enumerate}

\subsection{$L^p$ spaces}
$(X, \mathcal B(X), \lambda)$ denotes a measure space $X$ with
$\sigma$-algebra $\mathcal B(X)$ and $\sigma$-finite measure $\lambda$.
\begin{enumerate}[label*={\arabic*}.]

\item
  The set of  $\lambda$-measurable
complex-valued functions defined on $X$ will be denoted by $L^0(X)$.
\item The set of functions in $L^0(X)$ whose modulus is integrable
with $p$-th power is denoted by $L^p(X)$ for $p\in(0, \infty)$,
whereas $L^{\infty}(X)$ denotes the space of all essentially bounded
functions in $L^0(X)$.

\item We will say that a function $f\in L^0(X)$ is $1$-bounded if $f\in L^{\infty}(X)$ and $\|f\|_{L^{\infty}(X)}\le 1$.

\item For any $n\in\Z_+$ the measure $\lambda^{\otimes n}$ will denote the product measure $\lambda\otimes\ldots\otimes\lambda$ on the product space $X^n$ with the product $\sigma$-algebra $\mathcal B(X)\otimes\ldots \otimes\mathcal B(X)$.

\end{enumerate}

\subsection{Gowers box and  uniformity norms}
We will use the Gowers norm and Gowers box norm of a function $f$ which is defined in terms
of the multiplicative discrete derivatives $\Delta_{h_1.\ldots, h_s} f(x)$: for $x, h \in \KK$, we set
$\Delta_h f(x) = f(x){\overline{f(x+h)}}$ and iteratively, we define
$$
\Delta_{h_1, \ldots, h_s} f(x) \ = \ \Delta_{h_1}(\Delta_{h_2}( \cdots (\Delta_{h_s} f(x))\cdots))
\ \ \ {\rm where} \ \ x, h_1, \ldots, h_s \in \KK.
$$
When $h = (h_1, \ldots, h_s) \in \KK^s$, we often write $\Delta_{h_1, \ldots, h_s}f(x)$ as $\Delta_h f(x)$ or $\Delta_h^s f(x)$.
For $\omega = (\omega_1, \ldots, \omega_s) \in \{0,1\}^s$, we write $\omega \cdot h := \sum_{i=1}^s \omega_i h_i$
and $|\omega|:= \omega_1 + \cdots + \omega_s$. 
If ${\mathcal C} z = {\overline{z}}$ denotes the conjugation operator, we observe that
\begin{align}
\label{product}
\Delta_h f(x) \ = \ \prod_{\omega \in \{0,1\}^s} {\mathcal C}^{|\omega|} f(x + \omega \cdot h).
\end{align}
For any integer $s\ge 1$, we define the Gowers $U^s$ norm of $f$ by
$$
\|f\|_{U^s}^{2^s} \  = \ \int_{\KK^{s+1} }\Delta_{h_1, \ldots, h_s} f(x) \
d\mu(h_1) \cdots d\mu(h_s) d\mu(x).
$$
We note that $\|f\|_{U^2} = \|{\widehat{f}}\|_{L^4}$. 

For intervals $I, I_1, \ldots, I_s$, we define the Gowers box norm as
$$
\|f \|_{\square^s_{I_1, \ldots, I_s}(I)}^{2^s} \ = \ \frac{1}{\mu(I)} \int_{\KK^{s+1} }\Delta_{h_1, \ldots, h_s} f(x) \
d\nu_{I_1}(h_1) \cdots d\nu_{I_s}(h_s) d\mu(x) .
$$
From \eqref{product}, we see that
\begin{align}
\label{s+1-s}
\|f\|_{\square_{I_1, \ldots, I_{s+1}}^{s+1}(I)}^{2^{s+1}} \ = \
\int_{\KK}\|\Delta_h f\|_{\square_{I_1, \ldots, I_s}^s(I)}^{2^s}d\nu_{I_{s+1}}(h).
\end{align}
A similar formula relates the Gowers $U^{s+1}$ norm to the Gowers $U^s$ norm.

\subsection{The Gowers--Cauchy--Schwarz inequality}
When $s\ge 2$, both the Gowers uniformity norm and the Gowers box norm are in fact norms. In particular the triangle
inequality holds. The triangle inequality also holds when $s=1$ and so we have that
\begin{align}\label{triangle-ineq}
\|f + g\|_{U^s} \le \|f\|_{U^s} + \|g\|_{U^s} \ \ \ {\rm and} \ \ \ \|f + g\|_{\square_{I_1, \ldots, I_{s}}^{s}(I)} \le \|f\|_{\square_{I_1, \ldots, I_{s}}^{s}(I)} +
\|g\|_{\square_{I_1, \ldots, I_{s}}^{s}(I)}
\end{align}
holds for every $s\ge 1$. These inequalities follow from a more general inequality which we will find useful. 

Let $A$ be a finite set and for each $\alpha \in A$,
let $(X_{\alpha}, du_{\alpha})$ be a probability space. Set $X = \prod_{\alpha\in A} X_{\alpha}$ and let $f : X \to {\mathbb C}$
be a complex-valued function. For any $x^{(0)} = (x_{\alpha}^{(0)})_{\alpha \in A}$ and $x^{(1)} = (x_{\alpha}^{(1)})_{\alpha \in A}$ in  $X$ and
$\omega = (\omega_{\alpha})_{\alpha\in A} \in \{0,1\}^A$, we write $x^{(\omega)} = (x_{\alpha}^{(\omega_{\alpha})})_{\alpha\in A}$. We
define the {\it generalised Gowers box norm of $f$ on $X$} as
$$
\|f\|_{\square (X)}^{2^{|A|}} \ = \ \iint_{X^2}  \prod_{\omega \in \{0,1\}^A} {\mathcal C}^{|\omega|} f(x^{(\omega)}) \ du(x^{(0)}) \, du(x^{(1)})
$$
where $du$ denotes the product measure $\otimes_{\alpha \in A} du_{\alpha}$. The following lemma is established in 
\cite{GT-primes}.

\begin{lemma} [Gowers--Cauchy--Schwarz inequality]
\label{GCS}
With the set-up above, let $f_{\omega} : X \to {\mathbb C}$ for every $\omega \in \{0,1\}^A$. We have
\begin{align}
\label{GCS-ineq}
\Big| \iint_{X^2} \prod_{\omega \in \{0,1\}^A} {\mathcal C}^{|\omega|} f_{\omega} (x^{(\omega)}) \ du(x^{(0)}) \, du(x^{(1)}) \Bigr| \ \le \ 
\prod_{\omega \in \{0,1\}^A} \|f_{\omega}\|_{{\square (X)}}.
\end{align}
\end{lemma}
We will need the following consequence.

\begin{cor}\label{GCS-indep} Let $f : X \to {\mathbb C}$ and for each $\alpha \in A$, suppose
$g_{\alpha}: X \to {\mathbb C}$ is a 1-bounded function that is independent of the $x_{\alpha}$ variable. Then
\begin{align}\label{GCS-ineq-2}
\Bigl| \int_X f(x) \prod_{\alpha \in A} g_{\alpha}(x) du(x) \Bigr|^{2^{|A|}} \ \le \ \int_{X^2} \prod_{\omega\in \{0,1\}^A}
 {\mathcal C}^{|\omega|} f(x^{(\omega)}) \, du(x^{(0)}) du(x^{(1)}).
\end{align}
\end{cor}

\begin{proof} For $\omega^0 = (0, \ldots, 0)$, set $f_{\omega^0} = f$ and for $\omega^{\beta} = (\omega_{\alpha})_{\alpha\in A}$ with
$\omega_{\alpha} = 0$ when $\alpha \not= \beta$ and $\omega_{\beta} = 1$, set 
$f_{\omega^{\beta}} = {\overline{g_{\beta}}}$. For
all other choices of $\omega \in \{0,1\}^A$, set $f_{\omega} = 	1$. Hence 
$$
\prod_{\omega\in \{0,1\}^A}{\mathcal C}^{|\omega|}  f_{\omega}(x^{(\omega)}) = f(x^{(0)}) \prod_{\alpha \in A} g_{\alpha}(x^{(0)})
$$
since $g_{\alpha}$ is independent of the $\alpha$ variable. Therefore the inequality \eqref{GCS-ineq} implies
$$
\Bigl| \int_X f(x) \prod_{\alpha \in A} g_{\alpha}(x) du(x) \Bigr| \le \prod_{\omega\in \{0,1\}^A} 
\|f_{\omega}\|_{{\square (X)}} \le \|f\|_{{\square (X)}}
$$
by the 1-boundedness of each $g_{\alpha}$. This proves \eqref{GCS-ineq-2}.
\end{proof}

\section{Some preliminaries}\label{preliminaries}

In this section, we establish a few useful results which we will need in our arguments.

\subsection{$U^2$-inverse theorem} 

We will use the following inverse theorem for the Gowers box norms.

\begin{lemma}[$U^2$-inverse theorem]
\label{2.10}
Let $H_1$ and $H_2$ be two scales and
let $f$ be a 1-bounded function supported in an interval $I$. Then
\begin{equation}\label{U2-FT}
\|f\|_{\square^2_{[H_1],[H_2]}(I)}^4 \ \le \ (H_1 H_2)^{-1} \, \|{\widehat{f}}\|_{L^{\infty}(\KK)}^2.
\end{equation}
\end{lemma}

\begin{proof} We have
$$
\|f\|_{\square^2_{[H_1],[H_2]}(I)}^4 \ = \ \frac{1}{\mu(I)}
{\mathop{\iiint}_{\KK^3}} \Delta_{h_1, h_2} f(x) d\nu_{[H_1]}(h_1) d\nu_{[H_2]}(h_2)
d\mu(x)
$$
$$
= \ {\mathop{\iint}_{\KK^2}} g(h_1, h_2) \, d\nu_{[H_1]}(h_1) d\nu_{[H_2]}(h_2) \ = \ 
{\mathop{\iint}_{\KK^2}} {\widehat{g}}(\xi_1, \xi_2) \, 
{\overline{{\widehat{\nu_{[H_1]}}}(\xi_1) {\widehat{\nu_{[H_2]}}}(\xi_2)}} \, d\mu(\xi_1) d\mu(\xi_2)
$$
where
$$
g(h_1, h_2) \ = \ \frac{1}{\mu(I)} \int_{\KK} \Delta_{h_1, h_2} f(x) \, d\mu(x).
$$
Hence
$$
\|f\|_{\square^2_{[H_1],[H_2]}(I)}^4 \ \le \ \|{\widehat{\nu_{[H_1]}}}\|_{L^1} 
\|{\widehat{\nu_{[H_2]}}}\|_{L^1} \ \sup_{{\underline{\xi}}\in \KK^2} 
|{\widehat{g}}(\xi_1, \xi_2)|
$$
$$
= \ \frac{H_1^{-1}H_2^{-1}}{\mu(I)} \ \sup_{{\underline{\xi}}\in \KK^2} 
\Bigl| {\mathop{\iiint}_{\KK^3}} f_{00}(x) {\overline{f_{10}(x+h_1)}} 
{\overline{f_{01}(x+h_2)}} f_{11}(x+h_1+h_2) \, d\mu(x) d\mu(h_1) d\mu(h_2) \Bigr|
$$
where $f_{00}(x) =  f(x) {\rm e}(-\xi_1 x - \xi_2 x),$
$$
 \ f_{10}(x) \ = \ f(x) {\rm e}(-\xi_1 x), \
f_{01}(x) = f(x) {\rm e}(-\xi_2 x) \ \ {\rm and}
\ \ f_{11}(x) \ = \ f(x).
$$
The final equality follows since $|\widehat{\nu}_{[H_j]}(\xi)| = |\widehat{\mu}_{[H_j]}(\xi)|^2$ and so
\begin{align*}
\|\widehat{\nu}_{[H_j]}\|_{L^1(\KK)} = \| \widehat{\mu}_{[H_j]} \|_{L^2(\KK)}^2 = \| H_j^{-1}\ind{[H_j]} \|_2^2 = H_j^{-1}
\qquad {\rm for} \ j \in \{1,2\}
\end{align*}
by Plancherel's theorem.
Furthermore
$$
{\widehat{g}}(\xi_1, \xi_2) \ = \ \frac{1}{\mu(I)} {\mathop{\iiint}_{\KK^3}} \Delta_{h_1, h_2} f(x) \, 
{\rm e}(\xi_1 h_1 + \xi_2 h_2) \, d\mu(h_1) d\mu(h_2) d\mu(x).
$$
Appealing to the Gowers--Cauchy--Schwarz inequality \eqref{GCS-ineq},
we see that
$$
\|f\|_{\square^2_{[H_1],[H_2]}(I)}^4 \ \le \ (\mu(I) H_1 H_2)^{-1} \|f\|_{U^2}^4 \ = \ 
(\mu(I) H_1 H_2)^{-1}\|{\widehat{f}}\|_{L^4}^4 \ \le \
(H_1 H_2)^{-1} \|{\widehat{f}}\|_{L^{\infty}}^2
$$
as desired.
The last inequality follows from Plancherel's theorem, the 1-boundedness of $f$ and ${\rm supp}(f) \subset I$ which implies
$$
\|{\widehat{f}}\|_{L^4}^4 \le \|{\widehat{f}}\|_{L^{\infty}}^2 \|{\widehat{f}}\|_{L^2}^2 =
\|{\widehat{f}}\|_{L^{\infty}}^2 \|f\|_{L^2}^2 \le \mu(I) \|{\widehat{f}}\|_{L^{\infty}}^2 .
$$
\end{proof}

\subsection{van der Corput's inequality} 

We will need the following useful inequality.

\begin{lemma}[van der Corput's inequality]
\label{lem:vdC}
Let $\mathfrak g \in L^1(\KK)$ and let $J = B_{r_J}(x_J)$ be an interval. Then for any scale $H$, $0<H\le \mu(J)$, we have
\begin{equation}\label{vc-ineq-1}
\bigg|\int_{\KK} \mathfrak g(y)d\mu_J(y)\bigg|^2\leq \frac{C}{\mu(J)}\int_{\KK}\int_{J\cap(J-h)}\Delta_h\mathfrak g(y)
d\mu(y) d\nu_{[H]}(h).
\end{equation}
We can take $C = 4$ when $\KK$ is archimedean.
When $\KK$ is non-archimedean, we can take $C = 1$ and furthermore,
$\ind{J}(y) \ind{J}(y+h) = \mathbbm{1}_{J\cap(J-h)}(y) = \mathbbm{1}_J(y)$ for any $h \in [H]$ so that the above inequality can be expressed as
\begin{equation}\label{vc-ineq-2}
\bigg|\int_{\KK} \mathfrak g(y)d\mu_J(y)\bigg|^2\leq \iint_{\KK^2}\Delta_h\mathfrak g(y) d\mu_{[H]}(h) d\mu_J(y)
\end{equation}
since $d\nu_{[H]} = d\mu_{[H]}$ in this case.
\end{lemma}

\begin{proof}
We define $\mathfrak g_J(y):=\mathfrak g(y)\ind{J}(y)$. By a change of variables and Fubini's theorem we note 
\[
\int_{\KK}\mathfrak g(y) d\mu_J(y) = \frac{1}{\mu(J)}\iint_{\KK^2} \mathfrak g_J(y+h) d\mu_{[H]}(h) d\mu(y).
\]
The function $y\mapsto \int_{\KK}\mathfrak g_J(y+h) d\mu_{[H]}(h)$ is supported on the set $J - [H]$ which in turn
lies in $B_{2(r_J + H)}(x_J)$ (in the non-archimedean case, $J-[H] = J$). Hence 
by the Cauchy--Schwarz inequality and a change of variables, we conclude that
\begin{align*}
\bigg|\int_{\KK}\mathfrak g(y) d\mu_J(y)\bigg|^2 &=
\frac{1}{\mu(J)^2}\bigg|\iint_{\KK^2} \mathfrak g_J(y+h) d\mu_{[H]}(h) d\mu(y)\bigg|^2\\
&\le 2 \frac{\mu(J)+H}{\mu(J)^2}\iiint_{\KK^3} \mathfrak g_J(y+h_1)\overline{\mathfrak g_J(y+h_2)} d\mu_{[H]}(h_1)
d\mu_{[H]}(h_2) d\mu(y)\\
&= 2 \frac{\mu(J)+H}{\mu(J)^2}\iint_{\KK^2}\kappa_{[H]}(h)\mathfrak g_J(y)\overline{\mathfrak g_J(y+h)} d\mu(h)d\mu(y)\\
&\le 4 \mu(J)^{-1} \int_{\KK}\int_{J\cap(J-h)}\mathfrak g(y)\overline{\mathfrak g(y+h)} d\mu(y) d\nu_{[H]}(h),
\end{align*}
since $\kappa_{[H]}(h)=H^{-2}\int_{\KK}\ind{[H]}(h_1)\ind{[H]}(h+h_1)d\mu(h_1)$.
This gives the desired conclusion. 
\end{proof}

\subsection{Preparation for the PET induction scheme}
We now give a simple application of van der Corput's inequality
which will be repeatedly applied in the PET induction scheme.

\begin{lemma}
\label{lem:pet}
Let $c\ge1$ and let $I, J\subset \KK$ be two intervals with $\mu(I) = N_0$. Assume that 
$\mathfrak g_1\in L^\infty(\KK)$ and $\mathfrak g_2\in L^\infty(\KK^2)$ are $1$-bounded
functions such that
\begin{align}
\label{eq:65}
\|\mathfrak g_1\|_{L^1(\KK)}\le N_0,
\qquad \text{ and } \qquad
\sup_{y\in \KK}\|\mathfrak g_2(\cdot, y)\|_{L^1(\KK)}\le c N_0.
\end{align}
Suppose $H$ is a scale such that $0<H\le \mu(J)$. When $\KK$ is archimedean, we have
\begin{gather*}
\bigg|\frac{1}{N_0}\iint_{\KK^2}  \mathfrak g_1(x)\mathfrak g_2(x, y)d\mu_J(y)d\mu(x)\bigg|^2\\
\le 
4 \bigg|\frac{1}{N_0}\iiint_{\KK^3}\mathfrak g_2(x, y)\overline{\mathfrak g_2(x, y+h)}d\mu_J(y)d\nu_{[H]}(h)d\mu(x)\bigg|
+ 8 c \bigg[\frac{\mu([H])}{\mu(J)}\bigg]^{\theta}
\end{gather*}
where $\theta = 1$ when $\KK = {\mathbb R}$ and $\theta = 1/2$ when $\KK = {\mathbb C}$.  When
$\KK$ is non-archimedean, this improves to
\begin{gather*}
\bigg|\frac{1}{N_0}\iint_{\KK^2} \mathfrak g_1(x)\mathfrak g_2(x, y)d\mu_J(y)d\mu(x)\bigg|^2\\
\le 
\frac{1}{N_0} \iiint_{\KK^3}\mathfrak g_2(x, y)\overline{\mathfrak g_2(x, y+h)}d\mu_J(y)d\mu_{[H]}(h)d\mu(x) .
\end{gather*}
\end{lemma}
\begin{proof}
Applying the Cauchy--Schwarz inequality in the $x$ variable it follows that
$$
\bigg|\frac{1}{N_0} \iint_{\KK^2}\mathfrak g_1(x)\mathfrak g_2(x, y)d\mu_J(y)d\mu(x)\bigg|^2\leq \frac{1}{N_0}
\int_{\KK} \bigg|\int_{\KK} \mathfrak g_2(x, y)d\mu_J(y)\bigg|^2 d\mu(x),
$$
since by \eqref{eq:65} and the $1$-boundedness of $\mathfrak g_1$, we have $\|\mathfrak g_1 \|_{L^2(\KK)}^2 \le N_0$.
By van der Corput's inequality in Lemma \ref{lem:vdC}, we obtain
\begin{gather*}
\int_{\KK} \bigg|\int_{\KK} \mathfrak g_2(x, y)d\mu_J(y)\bigg|^2d\mu(x)\\
\leq
4 \int_{\KK}\int_{\KK}\kappa_{[H]}(h)\frac{1}{\mu(J)}\int_{J\cap(J-h)} \mathfrak g_2(x, y)\overline{\mathfrak g_2(x, y+h)}d\mu(y)d\mu(h)d\mu(x)
\end{gather*}
when $\KK$ is archimedean. In this case, we have $\mu(J\setminus{[J\cap(J-h)]}) \le 2 \mu([H])$ when $\KK = {\mathbb R}$
and $\mu(J\setminus{[J\cap(J-h)]}) \le 2 \sqrt{\mu([H]) \mu(J)}$ when $\KK = {\mathbb C}$. Hence
\begin{align*}
\frac{4}{N_0} \int_{\KK}\kappa_{[H]}(h)\frac{1}{\mu(J)}\int_{J\setminus (J\cap(J-h))}\int_{\KK}|\mathfrak g_2(x, y)| d\mu(x)d\mu(y)d\mu(h)\le  8 c \bigg[\frac{\mu([H])}{\mu(J)}\bigg]^{\theta}.
\end{align*}
In the last line we used Fubini's theorem and \eqref{eq:65} for $\mathfrak g_2$.
This gives the desired bound when $\KK$ is archimedean.

When $\KK$ is non-archimedean, the bound \eqref{vc-ineq-2} in Lemma \ref{lem:vdC} gives
\begin{gather*}
\frac{1}{N_0} \int_{\KK} \bigg|\int_{\KK} \mathfrak g_2(x, y)d\mu_J(y)\bigg|^2d\mu(x)\\
\leq
\frac{1}{N_0} \iiint_{\KK^3}\mathfrak g_2(x, y)\overline{\mathfrak g_2(x, y+h)}d\mu_J(y)d\mu_{[H]}(h)d\mu(x)
\end{gather*}
which is the desired bound in this case.
\end{proof}

The next result is an essential building block of
the PET induction scheme, which will be employed in  Section \ref{sec:inverse}.

\begin{proposition}
\label{prop:pet}
Let $N, N_0 >0$ be two scales, $I$ an interval such that $\mu(I) = N_0$, $m\in\NN$, $i_0\in \bra{m}$ and let $\mathcal P:=\{P_1,\ldots, P_m\}$ be a collection of polynomials. Suppose that
$\mathfrak f_0, \mathfrak f_1,\ldots, \mathfrak f_m\in L^0(\KK)$ are
$1$-bounded functions such that $\|\mathfrak f_i\|_{L^1(\KK)}\le N_0$ for every $i\in\bra{m}_0$.

Let $0<\delta\le 1$ and  suppose that
\begin{align}
\bigg|\frac{1}{N_0} \iint_{\KK^2}\mathfrak f_0(x)\prod_{i=1}^m\mathfrak f_{i}(x-P_i(y))d\mu_{[N]}(y)d\mu(x)\bigg|\ge\delta.
\end{align}
Then there exists an absolute constant $C\gtrsim_{\mathcal P}1$ such that for all  $\delta'\le\delta^4/C$ we have
\begin{align}
\bigg|\frac{1}{N_0} \iint_{\KK^2}\mathfrak f'_0(x)\prod_{i=1}^{m'}\mathfrak f'_{i}(x-P'_i(y))d\mu_{[N]}(y)d\mu(x)\bigg|\gtrsim_C\delta^2,
\end{align}
where $m'< 2m$ and $\mathcal P':=\{P_1',\ldots, P_{m'}'\}$ is a new collection of polynomials such that
\[
\mathcal P'=\{P_1(y)-P_{i_0}(y), P_1(y+h)-P_{i_0}(y),\ldots,  P_m(y)-P_{i_0}(y), P_m(y+h)-P_{i_0}(y)\},
\]
for some $\delta'\delta^2N/C^2\le |h|\le \delta'N\le  \delta^4N/C$,
where $P_{m'}'(y):=P_m(y)-P_{i_0}(y)$, and $\{\mathfrak f_0',\ldots, \mathfrak f_{m'}'\}:=\{\mathfrak f_1, \overline{\mathfrak f_1},\ldots, \mathfrak f_m, \overline{\mathfrak f_m}\}$ with $\mathfrak f_{m'}':=\mathfrak f_{m}$.
\end{proposition}

\begin{proof}
Let $\I:=\bra{m}$ and $C\ge 1$ be a large constant to be
determined later. We shall apply Lemma \ref{lem:pet} with 
$J=[N]$, the functions $\mathfrak g_1(x)=\mathfrak f_0(x)$ and
$\mathfrak g_2(x, y)=\prod_{i\in\I}\mathfrak f_{i}(x-P_i(y))$, and the
parameter $H=\delta'N$. Note that
$\|\mathfrak g_1\|_{L^{\infty}(\KK)}\le 1$ and
$\|\mathfrak g_2\|_{L^{\infty}(\KK^2)}\le 1$, since
$\|\mathfrak f_i\|_{L^{\infty}(\KK)}\le 1$ for all $i\in\I$. Moreover,
$\mathfrak g_1$ and $\mathfrak g_2$ satisfy \eqref{eq:65}. If
$\delta'\le \delta^4/C$ and $C\ge1$ is sufficiently large, using Lemma
\ref{lem:pet}, we conclude 
\begin{align*}
\bigg|\frac{1}{N_0}\iiint_{\KK^3}\mathfrak g_2(x, y)\overline{\mathfrak g_2(x, y+h)}d\mu_{[N]}(y)d\nu_{[H]}(h) d\mu(x)\bigg|
\gtrsim \delta^2.
\end{align*}
By the pigeonhole principle, there exists $|h|\ge \delta^2 H/C^2$ so that 
\begin{align*}
\bigg|\frac{1}{N_0} \iint_{\KK^2}\mathfrak g_2(x, y)\overline{\mathfrak g_2(x, y+h)}d\mu_{[N]}(y)d\mu(x)\bigg|
\gtrsim \delta^2.
\end{align*}
We make the change of variables $x\mapsto x+P_{i_0}(y)$ to conclude
\begin{align*}
\bigg|\frac{1}{N_0}\iint_{\KK^2}
\prod_{i\in\I}\mathfrak f_{i}(x-P_i(y)+P_{i_0}(y))\overline{\mathfrak f_{i}(x-P_i(y+h)+P_{i_0}(y))}
d\mu_{[N]}(y)d\mu(x)\bigg|\gtrsim\delta^2.
\end{align*} 
This completes the proof.
\end{proof}

\section{The $L^\infty$-inverse theorem}\label{sec:inverse}

The goal of this section is to present the proof of Theorem \ref{thm:inverse-informal}, the key $L^\infty$-inverse
theorem for general polynomials with distinct degrees, which we now restate in a more formal, precise way.
\begin{theorem}[Inverse theorem for $(m+1)$-linear forms]
\label{thm:inverse}
Let $N \ge 1$ be a scale,  $m\in\Z_+$ and $0<\delta\le 1$ be given. Let $\mathcal P:=\{P_1,\ldots, P_m\}$
be a collection of polynomials such that
$1\le \deg{P_1}<\ldots<\deg{P_m}$. Set $N_0 = N^{\deg(P_m)}$ and let
$f_0, f_1,\ldots, f_m\in L^0(\KK)$ be $1$-bounded functions supported
on an interval $I\subset \KK$ of measure $N_0$.  Define an $(m+1)$-linear form
corresponding to the pair $(\mathcal P; N)$ by
\begin{align}
\label{eq:6}
\Lambda_{\mathcal P; N}(f_0,\ldots, f_m):=
\frac{1}{N_0} \int_{\KK^2}f_0(x)\prod_{i=1}^mf_{i}(x-P_i(y))d\mu_{[N]}(y)d\mu(x).
\end{align}
Suppose that
\begin{align}
\label{eq:2}
|\Lambda_{\mathcal P; N}(f_0,\ldots, f_m)|\ge\delta.
\end{align}
Then there exists $N_1\simeq \delta^{O_{\mathcal P}(1)}N^{\deg(P_1)}$ so that
\begin{align}
\label{eq:19}
N_0^{-1}\big\| \mu_{[N_1]}*f_1\big\|_{L^1(\KK)} \gtrsim_{\mathcal P} \delta^{O_{\mathcal P}(1)}.
\end{align}
\end{theorem}

If necessary  we will also write
$\Lambda_{\mathcal P; N}(f_0,\ldots, f_m)=\Lambda_{\mathcal P; N, I}(f_0,\ldots, f_m)$
in order to emphasize that the functions $f_0, f_1,\ldots, f_m$ are
supported on $I$.

\paragraph{\bf Remark} When $\KK = {\mathbb C}$ is the complex field, the proof of Theorem \ref{thm:inverse} will
also hold if the form $\Lambda_{\mathcal P; N}$ is defined with the disc $[N] = {\mathbb D}_{\sqrt{N}}$ replaced
by the square 
$$
[N]_{sq} \ := \ \{x + i y \in {\mathbb C} : |x|\le \sqrt{N}, |y| \le \sqrt{N}\}.
$$ 
In this case, the conlusion is $N_0^{-1} \| \mu_{[N_1]_{sq}} * f_1\|_{L^1({\mathbb C})} \gtrsim \delta^{O_{\mathcal P}(1)}$.
This observation will be needed at one point in the proof of Theorem \ref{sobolev-informal}.

The proof of Theorem \ref{thm:inverse} breaks into two main steps:
first, an application of PET induction to show that whenever
\[ |\Lambda_{\mathcal{P};N}(f_0,f_1,\dots,f_m) | \geq \delta \] is
large, then necessarily $f_m$ has a fairly large $U^s$ norm for an
appropriately large $s = s_{\mathcal{P}}$. Second, an inductive
``degree-lowering'' step to reduce $U^s$ control to $U^2$ control. We
accordingly subdivide the argument into two subsections.

\subsection{PET induction}
Our first goal is to show that whenever the multi-linear form
$\Lambda_{\mathcal P;I}$ is large, necessarily $f_m$ has some fairly
large (sufficiently high degree) Gowers box norm. We begin with 
the definition of $(d,j)$-admissible polynomials. Recall that for a polynomial $P \in \KK[{\rm y}]$,
the leading coefficient is denoted by $\ell(P)$.
\begin{definition}[The class of $(d,j)$-admissible polynomials]
\label{def:1}
Let $N\ge 1$ be a scale, $0<\delta\le 1$,  $d\in\Z_+$, $j\in\bra{d}$ and parameters $A_0\ge1$ and $A\ge0$ be
given. Assume that a finite collection of polynomials $\mathcal P$ has degree $j$ and 
define  $\mathcal{P}_j := \{ P\in\mathcal P : \deg(P) = j\}$.
We will say that $\mathcal P$ is $(d, j)$-admissible with tolerance $(A_0, A)$
if the following properties are satisfied:
\begin{enumerate}[label*={\arabic*}.]
\item For every $P\in \mathcal P_j$ we have
\begin{align}
\label{eq:24}
A_0^{-1}\delta^{A} N^{d-j} \leq |\ell(P)| \leq  A_0\delta^{-A}N^{d-j}.
\end{align}
\item Whenever $P, Q\in\mathcal P_j$ and  $\ell(P) \neq \ell(Q)$ we have
\begin{align}
\label{eq:25}
A_0^{-1}\delta^{A} N^{d-j} \leq |\ell(P) - \ell(Q)| \leq A_0\delta^{-A} N^{d-j}.
\end{align}
\item Whenever $P, Q\in\mathcal P_j$ and $P\neq Q$ and $\ell(P) = \ell(Q)$ we have  
\begin{align}
\label{eq:26}
A_0^{-1}\delta^{A} N^{d-j+1} \leq |\ell(P - Q)| \leq A_0\delta^{-A} N^{d-j+1},
\end{align}
and $\deg(P-Q) = j-1$.
\end{enumerate}
In the special case where the polynomials in $\mathcal{P}$ are linear, we require that
$\ell(P) \neq \ell(Q)$ for each $P, Q\in\mathcal P$. The constants
$A_0, A$ will be always independent of $\delta$ and $N$, but may
depend on $\mathcal P$. In our applications the exact values of
$A_0, A$ will be unimportant and then we will simply say that the
collection $\mathcal{P}$ is $(d, j)$-admissible.

\end{definition}

\begin{remark}
\label{rem:3}
Under the hypotheses of Theorem \ref{thm:inverse} it is not difficult
to see that the collection of polynomials
$\mathcal P=\{P_1,\ldots, P_m\}$  such that
$1\le \deg{P_1}<\ldots<\deg{P_m}=d$ is $(d,d)$-admissible with the
tolerance $(\max\{|\ell(P_m)|^{-1}, |\ell(P_m)|\}, 0)$. Indeed, condition
\eqref{eq:24} can be easily verified and conditions \eqref{eq:25} and
\eqref{eq:26} are vacuous as $\mathcal P_d=\{P_m\}$.
\end{remark}

The main result of this subsection is the following theorem.

\begin{theorem}[Gowers box norms  control  $(m+1)$-linear forms]
\label{thm:Us}
Let $\mathcal P:=\{P_1,\ldots, P_m\}$ be a collection of
$(d, d)$-admissible polynomials such that
$1\le \deg{P_1}\le\ldots\le\deg{P_m}=d$.  
Let $N, N_0\ge 1$ be two scales, $I$ an interval with measure $N_0$ and $0<\delta\le 1$  be given and let
$f_0, f_1,\ldots, f_m\in L^0(\KK)$ be $1$-bounded functions
such that $\|f_i\|_{L^1(\KK)}\le N_0$ for all $i\in\bra{m}_0$.
If \eqref{eq:2} is
satisfied, then there exists $s:=s_{\mathcal P}\in\Z_+$ such that
\begin{align}
\label{eq:32}
\|f_m\|_{\square_{[H_1], \ldots, [H_s]}^s(I)} \gtrsim_{\mathcal P} \delta^{O_{\mathcal P}(1)},
\end{align}
where
$H_i\simeq \delta^{O_{\mathcal P}(1)}N^{\deg(P_m)}$ for $i\in\bra{s}$.
\end{theorem}

The proof of Theorem \ref{thm:Us} requires a subtle downwards
induction based on a repetitive application of Proposition
\ref{prop:pet} on the class of $(d, j)$-admissible polynomials.  To
make our induction rigorous, we will assign a weight vector to each collection
$\mathcal{P}\subset \KK[{\rm t}]$ of polynomials.

\begin{definition}[Weight vector]
For any finite $\mathcal{P}\subset \KK[{\rm t}]$ define the weight vector
\[
v(\mathcal{P}) := (v_1, v_2,\dots) \in \mathbb{N}^{\Z_+},
\]
where 
\[
v_j :=v_j(\mathcal P):= \#\{ \ell(P) : P \in \mathcal{P} \text{ and } \deg(P) = j\},
\]
is the number of distinct leading coefficients of $\mathcal{P}$ of degree $j\in\Z_+$.
\end{definition}

For example, the weight vector for the family
$\mathcal P=\{x, 5x, x^2, x^2+x, x^4\}$ is
$v(\mathcal P)=(2, 1, 0, 1, 0, 0, \ldots)$. There is a natural
ordering on the set of weight vectors.

\begin{definition}[Well-ordering on the set of weight vectors]
For any two weight vectors $v(\mathcal P)=(v_1(\mathcal P),v_2(\mathcal P),\dots)$ and $ v(\mathcal Q)=(v_j(\mathcal Q),v_j(\mathcal Q),\dots)$ corresponding to finite collections $\mathcal{P}, \mathcal Q\subset \KK[{\rm t}]$ we define an ordering $\prec$ on the set of weight vectors by declaring that
\[
v(\mathcal P)\prec v(\mathcal Q)
\]
if there is a degree $j\in\Z_+$ such that $v_j(\mathcal P)<v_j(\mathcal Q)$ and $v_k(\mathcal P)=v_k(\mathcal Q)$ for all $k>j$.
\end{definition}

It is a standard fact that $\prec$ is a well ordering, we omit the details.

\begin{proof}[Proof of Theorem \ref{thm:Us}]
We begin by stating the following claim:
\begin{claim}
\label{claim:1}
Let $N, N_0\ge 1$ be two scales, $0<\delta\le 1$, $d, m\in\Z_+$  and
$j\in\bra{d}$ be given and let $\mathcal P:=\{P_1,\ldots, P_m\}$
be a collection of $(d, j)$-admissible polynomials with tolerance $(A_0, A)$
such that $\deg{P_1}\le\ldots\le\deg{P_m} =j$. Let $I$ be an interval with $\mu(I) = N_0$ and let
$f_0, f_1,\ldots, f_m \in L^0(\KK)$ be $1$-bounded functions
such that $\|f_i\|_{L^1(\KK)}\le N_0$ for all $i\in\bra{m}_0$.
Suppose that
\begin{align}
\label{eq:49}
|\Lambda_{\mathcal P; N}(f_0,\ldots, f_m)|\ge\delta.
\end{align}
Then there exists a collection $\mathcal P':=\{P_1',\ldots, P_{m'}'\}$ of $(d, j-1)$-admissible polynomials with tolerance $(A_0', A')$
and $m':=\#\mathcal P'$ so
that $\deg(P_1')\le\ldots\le \deg(P_{m'}')=j-1$, and $1$-bounded functions
$f_0', f_1',\ldots, f_{m'}'\in L^0(\KK)$ such that $\|f_i'\|_{L^1(\KK)}\le N_0$ for all $i\in\bra{m'}_0$ with $f_{m'}':=f_m$ and satisfying
\begin{align}
\label{eq:50}
|\Lambda_{\mathcal P'; N}(f_0',\ldots, f_{m'}')|\gtrsim_{\mathcal P}\delta^{O_{\mathcal P}(1)}.
\end{align}
\end{claim}

The proof of Claim \ref{claim:1} will use the polynomial exhaustion
technique based on an iterative application of the PET induction
scheme from Proposition \ref{prop:pet}.  The key steps of this method
are gathered in Proposition \ref{prop:PETiterate}.  Assuming
momentarily that Claim \ref{claim:1} is true we can easily close the
argument to prove Theorem \ref{thm:Us}.  We begin with a collection of $(d, d)$-admissible
polynomials such that $\deg{P_1}\le\ldots\le\deg{P_m}=d$ and apply
our claim $d-1$ times until we reach a collection of
$(d,1)$-admissible linear polynomials $\mathcal L$ with distinct
leading terms, which satisfies \eqref{eq:50} with
$\mathcal P'=\mathcal L$. In the special case where all polynomials
are linear matters simplify and can be handled using the next result, Proposition
\ref{prop:linear}, which in turn  implies \eqref{eq:32} from Theorem
\ref{thm:Us} as desired.
\end{proof}

\begin{proposition}
\label{prop:linear}
Let $N, N_0 \ge 1$ be two scales, $I$ an interval with $\mu(I) = N_0$,  $0<\delta\le 1$, $d, m\in\Z_+$ be given and let
$\mathcal L:=\{L_1,\ldots, L_m\}$ be a collection of
$(d,1)$-admissible linear polynomials. Let
$f_0, f_1,\ldots, f_m\in L^0(\KK)$ be $1$-bounded functions
such that $\|f_i\|_{L^1(\KK)}\le N_0$ for all $i\in\bra{m}_0$.
Suppose that
\begin{align}
\label{eq:18}
|\Lambda_{\mathcal L; N}(f_0,\ldots, f_m)|\ge\delta.
\end{align}
Then we have
\begin{align}
\label{eq:20}
\|f_m\|_{\square_{[H_1], \ldots, [H_m]}^m(I)} \gtrsim_{\mathcal L} \delta^{2^{m-1}},
\end{align}
where
$H_i\simeq \delta^{O_{\mathcal L}(1)}N^{d}$ for $i\in\bra{s}$.
\end{proposition}

In fact Proposition \ref{prop:linear} is a special case of Theorem
\ref{thm:Us} with the collection of linear polynomials $\mathcal L$ in
place of $\mathcal P$.

\begin{proof}[Proof of Proposition \ref{prop:linear}]
Defining $\mathcal{L}'=\{L_i':=L_i-L_i(0): i\in\bra{m}\}$ we see that each $L'\in \mathcal L'$ is linear 
with vanishing constant term and
\begin{align*}
\Lambda_{\mathcal L; N}(f_0,\ldots, f_m)=\Lambda_{\mathcal L'; N}(g_0,\ldots, g_m),
\end{align*}
where $g_i(x)=\Tra{-L_i(0)}f_i(x)=f_i(x+L_i(0))$ for each $i\in\bra{m}$.
We now apply Lemma \ref{lem:pet} with functions
$\mathfrak g_1(x)=g_0(x)$ and
$\mathfrak g_2(x, y)=\prod_{i=1}^mg_i(x-L_i'(y))$ and intervals
$J=[N]$, and a parameter $H=\delta^{M}N/M$ for some
large absolute constant $M\ge1$, which will be specified later.
Using Lemma \ref{lem:pet} and changing the variables $x\mapsto x-L_1(y)$ we obtain
\begin{align*}
\bigg|\frac{1}{N_0} \iiint_{\KK^3}\Delta_{\ell(L_1)h} g_1(x)\prod_{i=2}^m\Delta_{\ell(L_i)h}g_{i}(x-(L_i-L_1)(y))d\mu_{[N]}(y)d\mu(x)d\nu_{[H]}(h)\bigg|\gtrsim_M \delta^2.
\end{align*}
Applying Lemma \ref{lem:pet} $m-2$ more times  and changing the variables $x\mapsto x-L_m(0)$ we obtain 
\begin{align*}
\bigg|\frac{1}{N_0} \int_{\KK^{m+1}}\Delta_{u_1h_1}\cdots\Delta_{u_{m-1}h_{m-1}}\Delta_{\ell(L_m)h_m}f_{m}(x)
d\nu_{[H]}^{\otimes m}(h_1,\ldots, h_m)d\mu(x)\bigg|\gtrsim_M \delta^{2^{m-1}},
\end{align*}
where $u_i:=\ell(L_m)-\ell(L_{i})$ for $i\in\bra{m-1}$. By another change of variables we obtain \eqref{eq:20} with
\begin{align*}
H_m= |\ell(L_m)|\delta^M N/M,
\quad \text{ and } \quad
H_i=|\ell(L_m)-\ell(L_{i})|\delta^M N/M
\end{align*}
for $i\in\bra{m-1}$. Using \eqref{eq:24} with $P=L_m$, and
\eqref{eq:25} with $P=L_m$ and $Q=L_i$ we obtain that
$H_i\simeq \delta^{O_{\mathcal L}(1)}N^{d}$ for $i\in\bra{s}$
provided that $M\ge1$ is sufficiently large. This completes the proof
of Proposition \ref{prop:linear}.
\end{proof}

\begin{proposition}
\label{prop:PETiterate}
Let $N, N_0>0$ be two scales, $0<\delta\le 1$, $d, m\in\Z_+$ and $i, j\in\bra{d}$ be given and
let $\mathcal P:=\{P_1,\ldots, P_m\}$ be a collection of
$(d, j)$-admissible polynomials with tolerance $(A_0, A)$ such that
$i =\deg{P_1}\le\ldots\le\deg{P_m} =j$. Let $I$ be an interval with $\mu(I) = N_0$ and let
$f_0, f_1,\ldots, f_m\in L^0(\KK)$ be $1$-bounded functions
such that $\|f_i\|_{L^1(\KK)}\le N_0$ for all $i\in\bra{m}_0$.
Suppose that
\begin{align}
\label{eq:21}
|\Lambda_{\mathcal P; N}(f_0,\ldots, f_m)|\ge\delta.
\end{align}
Then there exists a
collection of polynomials $\mathcal P':=\{P_1',\ldots, P_{m'}'\}$ with
$m':=\#\mathcal P'<2\#\mathcal P$ satisfying $P_{m'}':=P_m-P_1$ and
$\deg(P_1')\le\ldots\le \deg(P_{m'}')$, and $1$-bounded functions
$f_0', f_1',\ldots, f_{m'}'\in L^0(\KK)$ such that
$\|f_i'\|_{L^1(\KK)}\le N_0$ for all $i\in\bra{m'}_0$ and satisfying
\begin{align}
\label{eq:34}
|\Lambda_{\mathcal P'; N}(f_0',\ldots, f_{m'}')|\gtrsim_{\mathcal P}\delta^2.
\end{align}
We also know that
$\{f_0', f_1',\ldots, f_{m'}'\}=\{f_1,\overline{f_1},\ldots, f_m,\overline{f_m}\}$
with $f_{m'}' =f_m$.

Moreover, $v(\mathcal P')\prec v(\mathcal P)$, and one of the following three scenarios occurs.
\begin{enumerate}[label*=({\roman*})]
\item\label{item:1}  The collection $\mathcal P$ is of type I; that is, 
$\mathcal P\neq\mathcal P_j$. In this case, $\mathcal P'$ is a
$(d, j)$-admissible collection of  polynomials with tolerance $(A_0', A')$ and for some $1\le i \le j-1$,
\begin{align}
\label{eq:23}
\qquad v(\mathcal P')=(v_1(\mathcal P'),\ldots, v_{i-1}(\mathcal P'), v_i(\mathcal P)-1, v_{i+1}(\mathcal P), \ldots, v_{j}(\mathcal P),0, 0, \ldots).
\end{align}

\item\label{item:2}  The collection $\mathcal P$ is of type II; that is,
$\mathcal P=\mathcal P_j$ and $v_j(\mathcal P)>1$. In this case,  $\mathcal P'$
is a $(d, j)$-admissible collection of polynomials with tolerance $(A_0', A')$  and
\begin{align}
\label{eq:27}
v(\mathcal P')=(v_1(\mathcal P'),\ldots, v_{j-1}(\mathcal P'), v_j(\mathcal P)-1, 0,0, \ldots).
\end{align}

\item\label{item:3}  The collection $\mathcal P$ is of type III; that is,
$\mathcal P=\mathcal P_j$ and  $v_j(\mathcal P)=1$. In this case, $\mathcal P'$
is a $(d, j-1)$-admissible collection of polynomials with tolerance $(A_0', A')$  and
\begin{align}
\label{eq:33}
v(\mathcal P')=(0,\ldots,0, v_{j-1}(\mathcal P'), 0,0, \ldots).
\end{align}
Moreover, the leading coefficients of the polynomials in $\mathcal P'$ are pairwise distinct. 
\end{enumerate}
The tolerance $(A_0', A')$ of the collection $\mathcal P'$ only
depends on the tolerance $(A_0, A)$ of the collection $\mathcal P$, and
is independent of $\delta$ and $N$.
\end{proposition}

Using Proposition \ref{prop:PETiterate} we now prove Claim \ref{claim:1}.
\begin{proof}[Proof of Claim \ref{claim:1}]
We may assume, without loss of generality, that the collection
$\mathcal P$ from Claim \ref{claim:1} is of type I or type II. Then we apply
Proposition \ref{prop:PETiterate} until we reach a collection of
polynomials of type III with weight vector
$v(\mathcal P)=(0,\ldots,0, v_{j}(\mathcal P), 0,0, \ldots)$ where $v_j(\mathcal P) = 1$
and such that \eqref{eq:50} holds. 
We apply Proposition \ref{prop:PETiterate} once more to reach a 
collection of $(d, j-1)$-admissible polynomials satisfying \eqref{eq:50}.
This completes the proof of the claim. 
\end{proof}

\begin{proof}[Proof of Proposition \ref{prop:PETiterate}]
Appealing to Proposition \ref{prop:pet} with $i_0=1$ we may conclude that
there exists a  collection of polynomials $\mathcal P':=\{P_1',\ldots, P_{m'}'\}$ with 
$m'=\#\mathcal P'<2\#\mathcal P$  and $P_{m'}'=P_m-P_1$ such that
\[
\mathcal P'=\{P_1(y)-P_{1}(y), P_1(y+h)-P_{1}(y),\ldots,  P_m(y)-P_{1}(y), P_m(y+h)-P_{1}(y)\},
\]
for some $\delta'\delta^2N/C^2\le |h| \le \delta'N\le \delta^4N/C$.
Proposition \ref{prop:pet} also ensures that bound \eqref{eq:34} holds
for certain $1$-bounded functions
$f_0', f_1',\ldots, f_{m'}'\in L^0(\KK)$ such that
$\|f_i'\|_{L^1(\KK)}\le N_0$ for all $i\in\bra{m'}_0$ and satisfying
$\{f_0', f_1',\ldots, f_{m'}'\}=\{f_1,\overline{f_1},\ldots, f_m,\overline{f_m}\}$
with $f_{m'}'=f_m$.
Now it remains to verify conclusions from
\ref{item:1}, \ref{item:2} and \ref{item:3}. For this purpose we will
have to adjust $\delta'\le \delta^4/C$, which can be made as small as necessary.

\medskip

\paragraph{\textit{Proof of the conclusion from \ref{item:1}}} Suppose
that the collection $\mathcal P$ is of type I.  Then
$i=\deg(P_1)<\deg(P_m)=j$ and
$v(\mathcal P)=(0,\ldots, 0, v_i(\mathcal P),\ldots, v_j(\mathcal P), 0, 0, \ldots)$. To
establish \eqref{eq:23} we consider three cases. Let $P\in\mathcal P$. If $\deg(P)>i$,
then 
\begin{align}
\label{eq:35}
\begin{gathered}
\deg(P-P_1)=\deg(P(\cdot+h)-P_1)=\deg(P), \\
\ell(P-P_1)=\ell(P(\cdot+h)-P_1)=\ell(P),
\end{gathered}
\end{align}
which yields that $v_k(\mathcal P')=v_k(\mathcal P)$ for all $k>i$.
If $\deg(P)=i$ and $\ell(P)\neq\ell(P_1)$, then 
\begin{align}
\label{eq:36}
\begin{gathered}
\deg(P-P_1)=\deg(P(\cdot+h)-P_1)=i, \\
\ell(P-P_1)=\ell(P(\cdot+h)-P_1)=\ell(P)- \ell(P_1).
\end{gathered}
\end{align}
If $\deg(P)=i$ and $\ell(P)=\ell(P_1)$, then 
\begin{align*}
\deg(P-P_1)<i,
\quad\text{ and } \quad
\deg(P(\cdot+h)-P_1)<i.
\end{align*}
The latter two cases show that $v_{k}(\mathcal P')\ge0$ for all
$k\in \bra{i-1}$ and $v_i(\mathcal P')=v_i(\mathcal P)-1$. Hence
\eqref{eq:23} holds. We now show that $\mathcal P'$ is
$(d, j)$-admissible.

We begin with verifying \eqref{eq:24} for $P'\in\mathcal P_j'$.  We
may write $P'=P(\cdot+\varepsilon h)-P_1$ for some $P\in\mathcal P_j$
and $\varepsilon\in\{0, 1\}$. By \eqref{eq:35} and \eqref{eq:24} for
$P\in\mathcal P_j$ we obtain
\begin{align}
\label{eq:41}
A_0^{-1}\delta^{A} N^{d-j} \leq |\ell(P')| \leq  A_0\delta^{-A}N^{d-j}.
\end{align}

We now verify \eqref{eq:25} for $Q_1', Q_2'\in\mathcal P'_j$ with $\ell(Q_1')\neq\ell(Q_2')$. We may
write
\begin{align}
\label{eq:40}
Q_1'=Q_1(\cdot+\varepsilon_1 h)-P_1,
\qquad \text{ and } \qquad
Q_2'=Q_2(\cdot+\varepsilon_2 h)-P_1
\end{align}
for some
$Q_1, Q_2\in\mathcal P_j$ and
$\varepsilon_1, \varepsilon_2\in\{0, 1\}$.  By \eqref{eq:35} we have
$\ell(Q_1')=\ell(Q_1)$ and $\ell(Q_2')=\ell(Q_2)$. 
Then $\ell(Q_1)\neq\ell(Q_2)$ and
by \eqref{eq:25} for $Q_1, Q_2\in\mathcal P_j$ we deduce
\begin{align}
\label{eq:42}
A_0^{-1}\delta^{A} N^{d-j} \leq |\ell(Q_1') - \ell(Q_2')| \leq A_0\delta^{-A} N^{d-j}.
\end{align}

We finally verify \eqref{eq:26} for $Q_1', Q_2'\in\mathcal P'_j$ as in \eqref{eq:40} such
that $Q_1'\neq Q_2'$ and $\ell(Q_1')=\ell(Q_2')=\ell$.  By
\eqref{eq:35} we see that $\ell(Q_1)=\ell(Q_2)=\ell$.
Since $\mathcal P$ is $(d, j)$-admissible,
using \eqref{eq:24}, we also have
\begin{align}
\label{eq:37}
A_0^{-1}\delta^{A} N^{d-j} \leq |\ell| \leq  A_0\delta^{-A}N^{d-j}.
\end{align}
Recall that
$\delta'\delta^2N/C^2\le |h| \le \delta'N$, where $\delta'>0$ is an
arbitrarily small number such that $\delta'\le \delta^4/C$. Set
$\delta':=\delta^{M}(CM)^{-1}$ for a large number $M\ge1$, which will be
chosen later.

First suppose $Q_1 = Q_2$. Then $\varepsilon_1 \not= \varepsilon_2$ and 
$\deg(Q_1'-Q_2')=j-1$. Furthermore $\ell(Q_1'-Q_2')= j\ell h(\varepsilon_1-\varepsilon_2)$
implying $|\ell(Q_1' - Q_2')| = |j \ell h|$ and so by \eqref{eq:37},
\begin{align}
\label{eq:38}
|j|(A_0C^3M)^{-1}\delta^{A+M+2} N^{d-j+1}\le |j\ell h|
\le |j| A_0(CM)^{-1}\delta^{M-A} N^{d-j+1}
\end{align}
and this verifies \eqref{eq:26} in the case $Q_1 = Q_2$.

Now suppose $Q_1 \not= Q_2$ so that $\deg(Q_1 - Q_2) = j-1$ and \eqref{eq:26} holds for $\ell(Q_1 - Q_2)$; that is,
\begin{align}
\label{eq:38a}
A_0^{-1}\delta^{A} N^{d-j+1} \leq |\ell(Q_1-Q_2)| \leq A_0\delta^{-A} N^{d-j+1}.
\end{align}
Taking $M:=\max\{2A, 2|j|A_0^2\}$ in \eqref{eq:38}, we see that $|j\ell h|
\le \frac12 A_0^{-1}\delta^{A} N^{d-j+1}$ if $C > 1$ is large enough.

In this case, $\ell(Q_1' - Q_2') = \ell(Q_1 - Q_2) + j h\ell (\varepsilon_1-\varepsilon_2)$ and so
\begin{align*}
|\ell(Q_1-Q_2)|-|j\ell h|\le |\ell(Q_1'-Q_2')|
\le |\ell(Q_1-Q_2)|+|j\ell h|.
\end{align*}
From \eqref{eq:38a} and $|j\ell h| \le \frac12 A_0^{-1} \delta^A N^{d-j+1}$, we conclude
\begin{align}
\label{eq:39}
\frac{1}{2}A_0^{-1}\delta^{A} N^{d-j+1} \leq |\ell(Q_1'-Q_2')| \leq \frac{3}{2}A_0\delta^{-A} N^{d-j+1}.
\end{align}
This verifies \eqref{eq:26} in the case $Q_1 \not= Q_2$.

In either case, we see that $\deg(Q_1' - Q_2') = j-1$ and (see \eqref{eq:38} and \eqref{eq:39}) we can 
find a tolerance pair $(A_0', A')$ for $\mathcal P'$  depending on the tolerance
$(A_0, A)$ of $\mathcal P$ and the constants $C$ and $M$ such that
\begin{align}
\label{eq:43}
(A_0')^{-1}\delta^{A'} N^{d-j+1} \leq |\ell(Q_1'-Q_2')| \leq A_0'\delta^{-A'} N^{d-j+1}
\end{align}
holds, establishing \eqref{eq:26}.

\medskip

\paragraph{\textit{Proof of the conclusion from \ref{item:2}}}
Suppose that the collection $\mathcal P$ is of type II.  Then
$\deg(P_1)=\ldots=\deg(P_m)=j$ and
$v(\mathcal P)=(0,\ldots, 0, v_j(\mathcal P), 0, 0, \ldots)$ with
$v_j(\mathcal P)>1$.
To establish \eqref{eq:27} we will proceed
in a similar way as in \ref{item:1}.
If $P\in\mathcal P=\mathcal P_j$ and $\ell(P)\neq\ell(P_1)$, then 
\begin{align}
\label{eq:44}
\begin{gathered}
\deg(P-P_1)=\deg(P(\cdot+h)-P_1)=j, \\
\ell(P-P_1)=\ell(P(\cdot+h)-P_1)=\ell(P)- \ell(P_1).
\end{gathered}
\end{align}
If $P\in\mathcal P=\mathcal P_j$ and $\ell(P)=\ell(P_1)$, then by the
fact that $\mathcal P$ is $(d, j)$-admissible and by \eqref{eq:26} we
see that
\begin{align}
\label{eq:45}
\deg(P-P_1)<j,
\quad\text{ and } \quad
\deg(P(\cdot+h)-P_1)<j.
\end{align}
This shows that $v_{k}(\mathcal P')\ge0$ for all $k\in \bra{j-1}$
and $v_j(\mathcal P')=v_j(\mathcal P)-1$. Hence \eqref{eq:27}
holds. We now show that $\mathcal P'$ is $(d, j)$-admissible.

We begin with verifying \eqref{eq:24} for $P'\in\mathcal P_j'$.  We
may write $P'=P(\cdot+\varepsilon h)-P_1$ for some $P\in\mathcal P_j$
such that $\ell(P)\neq\ell(P_1)$ and $\varepsilon\in\{0, 1\}$. Since
$\mathcal P$ is $(d, j)$-admissible, using \eqref{eq:44} and
\eqref{eq:25} (with $\ell(P)- \ell(P_1)$ in place of
$\ell(P)- \ell(Q)$) we obtain \eqref{eq:41} which is \eqref{eq:24} for $P' \in \mathcal  P_j'$.

We now verify \eqref{eq:25} for $Q_1', Q_2'\in\mathcal P'_j$ with
$\ell(Q_1')\neq\ell(Q_2')$. As in \eqref{eq:40} we may write
$Q_1'=Q_1(\cdot+\varepsilon_1 h)-P_1$, and
$Q_2'=Q_2(\cdot+\varepsilon_2 h)-P_1$ for some
$Q_1, Q_2\in\mathcal P_j$ and
$\varepsilon_1, \varepsilon_2\in\{0, 1\}$ such that
$\ell(Q_1)\neq\ell(P_1)$ and $\ell(Q_2)\neq\ell(P_1)$.  By
\eqref{eq:44} we have $\ell(Q_1')=\ell(Q_1)-\ell(P_1)$ and
$\ell(Q_2')=\ell(Q_2)-\ell(P_1)$.  Then $\ell(Q_1)\neq\ell(Q_2)$ and
\eqref{eq:42} is verified by appealing to \eqref{eq:25} (with
$\ell(Q_1) - \ell(Q_2)$ in place of $\ell(P) - \ell(Q)$).

We finally verify \eqref{eq:26} for $Q_1', Q_2'\in\mathcal P'_j$ as in
\eqref{eq:40} such that $Q_1'\neq Q_2'$ and
$\ell(Q_1')=\ell(Q_2')=\ell$.  By \eqref{eq:44},
$\ell(Q_1)-\ell(P_1)=\ell(Q_2)-\ell(P_1)=\ell$ and since $\mathcal P$ is $(d,j)$-admissible,
we see that $\ell$ satisfies \eqref{eq:37}. Now by following the last part of the proof
from \ref{item:1}, we conclude that \eqref{eq:43} holds.

\medskip

\paragraph{\textit{Proof of the conclusion from \ref{item:3}}}
Suppose that the collection $\mathcal P$ is of type III.  Then
$\deg(P_1)=\ldots=\deg(P_m)=j$ and
$v(\mathcal P)=(0,\ldots, 0, v_j(\mathcal P), 0, 0, \ldots)$ with
$v_j(\mathcal P)=1$, thus $\ell(P_1)=\ldots=\ell(P_m):=\ell$. To
establish \eqref{eq:33} we will proceed in a similar way as in
\ref{item:1} and \ref{item:2}.  If $P\in\mathcal P_j$ and
$\ell(P)=\ell$, then \eqref{eq:37} holds for $\ell$ and once again \eqref{eq:45}
holds.  This in turn implies that $v_{j-1}(\mathcal P')>0$ and $v_{k}(\mathcal P')=0$ for all
$k\neq j-1$. Hence \eqref{eq:33} holds. We
now show that $\mathcal P'$ is $(d, j-1)$-admissible.

We begin with verifying \eqref{eq:24} (or equivalently \eqref{eq:41} with $j$ replaced
by $j-1$) for $P'\in\mathcal P_{j-1}'$.  We
may write $P'=P(\cdot+\varepsilon h)-P_1$ for some $P\in\mathcal P_j$
such that $\ell(P)=\ell(P_1)$ and $\varepsilon\in\{0, 1\}$.
Then
\begin{align}
\label{eq:46}
\ell(P')=\ell(P(\cdot+\varepsilon h)-P_1)= \ell(P-P_1) + j h \ell \varepsilon.
\end{align}
As in \ref{item:1} we have $\delta'\delta^2N/C^2\le |h| \le \delta'N$, where 
$\delta':=\delta^{M}(CM)^{-1}$ for a large number $M\ge1$, which will be
chosen later. Furthermore if $P \not= P_1$, then $A_0^{-1} \delta^A N^{d-j+1} \le |\ell(P-P_1)| \le A_0 \delta^{-A} N^{d-j+1}$
since $\mathcal P$ is $(d,j)$-admissible and so \eqref{eq:26} holds with $Q = P_1$.
This takes care of the case $\varepsilon = 0$.

If $\varepsilon=1$ and $P = P_1$, then \eqref{eq:38} gives the desired bound for $|\ell(P')|$. When $P \not= P_1$, we use 
the upper bound from \eqref{eq:38}  
\begin{align}
\label{eq:47}
 |jh\ell| \leq |j|A_0(CM)^{-1}\delta^{M-A}N^{d-j+1} \le \frac12 A_0^{-1} \delta^{-A} N^{d-j+1}
\end{align}
when $M = \max(2A, 2|j|A_0^2)$ and $C>1$ chosen large enough. Thus, as before,
condition \eqref{eq:24} holds for $P'$ with some tolerance pair $(A_0', A')$ as desired. 

For $Q_1' \not= Q_2'\in\mathcal P'_{j-1}$,  we may write
$Q_1'=Q_1(\cdot+\varepsilon_1 h)-P_1$, and
$Q_2'=Q_2(\cdot+\varepsilon_2 h)-P_1$ for some
$Q_1, Q_2\in\mathcal P_j$ and $\varepsilon_1, \varepsilon_2\in\{0, 1\}$ such that
$\ell(Q_1)=\ell(Q_2) = \ell(P_1) = \ell$. We have $\ell(Q_1 - P_1) - \ell(Q_2 - P_1) = \ell(Q_1 - Q_2)$
and so by \eqref{eq:46},
\begin{align}
\label{eq:48}
\ell(Q_1')-\ell(Q_2')= \ell(Q_1 - Q_2) + j h \ell (\varepsilon_1 - \varepsilon_2).
\end{align}
We consider two cases.

If $Q_1 = Q_2$, then necessarily $|\varepsilon_1 - \varepsilon_2| = 1$ and so $\ell(Q_1') \not= \ell(Q_2')$,
$\deg(Q_1' - Q_2') = j-1$ and 
\eqref{eq:38} shows that \eqref{eq:25} holds for $Q_1', Q_2' \in\mathcal P'_{j-1}$.

If $Q_1 \not= Q_2$, then $A_0^{-1} \delta^A N^{d-j+1} \le |\ell(Q_1-Q_2)| \le A_0 \delta^{-A} N^{d-j+1}$
since $\mathcal P$ is $(d,j)$-admissible and so \eqref{eq:26} holds with $P=Q_1 $ and $Q=Q_2$.
From \eqref{eq:47}, we see that $\ell(Q_1') \not= \ell(Q_2')$ and 
\eqref{eq:48} implies that \eqref{eq:25} holds for $Q_1', Q_2'\in\mathcal P'_{j-1}$.

In either case, we see that \eqref{eq:26} is vacuously satisfied by
$\mathcal P'$ and \eqref{eq:25} holds for
$Q_1', Q_2' \in \mathcal P'_{j-1}$ with (necessarily)
$\ell(Q_1')\neq\ell(Q_2')$.

Concluding, we are able to find a
tolerance pair $(A_0', A')$ for $\mathcal P'$ depending on the
tolerance $(A_0, A)$ of $\mathcal P$ and the constants $C$ and $M$
such that the required estimates for \eqref{eq:46} and \eqref{eq:48}
hold. This completes the proof of Proposition \ref{prop:PETiterate}.
\end{proof}

\subsection{Degree-lowering}
Here, we establish a modulated version of the inverse theorem, which will imply Theorem \ref{thm:inverse}.

\begin{theorem}[Inverse theorem for modulated $(m+1)$-linear forms]
\label{thm:inversem}
Let $N\ge 1$ be a scale, and let $0<\delta\le 1$, $m\in\Z_+$ and $n\in\N$ be given. Let $\mathcal P:=\{P_1,\ldots, P_{m}\}$ and 
$\mathcal Q:=\{Q_1,\ldots, Q_n\}$ be
collections of polynomials such that 
\begin{gather*}
1\le \deg{P_1}<\ldots<\deg{P_{m}}<\deg{Q_1}<\ldots<\deg{Q_n}.
\end{gather*}
Let $f_0, f_1,\ldots, f_m\in L^0(\KK)$ be $1$-bounded functions
supported on an interval $I\subset \KK$ of measure $N_0 := N^{\deg{P_m}}$.
For $n\in\Z_+$ we define an
$(m+1)$-linear form corresponding to the triple
$(\mathcal P, \mathcal Q; N)$ and a frequency vector
$\xi=(\xi_1,\ldots, \xi_n)\in \KK^n$ by
\begin{align}
\label{eq:51}
\Lambda_{\mathcal P; N}^{\mathcal Q; \xi}(f_0,\ldots, f_m):=\frac{1}{N_0}
\int_{\KK^2}f_0(x)\prod_{i=1}^mf_{i}(x-P_i(y)){\rm e}\Big(\sum_{j=1}^n\xi_jQ_j(y)\Big)d\mu_{[N]}(y)d\mu(x).
\end{align}
For $n=0$ we set $\mathcal Q=\emptyset$ and we simply write 
$\Lambda_{\mathcal P; N}^{\mathcal Q; \xi}(f_0,\ldots, f_m):=\Lambda_{\mathcal P; N}(f_0,\ldots, f_m)$ as in  \eqref{eq:6}.
Suppose that
\begin{align}
\label{eq:52}
|\Lambda_{\mathcal P; N}^{\mathcal Q; \xi}(f_0,\ldots, f_m)|\ge\delta.
\end{align}
Then there exists a $C_1 = C_1(\mathcal P) \gg 1$ such that 
\begin{align}
\label{eq:53}
N_0^{-1}\big\| \mu_{[N_1]}*f_1\big\|_{L^1(\KK)} \gtrsim_{\mathcal P} \delta^{O_{\mathcal P}(1)},
\end{align}
for any $N_1 = \delta^C N^{\deg{P_1}}$ with $C\ge C_1$.
\end{theorem}

If necessary  we will also write
$\Lambda_{\mathcal P; N}^{\mathcal Q; \xi}(f_0,\ldots, f_m)=\Lambda_{\mathcal P; N, I}^{\mathcal Q; \xi}(f_0,\ldots, f_m)$
in order to emphasise that the functions $f_0, f_1,\ldots, f_m$ are
supported on $I$.

We first show how the Gowers box norms control the dual functions.
The dual function, or more precisely the $m$-th dual function,
corresponding to \eqref{eq:51} is defined as
\begin{align}
\label{eq:55}
F_m^{\xi}(x):=\int_{\KK} F_{m; y}^{\xi}(x) d\mu_{[N]}(y), \qquad x\in \KK,
\end{align}
where
\begin{align}
\label{eq:56}
F_{m; y}^{\xi}(x):=f_0(x+P_m(y))\prod_{i=1}^{m-1}f_{i}(x-P_i(y)+P_m(y)){\rm e}\Big(\sum_{j=1}^n\xi_jQ_j(y)\Big).
\end{align}

\begin{proposition}[Gowers box norms  control  the dual functions]
\label{prop:dual}
Let $N\ge 1$ be a scale, and let $0<\delta\le 1$, $d, m\in\Z_+$ with $m\ge 2$ and $n\in\N$ be given. Let
$\mathcal P:=\{P_1,\ldots, P_{m}\}$ and
$\mathcal Q:=\{Q_1,\ldots, Q_n\}$ be collections of polynomials such
that $\mathcal P$ is $(d, d)$-admissible and
$$1\le \deg{P_1}\le\ldots\le\deg{P_{m}}\le\deg{Q_1}\le\ldots\le\deg{Q_n}.$$
Let $f_0, f_1,\ldots, f_m\in L^0(\KK)$ be $1$-bounded functions
supported on an interval $I\subset \KK$ of measure $N_0 := N^{\deg{P_m}}$.
For $\xi\in \KK^n$, let $F_m^{\xi}$ be the dual function defined in
\eqref{eq:55}.  Suppose that \eqref{eq:52} is satisfied. Then for the exponent
$s\in\Z_+$ which appears in the conclusion of Theorem \ref{thm:Us}, we have
\begin{align}
\label{eq:66}
\|F_m^{\xi}\|_{\square_{[H_1], \ldots, [H_{s+1}]}^{s+1}(I)} \gtrsim_{\mathcal P} \delta^{O_{\mathcal P}(1)},
\end{align}
where $H_i\simeq \delta^{O_{\mathcal P}(1)}N^{\deg(P_m)}$ for
$i\in\bra{s+1}$.
\end{proposition}

\begin{proof}
By changing the variables $x\mapsto x+P_m(y)$ in \eqref{eq:51}  we may write
\begin{align*}
\Lambda_{\mathcal P; N}^{\mathcal Q; \xi}(f_0,\ldots, f_m)=
\frac{1}{N_0}\int_{\KK}\Big(\int_{\KK} F_{m; y}^{\xi}(x) d\mu_{[N]}(y)\Big)f_m(x) d\mu(x).
\end{align*}
By the Cauchy--Schwarz inequality (observing once again that $\|f_m\|_{L^2(\KK)}^2 \le N_0$), we have
\begin{align*}
\delta^2 &\le \frac{1}{N_0} \int_{\KK} \Big|\int_{\KK} F_{m; y}^{\xi}(x)d\mu_{[N]}(y)\Big|^2 d\mu(x) \\
& =  \frac{1}{N_0}\bigg|\int_{\KK^3}F_{m; y_1}^{\xi}(x)\overline{F_{m; y_2}^{\xi}(x)}d\mu_{[N]}^{\otimes2}(y_1, y_2)d\mu(x)\bigg| \\
&= |\Lambda_{\mathcal P; N}^{\mathcal Q; \xi}(f_0, f_1,\ldots, f_{m-1}, \overline{F_m^{\xi}})|, 
\end{align*}
where in the last step we changed variables $x\mapsto x-P_m(y_1)$. Denote $g_m:=\overline{F_m^{\xi}}$, and  $g_j:=f_j$ for $j\in\bra{m-1}_0$.
Our strategy will be to reduce the matter to Theorem \ref{thm:Us} with the family $\mathcal P$. Observe that  $g_j$ is a $1$-bounded function and $\|g_j\|_{L^1(\KK)}\lesssim N_0$ for all $j\in\bra{m}_0$.
Changing the variables $x \mapsto x+h$ in the definition of
$\Lambda_{\mathcal P; N}^{\mathcal Q; \xi}$ and averaging over
$h\in [H_{s+1}]$ where $H_{s+1} =\delta^{O(1)}N^{{\rm deg} P_m}$, we have 
\begin{align*}
\delta^4&\le |\Lambda_{\mathcal P; N}^{\mathcal Q; \xi}(g_0,\ldots, g_m)|^2\\
&\lesssim \frac{1}{N_0}\int_{\KK^2}\Big|\int_{\KK}g_0(x+h)\prod_{i=1}^mg_{i}(x+h-P_i(y)) d\mu_{[H_{s+1}]}(h)\Big|^2
d\mu_{[N]}(y)d\mu(x),
\end{align*}
where in the last line we have used the Cauchy--Schwarz inequality in the $x$ and  $y$ variables, noting that $x \to g_0(x+h)$ is supported
a fixed dilate of $I$ for every $h\in [H_{s+1}]$.
By another change of variables  we obtain
\begin{align*}
\int_{\KK}\Lambda_{\mathcal P; N}(\Delta_hg_0,\ldots, \Delta_hg_m)d\nu_{[H_{s+1}]}(h)\gtrsim \delta^4.
\end{align*}
Now we may find a measurable set $X \subseteq [H_{s+1}]$ such that 
\begin{align*}
|\Lambda_{\mathcal P; N}(\Delta_hg_0,\ldots, \Delta_hg_m)|\gtrsim \delta^4
\end{align*}
for all $h\in X$ and $\nu_{[H_{s+1}]}(X)\gtrsim \delta^4$. Since $\Delta_hg_j$ is a $1$-bounded function and
$\|\Delta_hg_j\|_{L^1(\KK)}\lesssim N_0$ for all $j\in\bra{m}_0$, we
may invoke Theorem \ref{thm:Us} and conclude that 
\begin{align*}
\|\Delta_hF_m^{\xi}\|_{\square_{[H_1], \ldots, [H_s]}^s(I)}=\|\Delta_hg_m\|_{\square_{[H_1], \ldots, [H_s]}^s(I)}
\gtrsim_{\mathcal P} \delta^{O_{\mathcal P}(1)}
\end{align*}
for all $h\in X$, where
$H_i\simeq \delta^{O_{\mathcal P}(1)}N^{\deg(P_m)}$ for $i\in\bra{s}$. Averaging over $h\in X$ and using $\nu_{[H_{s+1}]}(X)\gtrsim \delta^4$, we obtain
 \begin{align*}
\|F_m^{\xi}\|_{\square_{[H_1], \ldots, [H_{s+1}]}^{s+1}(I)}^{2^{s+1}} \ = \
\int_{\KK}\|\Delta_hF_m^{\xi}\|_{\square_{[H_1], \ldots, [H_s]}^s(I)}^{2^s}d\nu_{[H_{s+1}]}(h) \
\gtrsim_{\mathcal P} \ \delta^{O_{\mathcal P}(1)},
\end{align*}
which is \eqref{eq:66} as desired. 
\end{proof}

We first establish a simple consequence of the oscillatory integral bound \eqref{osc-int-est}
which will be important later.
\begin{lemma}
\label{lem:vdc-osc}
Let $N>1$ be a scale, $m\in\Z_+$ and $n\in\N$ be given. Let $\mathcal P:=\{P_1,\ldots, P_{m}\}$ and 
$\mathcal Q:=\{Q_1,\ldots, Q_n\}$ be
collections of polynomials such that 
\begin{gather}
\label{eq:70}
1\le \deg{P_1}<\ldots<\deg{P_{m}}<\deg{Q_1}<\ldots<\deg{Q_n}.
\end{gather}
Define the multiplier corresponding to the families $\mathcal P$ and $\mathcal Q$ as follows: 
\begin{align*}
m_N^{\mathcal P, \mathcal Q}(\zeta, \xi):=\int_\KK e\Big(\sum_{i=1}^m\zeta_iP_i(y)+\sum_{j=1}^n\xi_jQ_j(y)\Big)d\mu_{[N]}(y),
\end{align*}
where $\zeta=(\zeta_1,\ldots, \zeta_m)\in \KK^m$ and $\xi=(\xi_1,\ldots, \xi_n)\in \KK^n$.
Let $0<\delta\le 1$ and suppose that
\begin{align}
\label{eq:54}
|m_N^{\mathcal P, \mathcal Q}(\zeta, \xi)|\ge \delta.
\end{align}
Then there exists a  large constant  $A\gtrsim_{\mathcal P, \mathcal Q} 1$ such that 
\begin{align}
\label{eq:78}
\begin{split}
N^{\deg(Q_j)}|\xi_{j}|&\lesssim \delta^{-A}, \quad \text{ for } \quad j\in\bra{n},\\
N^{\deg(P_j)}|\zeta_{j}|&\lesssim \delta^{-A}, \quad \text{ for } \quad j\in\bra{m}.
\end{split}
\end{align}
\end{lemma}

\begin{proof}
Fix an element $\alpha\in \KK$ such that $|\alpha| = N$ and make the change of variables $y \to \alpha y$ to write
$$
m_N^{\mathcal P, \mathcal Q}(\zeta, \xi) \ = \ \int_{B_1(0)} {\rm e} \Big(\sum_{i=1}^m\zeta_iP_i(\alpha y)+\sum_{j=1}^n\xi_jQ_j(\alpha y)\Big)
d\mu(y).
$$
Define $R(y):=\sum_{i=1}^m\zeta_iP_i(y)+\sum_{j=1}^n\xi_jQ_j(y)$. Then $R(y)$ may be rewritten as
\begin{align*}
R(y)=\sum_{l=1}^{\deg{Q_n}}\coe_l(R)y^l
\end{align*}
The oscillatory integral bound \eqref{osc-int-est} implies
\begin{align}
\label{eq:72}
|m_N^{\mathcal P, \mathcal Q}(\zeta, \xi)|\lesssim \bigg(1+\sum_{l=1}^{\deg{Q_n}}|\coe_l(R)|N^l\bigg)^{-1/\deg{Q_n}}.
\end{align}
Hence \eqref{eq:54} implies $\max_l |\coe_l(R)| N^{l} \lesssim \delta^{-d_{*}}$ where $d_{*} = \deg{Q_n}$ and the maximum is taken
over all $l \in \bra{\deg(Q_n)}$. From this, we see that for any
sufficiently large $A\ge d_{*}$,
\begin{align}
\label{eq:77}
|\coe_l(R)|N^l\le \delta^{-A}/A
\end{align}
for all $l\in \bra{\deg(Q_n)}$.

Using \eqref{eq:70} we observe that
\begin{gather}
\label{eq:73}
\coe_{\deg{Q_{j}}}(R)=\sum_{k=j}^n\coe_{\deg{Q_{j}}}(Q_{k})\xi_{k}, \quad \text{ for } \quad j\in\bra{n}),\\
\label{eq:74}\coe_{\deg{P_j}}(R)=\sum_{k=1}^n\coe_{\deg{P_j}}(Q_{k})\xi_{k}+\sum_{k=j}^m\coe_{\deg{P_j}}(P_k)\zeta_k
, \quad \text{ for } \quad j\in\bra{m}.
\end{gather}
Using \eqref{eq:73} for $j=n$, we see that \eqref{eq:77} implies \eqref{eq:78} for $N^{\deg{Q_n}} |\xi_n|$. Inductively we
now deduce, using \eqref{eq:73}, that \eqref{eq:77} implies that \eqref{eq:78} holds for all $N^{\deg{Q_j}} |\xi_j|$, $j\in\bra{n}$.
Similarly, using \eqref{eq:74} and \eqref{eq:77}, we see that that the second displayed equation in \eqref{eq:78} holds.
\end{proof}

The key ingredient in the proof of Theorem \ref{thm:inversem} will be a degree-lowering
argument, which reads as follows.
\begin{theorem}[Degree-lowering argument]
\label{thm:deg-low}
Let $N\ge  1$ be a scale and let $0<\delta\le 1$, $m\in\Z_+$ and $n\in\N$ be given. Let
$\mathcal P:=\{P_1,\ldots, P_{m}\}$ and
$\mathcal Q:=\{Q_1,\ldots, Q_n\}$ be collections of polynomials such
that 
\begin{gather*}
1\le \deg{P_1}<\ldots<\deg{P_{m}}<\deg{Q_1}<\ldots<\deg{Q_n}.
\end{gather*}
For $\xi\in \KK^n$, let $F_m^{\xi}$ be the dual function from \eqref{eq:55} corresponding to the form \eqref{eq:51} and 
 $1$-bounded functions $f_0, f_1,\ldots, f_{m-1}\in L^0(\KK)$ supported
on an interval $I\subset \KK$ of measure $N_0 := N^{\deg{P_m}}$.
Suppose that for some integer $s\in\Z_+$ one has
\begin{align}
\label{eq:67}
\|F_m^{\xi}\|_{\square_{[H_1], \ldots, [H_s]}^s(I)} \ge \delta,
\end{align}
where $H_i\simeq \delta^{O_{\mathcal P}(1)}N^{\deg(P_m)}$ for
$i\in\bra{s}$. Then 
\begin{align}
\label{eq:68}
\|F_m^{\xi}\|_{\square_{[H_1], \ldots, [H_{s-1}]}^{s-1}(I)} \gtrsim_{\mathcal P} \delta^{O_{\mathcal P}(1)}.
\end{align}
\end{theorem}

Assuming momentarily Theorem \ref{thm:deg-low} we prove Theorem \ref{thm:inversem}.   
\begin{proof}[Proof of Theorem \ref{thm:inversem}]
Our goal is to prove \eqref{eq:53} when
$$
\delta \ \le \ |\Lambda_{\mathcal{P};N}^{\mathcal{Q}; \xi}(f_0,\ldots, f_m)|.
$$
The proof is by induction on $m\in \Z_+$. We divide the proof into two
steps.  In the first step we establish the base case for $m=1$.  In
the second step we will use Theorem \ref{thm:deg-low} to establish the inductive step.
\paragraph{\bf Step 1.}
Assume that $m=1$ so that $N_0 = N^{\deg{P_1}}$.
For $\zeta\in \KK$ and $\xi=(\xi_1,\ldots, \xi_n)\in \KK^n$ we define the multiplier
\begin{align*}
m_N(\zeta, \xi):=\int_\KK  e\Big(-\zeta P_1(y) + \sum_{j=1}^n \xi_j Q_j(y)\Big)d\mu_{[N]}(y).
\end{align*}
We now express
\begin{align*}
\Lambda_{\mathcal{P};N}^{\mathcal{Q}; \xi}(g_0,g_1)
=N_0^{-1}\int_{\KK} \widehat{g_0}(-\zeta) \widehat{g_1}(\zeta) m_N(\zeta, \xi)d\mu(\zeta).
\end{align*}
Using the Cauchy--Schwarz inequality and Plancherel's theorem we see
\begin{align}
\label{eq:61}
|\Lambda_{\mathcal{P};N}^{\mathcal{Q}; \xi}(g_0,g_1)|\leq  N_0^{-1}\| g_0 \|_{L^2(\KK)}  \|g_1\|_{L^2(\KK)}
\sup_{\zeta \in \supp{(\widehat{g_0}\widehat{g_1})}} |m_N(\zeta, \xi)|.
\end{align}

When $\KK$ is non-archimedean, let $\varphi(x) = \ind{[1]}(x) = \ind{B_1(0)}(x)$ so that
$\widehat{\varphi}(\zeta) = \ind{[1]}(\xi)$. When $\KK$ is archimedean,
choose a Schwartz function $\varphi: \KK \to \KK$ such that
\begin{align*}
\ind{[1]}(\zeta)\le \widehat{\varphi}(\zeta)\le \ind{[2]}(\zeta), \qquad  \zeta\in \KK.
\end{align*}
For a scale $M$, we set $\varphi_M(x) = M^{-1} \varphi(M^{-1} x)$ when $\KK = {\mathbb R}$ and
when $\KK = {\mathbb C}$, we set
$\varphi_{M}(z) = M^{-1} \varphi(M^{-1/2} z)$. When
$\KK$ is non-archimedean,  
we set $\varphi_{M}(x) = M^{-1} \ind{[M]}(x)$.  

Consider two scales $M_1\simeq\delta^{C}N^{\deg{P_1}}$ and $N_1\simeq\delta^{2C}N^{\deg{P_1}}/C$.
Then we obtain
\begin{align*}
\delta\le |\Lambda_{\mathcal{P};N}^{\mathcal{Q}; \xi}(f_0,f_1)|\le
|\Lambda_{\mathcal{P};N}^{\mathcal{Q}; \xi}(f_0,\varphi_{M_1}*f_1)|+
|\Lambda_{\mathcal{P};N}^{\mathcal{Q}; \xi}(f_0,f_1-\varphi_{M_1}*f_1)|.
\end{align*}
Note that
\begin{align*}
|\Lambda_{\mathcal{P};N}^{\mathcal{Q}; \xi}(f_0,\varphi_{M_1}*f_1)|
\le N_0^{-1}\|f_0\|_{L^{\infty}(\KK)}\|\varphi_{M_1}*f_1\|_{L^1(\KK)} \le N_0^{-1} \|\varphi_{M_1}*f_1\|_{L^1(\KK)},
\end{align*}
and
\begin{align*}
\|\varphi_{M_1}*f_1\|_{L^1(\KK)}&\le \|\varphi_{M_1}*\mu_{[N_1]}*f_1\|_{L^1(\KK)}+
\|(\varphi_{M_1}-\varphi_{M_1}*\mu_{[N_1]})*f_1\|_{L^1(\KK)}\\
&\lesssim \|\mu_{[N_1]}*f_1\|_{L^1(\KK)} +C^{-1}\delta^C N_0,
\end{align*}
since $\varphi_{M_1}-\varphi_{M_1}*\mu_{[N_1]} = 0$ when $\KK$ is non-archimedean and when $\KK$
is archimedean, we have the pointwise bound
\begin{align*}
|\varphi_{M_1}(x)-\varphi_{M_1}*\mu_{[N_1]}(x)|\lesssim C^{-1}\delta^C  \, M_1^{-1}\big(1+M_1^{-1}|x|\big)^{-10}.
\end{align*}
If $C\ge1$ is sufficiently large then we may write
\begin{align}
\label{eq:61a}
\delta\lesssim |\Lambda_{\mathcal{P};N}^{\mathcal{Q}; \xi}(f_0,f_1)|\le
N_0^{-1}\|\mu_{[N_1]}*f_1\|_{L^1(\KK)}+
|\Lambda_{\mathcal{P};N}^{\mathcal{Q}; \xi}(f_0,f_1-\varphi_{M_1}*f_1)|.
\end{align}
By \eqref{eq:61} we have that
\begin{align}
\label{eq:61b}
|\Lambda_{\mathcal{P};N}^{\mathcal{Q}; \xi}(f_0,f_1-\varphi_{M_1}*f_1)|\lesssim
\sup_{\zeta\in \KK: |\zeta|\ge M_1^{-1} } |m_N(\zeta, \xi)|,
\end{align}
since $\| f_0 \|_{L^2(\KK)}\le N_0^{1/2}$ and $\|f_1\|_{L^2(\KK)}\le N_0^{1/2}$.
We now prove that
\begin{align}
\label{eq:62}
\sup_{\zeta\in \KK: |\zeta|\ge M_1^{-1} } |m_N(\zeta, \xi)|\lesssim \delta^{2}.
\end{align}
Suppose that inequality \eqref{eq:62} does not hold, then one has
\begin{align*}
|m_N(\zeta, \xi)|\gtrsim \delta^{2}
\end{align*}
for some $\zeta\in \KK$ so that $|\zeta|\ge M_1^{-1}$. Then Lemma \ref{lem:vdc-osc} implies
$N^{\deg{P_1}} |\zeta| \lesssim \delta^{-A}$ for some large, fixed $A\gtrsim 1$ by \eqref{eq:78}.
Since $M_1 = \delta^C N^{\deg{P_1}}$, we have $\delta^{-C} \lesssim \delta^{-A}$ which is
a contradiction if $C \gg A$. Thus \eqref{eq:62} holds.

Hence by \eqref{eq:62}, \eqref{eq:61b} and \eqref{eq:61a}, we see that
$$
\delta \lesssim N_0^{-1}\|\mu_{[N_1]}*f_1\|_{L^1(\KK)}
$$
which establishes Theorem \ref{thm:inversem} when $m=1$.

\paragraph{\bf Step 2.} We now assume that Theorem \ref{thm:inversem}
is true for $m-1$ in place of $m$ for some integer $m\ge2$. Using
Theorem \ref{thm:deg-low} we show that this implies Theorem
\ref{thm:inversem} for $m\ge2$. Note that bound \eqref{eq:52} implies
inequality \eqref{eq:66} from Proposition \ref{prop:dual}. Now by
Theorem \ref{thm:deg-low} applied $s-2$ times we may conclude that
\begin{align*}
\|F_m^{\xi}\|_{\square_{[H_1], [H_2]}^2(I)} \gtrsim_{\mathcal P} \delta^{O_{\mathcal P}(1)},
\end{align*}
where $H_1, H_2\simeq \delta^{O_{\mathcal P}(1)}N^{\deg{P_m}}$. By Lemma \ref{2.10} we can find a $\xi_0\in \KK$ such that
\begin{align}
\label{eq:69}
N_0^{-1}\big|\widehat{F_m^{\xi}}(\xi_0)\big|  \gtrsim_{\mathcal P} \delta^{O_{\mathcal P}(1)},
\end{align}
since $N_0 = N^{\deg{P_m}}$.
By definitions \eqref{eq:55} and \eqref{eq:56} and making the change of variables $x\mapsto x-P_m(y)$ we may write
\begin{align*}
N_0^{-1}\widehat{F_m^{\xi}}(\xi_0)&=N_0^{-1}\iint_{\KK^2}F_{m;y}^{\xi}(x)e(-\xi_0 x)d\mu_{[N]}(y)d\mu(x)\\
=N_0^{-1}&\iint_{\KK^2}{\rm M}_{\xi_0}f_0(x)\prod_{i=1}^{m-1}f_{i}(x-P_i(y))e\Big(\xi_0 P_m(y)+\sum_{j=1}^n\xi_jQ_j(y)\Big)d\mu_{[N]}(y)d\mu(x)\\
&=M^{-1} \Lambda_{\mathcal P'; N}^{\mathcal Q', \xi'}({\rm M}_{\xi_0}f_0, f_1,\ldots, f_{m-1}),
\end{align*}
where ${\rm M}_{\xi_0}f_0(x):={\rm e}(-\xi_0 x)f_0(x), \, \mathcal P':=\mathcal P\setminus\{P_m\}, \, \mathcal Q':=\mathcal Q\cup\{P_m\}, \, \xi':=(\xi_0, \xi_1,\ldots, \xi_n)\in \KK^{n+1}$ and 
$M:=N^{\deg(P_m)-\deg(P_{m-1})}$. The parameter $N_0' := N_0 M^{-1}$ is what appears in the $m$-linear form
$\Lambda_{\mathcal P';N}^{\mathcal Q', \xi'}$. We note that  $N_0' = N^{\deg{P_{m-1}}}$.

Thus \eqref{eq:69} implies
\begin{align*}
M^{-1} |\Lambda_{\mathcal P'; N, I}^{\mathcal Q', \xi'}({\rm M}_{\xi_0}f_0, f_1,\ldots, f_{m-1})|\gtrsim \delta^{O(1)}.
\end{align*}
By translation invariance we may assume that all functions $f_0, f_1,\ldots, f_{m-1}$ are supported in $[N_0]$.
 We can partition $[N_0] = \bigcup_{k\in \bra{L}} E_k$ into $L \simeq M$ sets, each with measure $\simeq N_0'$
contained in an interval $I_k$ lying in an $O(N_0')$ neighbourhood of $E_k$.  Furthermore
$E_k$ is an $O(N_1)$ neighbourhood of a set $F_k$ such that $\mu(E_k\setminus F_k) \lesssim N_1$ and 
${\rm supp}(\ind{F_k}*\mu_{[N_1]}) \subseteq E_k$. Here $N_1 \simeq \delta^{O_{\mathcal P}(1)} N^{\deg(P_1)}$. 
In the non-archimedean setting, this decomposition is straightforward; in this case, we can take $F_k = E_k = I_k$.
If fact if $N_0 = q^{n_0}$ and $N_0' = q^{n_0 - \ell}$ so that $M = q^{\ell}$, then 
\begin{align*}
[N_0]  \  = \  B_{q^{n_0}}(0) \ = \ \bigcup_{x \in {\mathcal F}} B_{q^{n_0 -\ell}}(x)  
\end{align*}
gives our partition of $[N_0]$ where ${\mathcal F} = \{ x = \sum_{j=0}^{\ell-1} x_j \pi^{-n_0 + j} : x_j \in o_\KK/m_\KK \}$. Note $\# {\mathcal F} = q^{\ell} = M$. When $\KK = {\mathbb R}$, one simply decomposes the interval $[N_0] = [-N_0, N_0]$ into
$M$ subintervals $(E_k)_{k\in \bra{L}}$ of equal length and then extend and shrink to obtain intervals $I_k$ and $F_k$ with the desired properties.

When $\KK = {\mathbb C}$, the set $[N_0]$ is a disc and the decomposition is not as straightforward but not difficult to construct
by starting with a mesh of squares of side length $\sqrt{N_0'}$ which cover $[N_0]$. It is important that for this case
(when $\KK = {\mathbb C}$) that we allow the sets $E_k$ and $F_k$ to be general sets (not necessarily intervals)
with the above properties. The picture should be clear.

Hence by changing variables $x \to x+P_1(y)$ and then back again,
\begin{align*}
\ \ \ \ \ \ &M^{-1} \Lambda_{\mathcal P'; N, I}^{\mathcal Q', \xi'}({\rm M}_{\xi_0}f_0, f_1,\ldots, f_{m-1})  \\
= N_0^{-1} &\sum_{k\in \bra{L}}  \iint_{E_k\times \KK}{\rm M}_{\xi_0}f_0(x)\prod_{i=1}^{m-1}f_{i}(x-P_i(y))e\Big(\xi_0 P_m(y)+\sum_{j=1}^n\xi_jQ_j(y)\Big)d\mu_{[N]}(y)d\mu(x)\\
= N_0^{-1}&\sum_{k \in \bra{L}} 
\iint_{\KK^2}f_0^k(x)g^k(x-P_1(y))\prod_{i=2}^{m-1}f_{i}^k(x-P_i(y))e\Big(\xi_0 P_m(y)+\sum_{j=1}^n\xi_jQ_j(y)\Big)d\mu_{[N]}(y)d\mu(x)\\
&=M^{-1} \sum_{k\in \bra{L}} \Lambda_{\mathcal P'; N, I_k}^{\mathcal Q', \xi'}(f_0^k, g^k,f_2^k, \ldots, f_{m-1}^k),
\end{align*}
where 
$f_0^k:={\rm M}_{\xi_0}f_0\ind{I_k}, f_2^k:=f_2\ind{I_k}, \ldots, f_{m-1}^k:=f_{m-1}\ind{I_k}$ and $g^k = f_1 \ind{E_k}$.

By the pigeonhole principle there exists $L_0\subseteq \bra{L}$ such that $\# L_0\gtrsim \delta^{O_{\mathcal P'}(1)}M$ and 
for every $k\in L_0$ we have
\begin{align*}
|\Lambda_{\mathcal P'; N, I_k}^{\mathcal Q', \xi'}(f_0^k, g^k, f_2^k,\ldots, f_{m-1}^k)|\gtrsim \delta^{O_{\mathcal P'}(1)}.
\end{align*}
By the inductive hypothesis, we have 
\begin{align*}
(N_0')^{-1}\big\|\mu_{[N_1]}*(f_1\ind{E_k})\big\|_{L^1(\KK)}\gtrsim \delta^{O_{\mathcal P'}(1)}
\end{align*}
for every $k\in L_0$ and for every $N_1 = \delta^C N^{\deg{P_1}}$ with $C \ge C_1(\mathcal P')$.
Note that
\begin{align}
\label{eq:1}
(N_0')^{-1} \|\mu_{[N_1]}*(f_1(\ind{E_k} - \ind{F_k}))\|_{L^1(\KK)} \lesssim N_1 (N_0')^{-1} \lesssim \delta^{C} \ \ \
\end{align}
and hence  for $C\gg 1$ large  enough,
\begin{align*} 
(N_0')^{-1}\big\|\mu_{[N_1]}*(f_1\ind{F_k})\big\|_{L^1(\KK)}\gtrsim \delta^{O_{\mathcal P}(1)} \ \ \ {\rm for \ every}
\ \ k \in L_0.
\end{align*}
Now we can sum over $k\in L_0$, using the bound $\# L_0\gtrsim \delta^{O(1)}M$ and the pairwise disjoint supports
of 
$(\mu_{[N_1]}*(f_1 \ind{F_k}))_{k\in \bra{L}}$, 
 we obtain
\begin{align*}
N_0^{-1}\Big\|\mu_{[N_1]}*\Big(\sum_{k\in \bra{L}}f_1\ind{F_k}\Big)\Big\|_{L^1(\KK)} &\ge N_0^{-1} \sum_{k\in \bra{L}} \big\|\mu_{[N_1]}*(f_1 \ind{F_k})\big\|_{L^1(\KK)} \\
&\ge M^{-1} \sum_{k\in L_0} (N_0')^{-1} \big\|\mu_{[N_1]}*(f_1 \ind{F_k})\big\|_{L^1(\KK)} \
\gtrsim \ \delta^{O_{\mathcal P}(1)},
\end{align*}
which by \eqref{eq:1} yields
\begin{align*}
N_0^{-1}\big\|\mu_{[N_1]}*f_1\big\|_{L^1(\KK)} 
\gtrsim \ \delta^{O_{\mathcal P}(1)},
\end{align*}
as desired.
\end{proof}

We now  state two auxiliary technical lemmas which will be needed in the 
proof of Theorem \ref{thm:deg-low}. For $\omega = (\omega_1, \ldots, \omega_n) \in \{0,1\}^n$
and $h = (h_1, \ldots, h_n) \in \KK^n$, we write $\omega \cdot h = \sum_{i=1}^n \omega_i h_i$
and $1-\omega = (1-\omega_1, \ldots, 1-\omega_n)$.

\begin{lemma}
\label{lem:1}
Let $N\ge 1$ be a scale and let $0<\delta\le 1$, $m\in\Z_+$ with $m\ge 2$, $n\in\N$ and scales
$H_1,\ldots, H_n$ with each $H_i\le N$ be given.  Assume that $\phi:X\to\RR$
is a measurable function defined on a measurable set
$X\subseteq H:=\prod_{i=1}^n [H_i]$.  Let
$\mathcal P:=\{P_1,\ldots, P_{m}\}$ and
$\mathcal Q:=\{Q_1,\ldots, Q_n\}$ be collections of polynomials.  For $\xi\in \KK^n$, let
$F_m^{\xi}$ be the dual function defined in \eqref{eq:55} that
corresponds to the form \eqref{eq:51} and $1$-bounded functions
$f_0, f_1,\ldots, f_{m-1}\in L^0(\KK)$ supported on an interval $I\subset \KK$ of
measure $N_0 := N^{\deg{P_m}}$. Suppose that
\begin{align}
\label{eq:57}
\int_X\big|N^{-1}_0\reallywidehat{{\Delta}_h^n F_m^{\xi}}(\phi(h))\big|^2
d\Big(\bigotimes_{i=1}^{n}\nu_{[H_i]}\Big)(h)\ge \delta.
\end{align}
Then 
\begin{align}
\label{eq:81}
\int_{\square_n(X)}\bigg|N_0^{-1}\int_{\KK} F_m(x; h, h')e(-\psi(h, h')x)d\mu(x)\bigg|^2
d\Big(\bigotimes_{i=1}^{n}\nu_{[H_i]}\Big)^{\otimes2}(h, h')
\gtrsim \delta^{O(1)},
\end{align}
where
\begin{align*}
\square_n(X):=\Big\{(h, h')\in H^2: \omega \cdot h + (1-\omega)\cdot h' \in X \text{ for every } \omega\in\{0, 1\}^n\Big\},
\end{align*}
and
\begin{gather*}
F_m(x; h, h'):=\int_{\KK}{\Delta}_{h'-h}^nf_0(x+P_m(y))\prod_{i=1}^{m-1}
{\Delta}_{h'-h}^nf_{i}(x-P_i(y)+P_m(y))d\mu_{[N]}(y),\\
\psi(h, h'):=\sum_{\omega\in\{0, 1\}^n}(-1)^{|\omega|}\phi\Big(\omega \cdot h + (1-\omega)\cdot h'\Big).
\end{gather*}
\end{lemma}

\begin{proof}
We shall write ${\bm \nu}_n:=\bigotimes_{i=1}^{n}\nu_{[H_i]}$.
Using \eqref{product}, \eqref{eq:55} and \eqref{eq:56} we see that the left-hand side of 
\eqref{eq:57} can be written as
\begin{align*}
\frac{1}{N_0^2} \int_{\KK^{2^{n+1}}}\int_{\KK^2}\int_{\KK^n}G_0(x, z, h;{\bm y})
d{\bm \nu}_n(h)d\mu(x)d\mu(z)d\mu^{\otimes 2^{n+1}}_{[N]}({\bm y})
\end{align*}
where for ${\bm y}=(y_{(\omega, 0)}, y_{(\omega, 1)})_{\omega\in \{0, 1\}^n}\in \KK^{2^{n+1}}$, $x, z\in \KK$ and $h\in \KK^n$. We have set
\begin{gather*}
G_0(x, z, h;{\bm y})
:=\ind{X}(h) {\rm e} (-\phi(h)(x-z))\prod_{\omega\in \{0, 1\}^n} {\mathcal C}^{|\omega|} F_{m; y_{(\omega, 0)}}^{\xi}(x+ h\cdot\omega)\overline{F_{m; y_{(\omega, 1)}}^{\xi}(z+ h\cdot\omega)} .
\end{gather*}
Write elements in $X$ as $(h_1, h)$ with $h_1 \in \KK$ and apply the Cauchy--Schwarz inequality in all but the $h_1$ variable
(noting that $(x,z) \to G_0(x, z, h;{\bm y})$ is supported in a product of intervals of measure $\simeq N_0^2$) to conclude
\begin{align}
\label{eq:58}
\frac{1}{N_0^2} \int_{\KK^{2^{n}}}\int_{\KK^2}\int_{\KK^{n-1}} H_0(x, z, h;{\bm y})
d{\bm \nu}_{n-1}(h)d\mu(x)d\mu(z)d\mu^{\otimes 2^{n}}_{[N]}({\bm y})\gtrsim \delta^{O(1)},
\end{align}
where
\begin{align*}
H_0(x, z, h;{\bm y}):=\bigg|\int_{\KK}G_0^1(x, z, (h_1, h);{\bm y})d\nu_{[H_1]}(h_1)\bigg|^2,
\end{align*}
and
\begin{gather*}
G_0^1(x, z, (h_1, h);{\bm y})
:=\ind{X}(h_1, h){\rm e}(-\phi(h_1, h)(x-z))\\
\times\prod_{\omega\in \{0, 1\}^{n-1}} {\mathcal C}^{|\omega|} F_{m; y_{(1, \omega, 0)}}^{\xi}(x+ (h_1, h)\cdot(1, \omega))
\overline{F_{m; y_{(1, \omega, 1)}}^{\xi}(z+ (h_1, h)\cdot(1, \omega))}
\end{gather*}
for ${\bm y}=(y_{(1, \omega, 0)}, y_{(1, \omega, 1)})_{(j,\omega)\in \{0, 1\}^{n}}\in \KK^{2^{n}}$ and  $x, z, h_1\in \KK$, $h\in \KK^{n-1}$.
Expanding the square and
changing variables $x\mapsto x-h_1$ and $z\mapsto z-h_1$ we may rewrite \eqref{eq:58} as 
\begin{gather}
\label{eq:59}
\nonumber \frac{1}{N_0^2} \int_{\KK^{2^{n}}}\int_{\KK^2}\int_{\KK^{n+1}} G_1(x, z, h_1, h_1', h;{\bm y})d\nu_{[H_1]}^{\otimes2}(h_1, h_1')
d{\bm \nu}_{n-1}(h)d\mu(x)d\mu(z)d\mu^{\otimes 2^{n}}_{[N]}({\bm y})\\
\gtrsim \delta^{O(1)},
\end{gather}
where
\begin{gather*}
G_1(x, z, h_1, h_1', h;{\bm y})
:=\ind{X}(h_1, h)\ind{X}(h_1', h) {\rm e}
(-(\phi(h_1, h)-\phi(h_1', h))(x-z))\\
\times\prod_{\omega\in \{0, 1\}^{n-1}} {\mathcal C}^{|\omega|}
\Delta_{h_1'-h_1}F_{m; y_{(1, \omega, 0)}}^{\xi}(x+ h\cdot\omega)
\overline{\Delta_{h_1'-h_1}F_{m; y_{(1, \omega, 1)}}^{\xi}(z+ h\cdot\omega)}.
\end{gather*}
Iteratively, for each $i\in\{2,\ldots, n\}$, we apply the Cauchy--Schwarz inequality in all but the $h_i$
variable to conclude that
\begin{align*}
\frac{1}{N_0^2}\int_{\KK^2}\int_{\KK^2}\int_{\square_n(X)} G_n(x, z, h, h';y, y')
d{\bm \nu}_n^{\otimes2}(h, h')d\mu(x)d\mu(z)d\mu^{\otimes 2}_{[N]}(y, y')
\gtrsim \delta^{O(1)},
\end{align*}
where
\begin{align*}
G_n(x, z, h, h';y, y'):={\Delta}_{h'-h}^nF_{m; y}^{\xi}(x)\overline{{\Delta}_{h'-h}^nF_{m; y'}^{\xi}(z)} 
{\rm e}(-\psi((h, h'))(x-z)).
\end{align*}
We have arrived at \eqref{eq:81},  completing  the proof of the lemma
\end{proof}

The following lemma is a slight variant of a result found in \cite{PEL}.

\begin{lemma}
\label{lem:3}
Given a scale $N\ge 1$, $0<\delta\le 1$, $m\in\Z_+$ with $m\ge 2$, $n\in\N$ and scales
$H_1,\ldots, H_{n+1}$ with each $H_i \le N$.  We assume for every $i\in\bra{n}$ that $\varphi_i:\KK^n\to \KK$  
is a measurable function independent of the variable $h_i$ in a vector $h=(h_1, \ldots, h_n)\in \KK^n$.  Let
$\mathcal P:=\{P_1,\ldots, P_{m}\}$ and
$\mathcal Q:=\{Q_1,\ldots, Q_n\}$ be collections of polynomials.  For $\xi\in \KK^n$ let
$F_m^{\xi}$ be the dual function defined in \eqref{eq:55} that
corresponds to the form \eqref{eq:51} and $1$-bounded functions
$f_0, f_1,\ldots, f_{m-1}\in L^0(\KK)$ supported on an interval $I\subset \KK$ of
measure $N_0 = N^{\deg{P_m}}$. Suppose that 
\begin{align}
\label{eq:82}
\int_{\KK^n}\Big|N^{-1}_0\reallywidehat{{\Delta}_h^n F_m^{\xi}}\Big(\sum_{i=1}^n\varphi_i(h)\Big)\Big|^2
d\Big(\bigotimes_{i=1}^{n}\nu_{[H_i]}\Big)(h)\ge \delta.
\end{align}
Then
\begin{align*}
\|F_m^{\xi}\|_{\square_{[H_1],\ldots,  [H_{n+1}]}^{n+1}(I)} \gtrsim_{\mathcal P} \delta^{O_{\mathcal P}(1)}.
\end{align*}
\end{lemma}

\begin{proof}
We shall write as before ${\bm \nu}_n:=\bigotimes_{i=1}^{n}\nu_{[H_i]}$
and also let ${\bm \mu}_n:=\bigotimes_{i=1}^n\mu_{[H_i]}$. Expanding the Fej{\'e}r kernel we may write
the left-hand side of \eqref{eq:82} as
\begin{gather*}
{\mathcal I}:=\int_{\KK^n}\Big|N_0^{-1}\reallywidehat{{\Delta}^n_hF_m^{\xi}}\Big(\sum_{i=1}^n\varphi_i(h)\Big)\Big|^2
d{\bm\nu}_n(h)\\
=
\int_{\KK^{2n}}\Big|N_0^{-1}\reallywidehat{{\Delta}_{h-h'}^nF_m^{\xi}}\Big(\sum_{i=1}^n\varphi_i(h-h')\Big)\Big|^2
d{\bm\mu}_n^{\otimes2}(h, h')\\
=\frac{1}{N^{2}_0}\int_{\KK^{2n+2}}{\Delta}_{h-h'}^nF_m^{\xi}(x)\overline{{\Delta}_{h-h'}^nF_m^{\xi}(z)}e\Big(-\sum_{i=1}^n\varphi_i(h-h')(x-z)\Big)d{\bm\mu}_n^{\otimes2}(h, h')d\mu(x)d\mu(z).
\end{gather*}
We apply the Cauchy--Schwarz inequality in the $x, z$ and $h'$
variables and Corollary \ref{GCS-indep} to deduce that 
\begin{align*}
{\mathcal I}^{2^n}&\le
\frac{1}{N_0^{2}}\int_{\KK^{3n+2}}  \prod_{\omega\in\{0, 1\}^n}\mathcal C^{|\omega|}\big({\Delta}_{h^{(\omega)}-h}^nF_m^{\xi}(x)
\overline{{\Delta}_{h^{(\omega)}-h}^nF_m^{\xi}(z)}\big)d{\bm\mu}_n^{\otimes3}(h^{(0)}, h^{(1)}, h)d\mu(x)d\mu(z)\\
=&  \frac{1}{N_0^2} \int_{\KK^{3n}}  
{\mathcal A}(x,z, h_n', h^{(0)}, h^{(1)}, h'){\mathcal B}(x,z,h^{(0)}, h^{(1)}, h') d\mu(x) d\mu(z) d{\bm\mu}_{n-1}^{\otimes3}(h^{(0)}, h^{(1)}, h')d\mu_{[H_n]}(h_n'),
\end{align*}
where 
\begin{align*}
{\mathcal A}(x,z,h_n', h^{(0)}, h^{(1)}, h') :=&
\int_{\KK^2} \prod_{\omega'\in\{0, 1\}^{n-1}}{\mathcal C}^{|\omega'|}
\Bigl[{\Delta}_{h^{(\omega')}-h'}^{n-1}  \Bigl( F_m^{\xi}(x + h_n^0 - h_n')\\
&\times{\overline{F_m^{\xi}(x+h_n^1 - h_n') F_m^{\xi}(z+ h_n^0 - h_n')}} F_m^{\xi}(z + h_n^1 - h_n') \Bigr)\Bigr]
d\mu_{[H_n]}^{\otimes2}(h_n^0, h_n^1),
\end{align*}
\begin{align*}
{\mathcal B}(x,z,h^{(0)}, h^{(1)}, h'):=\prod_{\omega'\in\{0, 1\}^{n-1}}  {\mathcal C}^{|\omega'|}
\Bigl[{\Delta}_{h^{(\omega')}-h'}^{n-1} \Bigl( |F_m^{\xi}(x)|^2 |F_m^{\xi}(z)|^2 \Bigr)\Bigr].
\end{align*}
Since ${\mathcal A}\ge 0$, we see that
\begin{align*}
{\mathcal I}^{2^n} \le &N_0^{-2} \int_{\KK^{3n}} {\mathcal A}(x,z, h_n', h^{(0)}, h^{(1)}, h') d\mu(x) d\mu(z) d{\bm\mu}_{n-1}^{\otimes3}(h^{(0)}, h^{(1)}, h')d\mu_{[H_n]}(h_n')\\
&=
\frac{1}{N_0^2} \int_{\KK^{3n+1}} 
\prod_{\omega'\in\{0, 1\}^{n-1}}{\mathcal C}^{|\omega'|}
\Bigl[{\Delta}_{h^{(\omega')}-h'}^{n-1} \Bigl( F_m^{\xi}(x + h_n^0)\overline{F_m^{\xi}(x+h_n^1)}\\
&\hspace{3cm} \times
 \overline{F_m^{\xi}(z+ h_n^0)} F_m^{\xi}(z + h_n^1) \Bigr)\Bigr] d\mu(x) d\mu(z)
 d{\bm\mu}_{n}^{\otimes2}(h^{(0)}, h^{(1)}) d{\bm\mu}_{n-1}(h')\\
 &= \frac{1}{N_0^2} \int_{\KK^{3n+1}} 
\mathcal C(x,z,h^{(0)}, h^{(1)}, h') d\mu(x) d\mu(z)
d{\bm\mu}_{n}^{\otimes2}(h^{(0)}, h^{(1)}) d{\bm\mu}_{n-1}(h'),
\end{align*}
where
\[
\mathcal C(x,z,h^{(0)}, h^{(1)}, h'):=\prod_{\omega'\in\{0, 1\}^{n-1}}{\mathcal C}^{|\omega'|}
\Bigl[{\Delta}_{h^{(\omega')}-h'}^{n-1} {\Delta}_{h_n^1 - h_n^0}\Bigl( F_m^{\xi}(x)
{\overline{F_m^{\xi}(z)}}\Bigr)\Bigr].
\]
In the penultimate equality we made the change of variables $x \to x -h_n^0+h_n'$ and $z \to z - h_n^0+ h_n'$.
Now proceeding inductively we see that
$$
{\mathcal I}^{2^n} \le \frac{1}{N_0^2}\int_{\KK^{2n+2}} {\Delta}_{h-h'}^n F_m^{\xi}(x) {\overline{{\Delta}^n_{h-h'} F_m^{\xi}(z)}}
d\mu(x) d\mu(z)
d{\bm\mu}_{n}^{\otimes2}(h, h').
$$
Inserting an extra average in the $x$ variable and using the pigeonhole principle to fix $z$, it follows that
\begin{align*}
{\mathcal I}^{2^n}\le \frac{1}{N_0}\int_{\KK^{2n+1}}
\overline{{\Delta}_{h-h'}^nF_m^{\xi}(z)}\int_{\KK}{\Delta}_{h-h'}^nF_m^{\xi}(x+w)d\mu_{[H_{n+1}]}(w)d{\bm\mu}_n^{\otimes2}(h,h')d\mu(x).
\end{align*}
To conclude we apply the Cauchy--Schwarz inequality to double the $w$ variable and so 
$$
\delta^{2^{n+1}} \le {\mathcal I}^{2^{n+1}} \le \frac{1}{N_0} \int_{\KK^{2n+3}} {\Delta}_{h-h'}^{n+1} F_m^{\xi}(x) d{\bm\mu}_{n+1}^{\otimes2}(h,h')d\mu(x)
= \|F_m^{\xi}\|_{\square_{[H_1],\ldots,  [H_{n+1}]}^{n+1}(I)}.
$$
This completes the proof of the lemma.
\end{proof}

\begin{proof}[Proof of Theorem \ref{thm:deg-low}]
The proof is by induction on $m\in \Z_+$.  The proof will consist of
several steps. We begin by establishing the following claim.
\begin{claim}
\label{claim:3}
Let $N\ge 1$ be a scale, $0<\delta\le 1$, $m\in\Z_+$ with $m\ge 2$ and $n\in\N$ be given. Let
$\mathcal P:=\{P_1,\ldots, P_{m}\}$ and
$\mathcal Q:=\{Q_1,\ldots, Q_n\}$ be collections of polynomials such
that 
$$
1\le \deg{P_1}<\ldots<\deg{P_{m}}<\deg{Q_1}<\ldots<\deg{Q_n}.
$$
For $\xi\in \KK^n$ let $F_m^{\xi}$ be the dual function defined in \eqref{eq:55} that corresponds to the form \eqref{eq:51} and 
 $1$-bounded functions $f_0, f_1,\ldots, f_{m-1}\in L^0(\KK)$ supported
on an interval $I\subset \KK$ of measure $N_0 := N^{\deg{P_m}}$.
Suppose that 
\begin{align}
\label{eq:60}
N^{-1}_0\big|\widehat{F_m^{\xi}}(\zeta)\big|  \ge \delta.
\end{align}
Then for any sufficiently large constant $C\gtrsim_{\mathcal P, \mathcal Q}1$ one has
\begin{align}
\label{eq:80}
|\zeta|\lesssim  \delta^{-C} N^{-\deg(P_m)},
\quad \text{ and } \quad 
|\xi_j|\lesssim \delta^{-C} N^{-\deg(Q_j)}
\quad \text{ for all } \quad j\in\bra{n}.
\end{align}

\end{claim}
The proof of Claim \ref{claim:3} for each integer $m\ge2$ is itself
part of the inductive proof of Theorem \ref{thm:deg-low}.  In the
first step we prove Claim \ref{claim:3} for $m=2$.  In the second step
we show that Claim \ref{claim:3} for all integers $m\ge2$ implies
Theorem \ref{thm:deg-low}, this in particular will establish Theorem
\ref{thm:deg-low} for $m=2$. In the third step we finally show that
Claim \ref{claim:3} for all integers $m\ge3$ follows from Claim
\ref{claim:3} and Theorem \ref{thm:deg-low} for $m-1$. Taken together,
this shows that Claim \ref{claim:3} and Theorem \ref{thm:deg-low} hold
for each integer $m\ge2$, completing the proof of Theorem \ref{thm:deg-low}.

\paragraph{\bf Step 1.}
We now prove Claim \ref{claim:3} for $m=2$. Here $N_0 = N^{\deg{P_2}}$.
For $\zeta_1, \zeta_2\in \KK$ and $\xi\in \KK^n$ we define the multiplier
\begin{align*}
m_N(\zeta_1, \zeta_2, \xi):=\int_{B_1(0)} {\rm e}\Big(-\zeta_1P_1(\alpha y)+\zeta_2 P_2(\alpha y)+\sum_{j=1}^n\xi_jQ_j(\alpha y)\Big)
d\mu(y)
\end{align*}
where $\alpha \in \KK$ satisfies $|\alpha| = N$.
By definitions \eqref{eq:55} and \eqref{eq:56} and making the change of variables $x\mapsto x-P_2(y)$ we may write
\begin{align*}
N_0^{-1}\widehat{F_2^{\xi}}(\zeta_2)&=N_0^{-1}\int_{\KK}\int_{\KK}F_{2;y}^{\xi}(x){\rm e}(-\zeta_2 x)d\mu_{[N]}(y)d\mu(x)\\
&=N_0^{-1}\int_{\KK}\widehat{f}_0(\zeta_2-\zeta_1)\widehat{f}_1(\zeta_1)m_N(\zeta_1, \zeta_2, \xi)d\zeta_1.
\end{align*}
By the Cauchy--Schwarz inequality and Plancherel's theorem we obtain
\begin{align*}
\delta\le N_0^{-1}|\widehat{F_2^{\xi}}(\zeta_2)|\lesssim  N_0^{-1}\|f_0\|_{L^2(\KK)}
\|f_1\|_{L^2(\KK)}\sup_{\zeta_1\in \KK}|m_N(\zeta_1, \zeta_2, \xi)|,
\end{align*}
which gives for some $\zeta_1\in \KK$ that
\begin{align*}
\delta\lesssim |m_N(\zeta_1, \zeta_2, \xi)|,
\end{align*}
since $\|f_0\|_{L^2(\KK)}, \|f_1\|_{L^2(\KK)}\lesssim N_0^{1/2}$. Applying Lemma \ref{lem:vdc-osc} with $\mathcal P=\{-P_1, P_2\}$ and $\mathcal Q=\{Q_1,\ldots, Q_n\}$ we deduce that for every sufficiently large $C\gtrsim 1$ one has
\begin{align*}
|\zeta_j|\lesssim \delta^{-C}N^{-\deg(P_j)}
\quad \text{ for all } \quad j\in\bra{2}, \quad\text{ and } \quad
|\xi_j|\lesssim \delta^{-C} N^{-\deg(Q_j)}
\quad \text{ for all } \quad j\in\bra{n}.
\end{align*}
This completes the proof of Claim \ref{claim:3} for $m=2$.
\paragraph{\bf Step 2.}
In this step we show that Claim \ref{claim:3} for all integers $m\ge2$
implies Theorem \ref{thm:deg-low}. In view of Step 1. this will in
particular establish Theorem \ref{thm:deg-low} for $m=2$, which is the
base case of our double induction.  As before we shall write
${\bm \nu}_{j}:=\bigotimes_{i=1}^{j}\nu_{[H_i]}$ for any
$j\in\Z_+$. Recall that
$N_0 = N^{\deg(P_m)}$
and note
\begin{align*}
\|F_m^{\xi}\|_{\square_{[H_1], \ldots, [H_s]}^s(I)}^{2^s}=
\int_{\KK^{s-2}}\|{\Delta}_h^{s-2}F_m^{\xi}\|_{\square_{[H_{s-1}], [H_s]}^2(I)}^4
d{\bm \nu}_{s-2}(h).
\end{align*}
By \eqref{eq:67} and the pigeonhole principle 
 there exists a measurable set 
$X\subseteq \prod_{i=1}^{s-2}[H_i]$ so that ${\bm\nu}_{s-2}(X)\gtrsim \delta^{O(1)}$, and
for all $h \in X$ one has
\begin{align*}
\|{\Delta}_h^{s-2}F_m^{\xi}\|_{\square_{[H_{s-1}], [H_s]}^2(I)}
\gtrsim\delta^{O(1)}.
\end{align*}
Here we used that $\supp F_m^{\xi}$ is a subset of an
interval whose measure is at most
$O(N_0)$. 
By Lemma \ref{2.10} we have
\begin{align*}
N_0^{-1}\bigl\|\reallywidehat{{\Delta}_h^{s-2}F_m^{\xi}}\bigr\|_{L^{\infty}(\KK)}  \gtrsim \delta^{O(1)}.
\end{align*}
Next we claim that there is a countable set ${\mathcal F} \subset \KK$, depending on $N$ and $\delta$ such that
\begin{align}\label{sup-F}
\sup_{\phi\in {\mathcal F}}N_0^{-1}\big|\reallywidehat{{\Delta}_h^{s-2}F_m^{\xi}}(\phi)\big|  \gtrsim \delta^{C_0}
\end{align}
for some absolute constant $C_0\in\Z_+$ and for all $h\in X$. When $\KK$ is non-archimedean, we take
$$
{\mathcal F} \ = \ \bigcup_{M\ge 1} \, \Bigl\{ z = \sum_{j= -M}^{L-1} z_j \pi^j \in \KK : z_j \in o_\KK/m_\KK\Bigr\}
$$
where $N_0 = q^L$. Let $x \in I = B_{N_0}(x_0)$. For any $\zeta \in \KK$, we have $\zeta \in B_{N_0^{-1}}(\zeta_0)$
for some $\zeta_0 \in {\mathcal F}$. Note that
$$
{\rm e}(-\zeta x) \ = \ {\rm e}(-x \zeta_0) \, {\rm e}(-(x-x_0)(\zeta - \zeta_0)) \, {\rm e}(-x_0(\zeta - \zeta_0)) \ = \
{\rm e}(-x \zeta_0) \, {\rm e}(-x_0(\zeta - \zeta_0))
$$
since $|(x-x_0)(\zeta - \zeta_0)| \le N_0 N_0^{-1} = 1$ and ${\rm e} = 1$ on $o_\KK$. Therefore
$|\reallywidehat{{\Delta}_h^{s-2}F_m^{\xi}}(\zeta)| = |\reallywidehat{{\Delta}_h^{s-2}F_m^{\xi}}(\zeta_0)|$
since ${\Delta}_h^{s-2}F_m$ is supported in $I$ whenever $h \in X$. This shows that \eqref{sup-F} holds for
non-archimedean fields.

When $\KK= {\mathbb R}$, we take ${\mathcal F} := T_0 {\mathbb Z}$, where 
\begin{align*}
T_0:=\delta^{C_0}\big(CN_0\big)^{-1}.
\end{align*}
for a sufficiently large constant $C\gtrsim1$. When $\KK = {\mathbb C}$, we take ${\mathcal F} := T_1 {\mathbb Z}^2$
where $T_1 := \delta^{C_0} (C \sqrt{N_0})^{-1}$. By the Lipschitz nature of characters on ${\mathbb R}$ or ${\mathbb C}$,
we again see that \eqref{sup-F} holds in the archimedean cases.
In particular, there exists a measurable function $\phi: X \to {\mathcal F}$ so
that
\begin{align}
\label{eq:64}
N^{-1}_0\big|\reallywidehat{{\Delta}_h^{s-2}F_m^{\xi}}(\phi(h))\big|\gtrsim \delta^{C_0}
\end{align}
for all $h\in X$. If necessary, we may additionally assume that the range of $\phi$ is finite.

By Lemma \ref{lem:1} it follows that
\begin{align*}
\int_{\square_{s-2}(X)}\bigg|N^{-1}_0\int_{\KK} F_m(x; h, h')e(-\psi((h, h'))x)d\mu(x)\bigg|^2
d{\bm \nu}_{s-2}^{\otimes2}(h, h')
\gtrsim \delta^{O(1)},
\end{align*}
where
\begin{gather*}
F_m(x; h, h'):=\int_{\KK}{\Delta}_{h'-h}^{s-2}f_0(x+P_m(y))\prod_{i=1}^{m-1}{\Delta}_{h'-h}^{s-2}f_{i}(x-P_i(y)+P_m(y))d\mu_{[N]}(y),\\
\psi(h, h'):=\sum_{\omega\in\{0, 1\}^{s-2}}(-1)^{|\omega|}\phi\Big(\omega\cdot h+(1-\omega)\cdot h')\Big).
\end{gather*}

Thus by the pigeonhole principle, there exists a measurable set $X_0\subseteq \square_{s-2}(X)$ with
${\bm \nu}_{s-2}^{\otimes2}(X_0)\gtrsim\delta^{O(1)}$ such that for every $(h, h')\in X_0$ one has
\begin{align*}
\bigg|N^{-\deg(P_m)}\int_{\KK} F_m(x; h, h')e(-\psi((h, h'))x)d\mu(x)\bigg|
\gtrsim \delta^{O(1)}.
\end{align*}
By  Claim \ref{claim:3} there is  a $c:=c_{m, s}\ge1$ such that for each $(h, h')\in X_0$, one has
\begin{align*}
|\psi((h, h'))|\lesssim_{m, s} \delta^{-c} N^{-\deg(P_m)}.
\end{align*}
By the pigeonhole principle there exists $h'\in \prod_{i=1}^{s-2}[H_i]$ and a measurable set 
\[
X_0(h'):=\big\{h\in X: (h, h')\in X_0\ \text{ and } \ |\psi((h, h'))|\lesssim \delta^{-c} N^{-\deg(P_m)}\big\}
\]
satisfying ${\bm \nu}_{s-2}(X_0(h'))\gtrsim \delta^{O(1)}$.
Since $\psi((h, h'))\in {\mathcal F}$ we see that
\begin{align*}
X_0(h')\subseteq\bigcup_{k\in {\mathcal \KK}}X_{0}^k(h')
\end{align*}
where ${\mathcal \KK} =  [{O(\delta^{-O(1)})}]\cap{\mathbb Z}$ when $\KK = {\mathbb R}$. In this case,
$X_{0}^k(h'):=\{h\in X: \psi((h, h'))=T_0k\}$. When $\KK = {\mathbb C}$, we have 
 ${\mathcal \KK} =  [{O(\delta^{-O(1)})}]\cap{\mathbb Z}^2$ and
$X_{0}^k(h'):=\{h\in X: \psi((h, h'))=T_1 k\}$. Finally when $\KK$ is non-archimedean, 
$$
{\mathcal \KK} =  \bigl[{O(\delta^{-O(1)})}\bigr] \cap \bigl\{ k = \sum_{j= -M}^{-1} k_j \pi^j \in \KK : k_j \in o_\KK/m_\KK\bigr\}
$$
and $X_{0}^k(h') := \{h\in X: \psi((h,h')) = \pi^L k \}$.

Thus by the pigeonhole principle  there is $k_0\in {\mathcal \KK}$ such that
${\bm \nu}_{s-2}(X_0^{k_0}(h'))\big)\gtrsim \delta^{O(1)}$.
When $\KK = {\mathbb R}$, this shows that
$\psi(h, h')=T_0 k_0 =: \phi_m$ for all $h\in X_0^{k_0}(h')$. When $\KK = {\mathbb C}$, we have
$\phi(h,h') = T_1 k_0$ for all $h \in X_0^{k_0}(h')$ and when $\KK$ is non-archimedean, $\phi(h,h') = \pi^L k_0$ for all 
$h \in X_0^{k_0}(h')$. We will denote these values by $\phi_m$ in all cases.

Set
\begin{align*}
\psi_1(h, h'):=(-1)^{s+1}\sum_{\substack{\omega\in\{0,1\}^{s-2} \\ \omega_1=0}}(-1)^{|\omega|}\phi\Big((\omega\cdot h+(1-\omega)\cdot h')\Big) + (-1)^s \phi_m
\end{align*}
and, for $i=\bra{s-2}\setminus\{1\}$, set
\begin{align*}
\psi_i(h, h'):=(-1)^{s+1}\sum_{\substack{\omega\in\{0,1\}^{s-2}\setminus\{0\} \\ \omega_1=\ldots=\omega_{i-1}=1 \\ \omega_{i}=0}}(-1)^{|\omega|}\phi\Big((\omega\cdot h+(1-\omega)\cdot h')\Big).
\end{align*}
Note that $\psi_i$ does not depend on $h_{i}$ and we can write
\begin{align*}
\phi(h)=\sum_{i=1}^{s-2}\psi_i(h, h').
\end{align*}
Averaging \eqref{eq:64} over $\mathbb{X}:=X_0^{k_0}(h')$ and using positivity, we obtain
\begin{gather*}
\int_{\KK^{s-2}}\Big|N_0^{-1}\reallywidehat{{\Delta}_h^{s-2}F_m^{\xi}}\Big(\sum_{i=1}^{s-2}\psi_i(h, h')\Big)\Big|^2
d{\bm\nu}_{s-2}(h)\\
\ge \int_{\mathbb{X}}\Big|N_0^{-1}\reallywidehat{{\Delta}_h^{s-2}F_m^{\xi}}(\phi(h))\Big|^2
d{\bm\nu}_{s-2}(h)\gtrsim \delta^{O(1)}.
\end{gather*}
Invoking Lemma \ref{lem:3} we conclude that
\begin{align*}
\|F_m^{\xi}\|_{\square_{[H_1], \ldots, [H_{s-1}]}^{s-1}(I)}\gtrsim \delta^{O(1)}.
\end{align*}

\paragraph{\bf Step 3.}
Gathering together the conclusions of Step 1. and Step 2. (for $m=2$), we see
that the base step of a double induction has been established. In this step we
shall illustrate how to establish the inductive step. We assume that
Claim \ref{claim:3} and Theorem \ref{thm:deg-low} hold for $m-1$ in
place $m$ for some integer $m\ge3$. Then we will prove that Claim
\ref{claim:3} holds for $m\ge3$, which in view of Step 2. will allow
us to deduce that Theorem \ref{thm:deg-low} also holds for $m\ge3$.
This will complete the proof of Theorem \ref{thm:deg-low}.

Recall that $N_0 = N^{\deg(P_m)}$. By definitions \eqref{eq:55} and
\eqref{eq:56} and making the change of variables $x\mapsto x-P_m(y)$
we may write
\begin{align*}
N_0^{-1}\widehat{F_m^{\xi}}(\zeta_m)&=N_0^{-1}\int_{\KK^2}F_{m;y}^{\xi}(x){\rm e}(-\zeta_m x)d\mu(x)d\mu_{[N]}(y)\\
=N_0^{-1}&\int_{\KK^2}{\rm M}_{\zeta_m}f_0(x)\prod_{i=1}^{m-1}f_{i}(x-P_i(y)){\rm e}\Big(\zeta_m P_m(y)+\sum_{j=1}^n\xi_jQ_j(y)\Big)d\mu_{[N]}(y)d\mu(x)\\
&=: M^{-1} \Lambda_{\mathcal P'; N}^{\mathcal Q', \xi'}({\rm M}_{\zeta_m}f_0, f_1,\ldots f_{m-1}),
\end{align*}
where ${\rm M}_{\zeta_m}f_0(x):={\rm e}(-\zeta_m x)f_0(z)$, $\mathcal P':=\mathcal P\setminus\{P_m\}$, $\mathcal Q':=\mathcal Q\cup\{P_m\}$, $\xi':=(\zeta_m, \xi_1,\ldots, \xi_n)\in \KK^{n+1}$ and $M = N_0 {N_0'}^{-1}$ where 
$N_0'$ is the scale $N^{\deg(P_{m-1})}$.

Thus \eqref{eq:60} implies
\begin{align*}
M^{-1}|\Lambda_{\mathcal P'; N}^{\mathcal Q', \xi'}({\rm M}_{\zeta_m}f_0, f_1,\ldots, f_{m-1})|\gtrsim \delta^{O(1)}.
\end{align*}
As in the proof of Theorem \ref{thm:inverse}, 
by the pigeonhole principle, we can find an interval
$I'\subset \KK$ of measure about $ N_0'$ 
such that
\begin{align*}
|\Lambda_{\mathcal P'; N, I'}^{\mathcal Q', \xi'}(f_0', f_1',\ldots, f_{m-1}')|\gtrsim \delta^{O(1)},
\end{align*}
where 
$f_0':={\rm M}_{\zeta_m}f_0\ind{I'}, f_1':=f_1\ind{I'},\ldots, f_{m-1}':=f_{m-1}\ind{I'}$.

Consequently, by Proposition \ref{prop:dual}, there exists an $s\in \Z_+$ such that
\begin{align*}
\|F_{m-1}^{\xi'}\|_{\square_{[H_1], \ldots, [H_s]}^s(N_0')}
\gtrsim \delta^{O(1)},
\end{align*}
where $F_{m-1}^{\xi'}$ is the dual function respect the form
$\Lambda_{\mathcal P'; N, I'}^{\mathcal Q', \xi'}(f_0', f_1',\ldots, f_{m-1}')$ and
$H_i\simeq \delta^{O_{\mathcal P'}(1)}N^{\deg(P_{m-1})}$ for $i\in\bra{s}$. By the induction hypothesis (for Theorem \ref{thm:deg-low})
we deduce that
\begin{align*}
\|F_{m-1}^{\xi'}\|_{\square_{[H_1], [H_2]}^2(N_0')}
\gtrsim \delta^{O(1)},
\end{align*}
which in turn by Lemma \ref{2.10} implies
\begin{align*}
(N'_0)^{-1}\big|\widehat{F_{m-1}^{\xi'}}(\zeta_{m-1})\big|  \gtrsim \delta^{O(1)}
\end{align*}
for some $\zeta_{m-1}\in \KK$. By the induction hypothesis (for Claim \ref{claim:3}) we deduce that
\begin{align*}
&|\zeta_j|\lesssim \delta^{-C} N^{-\deg(P_j)}
\qquad \text{ for all } \quad j\in \bra{m}\setminus\bra{m-2}, \quad\text{ and }\\
&|\xi_j|\lesssim \delta^{-C} N^{-\deg(Q_j)}
\qquad \text{ for all } \quad j\in\bra{n},
\end{align*}
which in particular implies \eqref{eq:80} and we are done.
\end{proof}

\section{Sobolev estimates}
\label{sec:sobolev}
As a consequence of the $L^\infty$-inverse theorem from the previous
section we establish some Sobolev estimates, which will be critical in
the proof of Theorem \ref{thm:main}.

We begin with a smooth variant of Theorem \ref{thm:inverse}. When $\KK$ is archimedean,
we fix a  Schwartz function $\varphi$ on $\KK$ so that
\begin{align*}
\ind{[1]}(\xi)\le \widehat{\varphi}(\xi)\le \ind{[2]}(\xi), \qquad \xi\in \KK.
\end{align*}
When $\KK={\mathbb R}$, we set $\varphi_N(x) = N^{-1} \varphi(N^{-1} x)$ for any
$N>0$ and when $\KK = {\mathbb C}$, we set
$\varphi_{N}(z) = N^{-1} \varphi(N^{-1/2} z)$ for any $N > 0$. When
$\KK$ is non-archimedean, we set $\varphi(x) = \ind{B_1(0)}(x)$ so that $\widehat{\varphi}(\xi) = \ind{B_1(0)}(\xi)$ and 
we set $\varphi_{N}(x) = N^{-1} \ind{[N]}(x)$ for any scale $N$.  

\begin{theorem}[A smooth variant of the inverse theorem]
\label{thm:inverse-s}
Let $N\ge 1$ be a scale, $0<\delta\le 1$, $m\in\Z_+$ be given. Let $\mathcal P:=\{P_1,\ldots, P_m\}$
be a collection of polynomials such that
$1\le \deg{P_1}<\ldots<\deg{P_m}$.  
Let $f_0, f_1,\ldots, f_m\in L^0(\KK)$ be $1$-bounded functions
supported on an interval $I\subset \KK$ of measure $N_0 = N^{\deg{P_m}}$.
Suppose that the $(m+1)$-linear form defined in
\eqref{eq:6} satisfies
\begin{align}
\label{eq:63}
|\Lambda_{\mathcal P; N}(f_0,\ldots, f_m)|\ge\delta.
\end{align}
Then for any $j\in\bra{m}$ there exists an  absolute constant $C_j\gtrsim_{\mathcal P}1$ so that
\begin{align}
\label{eq:19-s}
N_0^{-1}\big\| \varphi_{N_j}*f_j\big\|_{L^1(\KK)} \gtrsim_{\mathcal P} \delta^{O_{\mathcal P}(1)},
\end{align}
where $N_j\simeq \delta^{C_j}N^{\deg(P_j)}$, provided $N \gtrsim \delta^{-O_{\mathcal P}(1)}$.
\end{theorem}
\begin{proof}
By translation invariance we can assume that $f_j$ is supported on $[N_0]$ for every $j\in\bra{m}$.
The proof will consist of two steps. In the
first step we will invoke Theorem \ref{thm:inverse} to prove
\eqref{eq:19-s} for $j=1$. In the second step we will use \eqref{eq:19-s}
for $j=1$ to establish \eqref{eq:19-s} for $j=2$, and continuing
inductively we will obtain \eqref{eq:19-s} for all $j\in\bra{m}$.
\paragraph{{\bf Step 1.}} We first  establish \eqref{eq:19-s} for $j=1$. When $\KK$ is non-archimedean,
this is an immediate consequence of Theorem \ref{thm:inverse} since $\varphi_{N_1} = \mu_{[N_1]}$
in this case. Nevertheless we make the observation that
\begin{align}\label{var-form}
 |\Lambda_{\mathcal P; N}(f_0,\varphi_{N_1}*f_1, \ldots, f_m)|\gtrsim\delta
\end{align}
holds. In fact we will see that \eqref{var-form} holds for any $\KK$, non-archimedean or archimedean. First let us
see \eqref{var-form} when $\KK$ is non-archimedean. Suppose that $|\Lambda_{\mathcal P; N}(f_0,\varphi_{N_1}*f_1, \ldots, f_m)| \le c \, \delta$
for some small $c>0$. Then, since
$$
\delta \le |\Lambda_{\mathcal P; N}(f_0,f_1, \ldots, f_m)| \le
|\Lambda_{\mathcal P; N}(f_0,\varphi_{N_1}*f_1, \ldots, f_m)| +
|\Lambda_{\mathcal P; N}(f_0,f_1 -\varphi_{N_1}*f_1, \ldots, f_m)|,
$$
we conclude that $|\Lambda_{\mathcal P; N}(f_0,f_1 - \varphi_{N_1}*f_1, \ldots, f_m)|\gtrsim\delta$. Therefore Theorem \ref{thm:inverse}
implies that $N_0^{-1} \|\varphi_{N_1} *(f_1 - \varphi_{N_1} * f_1)\|_{L^1(\KK)}| \gtrsim \delta^{O(1)}$ but this is a
contradiction since $\varphi_{N_1} * \varphi_{N_1} = \varphi_{N_1}$ when $\KK$ is non-archimedean (in which case 
$\varphi_{N_1} = N_1^{-1} \ind{[N_1]}$) and so $\varphi_{N_1} * (f_1 - \varphi_{N_1}* f_1) \equiv 0$. 

We now turn to establish \eqref{eq:19-s} for $j=1$ when $\KK$ is archimedean (when $\KK = {\mathbb R}$ or $\KK = {\mathbb C}$).
Let $\eta:\KK \to [0, \infty)$ be  a Schwartz function so that $\int_{\KK}\eta=1$, ${\widehat{\eta}}\equiv 1$ near 0 and
${\rm supp}\:\widehat{\eta}\subseteq [2]$.
For $t>0$, we write $\eta_{t}(x):=t^{-1}\eta(t^{-1}x)$ when $\KK={\mathbb R}$ and $\eta_t(x):= t^{-2}\eta(t^{-1}x)$
when $\KK = {\mathbb C}$.  We will also need a Schwartz function
$\rho:\KK\to [0, \infty)$ such that
\begin{align*}
\ind{[1]\setminus [1-\delta^M]}(x)\le \rho(x)\le \ind{[1]}(x), \qquad x\in \KK
\end{align*}
for some large absolute constant $M\ge1$, which will be specified
later.  We shall also write $\rho_{(t)}(x):=\rho(t^{-1}x)$ for $t>0$ and $x\in \KK$.

Let $N_0'\simeq N_0$ when $\KK = {\mathbb R}$ and $N_0' \simeq \sqrt{N_0}$ when $\KK = {\mathbb C}$. 
Observe that \eqref{eq:63} implies that at least one of the following lower bounds holds:
\begin{gather}
\label{eq:86}
|\Lambda_{\mathcal P; N}(f_0,\varphi_{N_1}*f_1, \ldots, f_m)|\gtrsim\delta,\\
\label{eq:83}
|\Lambda_{\mathcal P; N}(f_0,\rho_{(N_0')}(f_1-\varphi_{N_1}*f_1), \ldots, f_m)|\gtrsim\delta,\\
\label{eq:84}
|\Lambda_{\mathcal P; N}(f_0,(1-\rho_{(N_0')})(f_1-\varphi_{N_1}*f_1), \ldots, f_m)|\gtrsim\delta.
\end{gather}
By Theorem \ref{thm:inverse} it is easy to see that \eqref{eq:86} yields
that
\begin{align*}
N_0^{-1}\big\| \varphi_{N_1}*f_1\big\|_{L^1(\RR)} \gtrsim \delta,
\end{align*}
which in turn will imply \eqref{eq:19-s} for $j=1$ provided that the remaining two alternatives \eqref{eq:83} and \eqref{eq:84} do not hold.  
If this is the case, then \eqref{var-form} also holds when $\KK={\mathbb R}, {\mathbb C}$ is archimedean. 

If the second alternative holds we let $f_1':= \rho_{(N_0')}(f_1-\varphi_{N_1}*f_1)$ and then Theorem \ref{thm:inverse} implies that 
\begin{align*}
N_0^{-1}\big\| \mu_{[N_1']}*f_1'\big\|_{L^1(\KK)} \gtrsim_{\mathcal P} \delta^{C_0'},
\end{align*}
with $N_1'\simeq \delta^{C_1'}N^{\deg(P_1)}$. By the Cauchy--Schwarz inequality (the support of $\mu_{[N_1']} *f_1'$ is contained
in a fixed dilate of $[N_0]$), we have
\begin{align*}
N_0^{-1}\big\| \mu_{[N_1']}*f_1'\big\|_{L^2(\KK)}^2 \gtrsim_{\mathcal P} \delta^{2C_0'}.
\end{align*}
Let $N_1'':= \delta^{A+C_1'}N^{\deg(P_1)}/A$ for some $A\ge1$ to be determined later. We now show that
\begin{align}\label{L1-norm}
\big\| \mu_{[N_1']}-\mu_{[N_1']}*\eta_{N_1''}\big\|_{L^1(\KK)}^2
\lesssim \sqrt{N_1'' /N_1'}
\lesssim \sqrt{\delta^{A}/A}.
\end{align}
We note that for $|x| \ge C N_1'$,
\begin{align*}
|\ind{[N_1']}(x)-\ind{[N_1']}*\eta_{N_1''}(x)|=\Big|\int_{\KK}\ind{[N_1']}(x-y))\eta_{N_1''}(y)d\mu(y)\Big|
\end{align*}
and so
\begin{align*}
\int_{|x|\ge CN_1'} |\ind{[N_1']}(x)-\ind{[N_1']}*\eta_{N_1''}(x)| d\mu(x) \lesssim N_1''.
\end{align*}
When $|x| \le C N_1'$ is small, we use the Cauchy--Schwarz inequality
\begin{align*}
\int_{|x|\le CN_1'} |\ind{[N_1']}(x)-\ind{[N_1']}*\eta_{N_1''}(x)| d\mu(x) \lesssim \sqrt{N_1'} \, 
\|\ind{[N_1']} * (\delta_0 - \eta_{N_1''})\|_{L^2(\KK)} 
\end{align*}
and then Plancherel's theorem,
\begin{align*}
\|\ind{[N_1']} * (\delta_0 - \eta_{N_1''})\|_{L^2(\KK)}^2 = \int_\KK |1 - {\widehat{\eta_{N_1''}}}(\xi)|^2 |\widehat{\ind{[N_1']}}(\xi)|^2 d\mu(\xi)
\lesssim \sqrt{N_1' N_1''}.
\end{align*}
Here we use the facts that $\widehat{\eta} \equiv 1$ near $0$ and the Fourier decay bound for euclidean balls,
$$
|\widehat{\ind{[N_1']}}(\xi)|^2 \lesssim |\xi|^{-2} \ {\rm when} \ \KK = {\mathbb R} \ \ {\rm and} \ \ 
|\widehat{\ind{[N_1']}}(\xi)|^2 \lesssim \sqrt{N_1'} |\xi|^{-3} \ {\rm when} \ \KK = {\mathbb C}.
$$
This establishes \eqref{L1-norm} and so
\begin{align*}
N_0^{-1}\big\| (\mu_{[N_1']}-\mu_{[N_1']}*\eta_{N_1''})*f_1'\big\|_{L^2(\KK)}^2&\lesssim
\big\| \mu_{[N_1']}-\mu_{[N_1']}*\eta_{N_1''}\big\|_{L^1(\KK)}^2\\
&\lesssim \sqrt{N_1'' /N_1'}
\lesssim \sqrt{\delta^{A}/A}.
\end{align*}
Consequently
\begin{align*}
\delta^{2C_0'}\lesssim_{\mathcal P}N_0^{-1}\big\| \mu_{[N_1']}*f_1'\big\|_{L^2(\KK)}^2\lesssim
N_0^{-1}\big\| \mu_{[N_1']}*\eta_{N_1''}*f_1'\big\|_{L^2(\KK)}^2+ \sqrt{\delta^{A}/A},
\end{align*}
which for sufficiently large $A\ge C_0'$ yields
\begin{align*}
N_0^{-1}\big\| \eta_{N_1''}*f_1'\big\|_{L^2(\KK)}^2\gtrsim_{\mathcal P} \delta^{2C_0'}.
\end{align*}
Taking $N_1:=\frac{1}{2}N_1''$ and using support properties of
$\widehat{\varphi}$ and $\widehat{\eta}$, by the Plancherel theorem we
may write (when $\KK = {\mathbb R}$)
\begin{gather*}
N_0^{-1}\big\| \eta_{N_1''}*f_1'\big\|_{L^2({\mathbb R})}^2
=N_0^{-1}\big\| \widehat{\eta}_{N_1''}\big(\widehat{\rho_{(N_0')}}*((1-\widehat{\varphi}_{N_1})\widehat{f}_1)\big)\big\|_{L^2(\RR)}^2\\
\lesssim N_0^{-1}\int_{\RR}\bigg(\int_{\RR}\frac{N_0'}{(1+N_0'|\xi-\zeta|)^{200}}|\widehat{f}_1(\zeta)(1-\widehat{\varphi}(N_1\zeta))||\widehat{\eta}(N_1''\xi)|\bigg)^2d\mu(\xi)\\
\lesssim N_0^{-1}\delta^{100(A+C_1')}\|f_1\|_{L^2(\RR)}^2.
\end{gather*}
A similar bound holds when $\KK = {\mathbb C}$. Therefore 
\begin{align*}
\delta^{2C_0'}\lesssim_{\mathcal P} N_0^{-1}\big\| \eta_{N_1''}*f_1'\big\|_{L^2(\KK)}^2\lesssim \delta^{100(A+C_1')},
\end{align*}
which is impossible if $A\ge1$ is large enough. Thus the second
alternative \eqref{eq:83} is impossible. To see that the third
alternative \eqref{eq:84} is also impossible observe that
\begin{align*}
\delta\lesssim
|\Lambda_{\mathcal P; N}(f_0,(1-\rho_{(N_0')})(f_1-\varphi_{N_1}*f_1), \ldots, f_m)|\lesssim
N_0^{-1}\int_{[N_0']}(1-\rho_{(N_0')})(x)d\mu(x)\lesssim \delta^M,
\end{align*}
which is also impossible if $M\ge1$ is sufficiently large. Hence \eqref{eq:86} must necessarily hold and
we are done. 

\paragraph{{\bf Step 2.}} 
Let $M\ge1$ be a large constant to be determined later, and define
$N'\simeq \delta^MN$ and $N_0' \simeq \delta^M N_0$.  The main idea is to partition the intervals $[N]$ and $[N_0]$ into
$\KK\simeq\delta^{-M}$ disjoint intervals of measure $\simeq N'$ and $\simeq N_0'$, respectively. Such partitions are straightforward
when $\KK = {\mathbb R}$. When $\KK$ is non-archimedean, we only need to partition $[N]$ and not $[N_0]$. Finally when
$\KK = {\mathbb C}$, intervals are discs and it is not possible to partition a disc into subdiscs and so we will need to be careful
with this technical issue.

We first concentrate
on the case when $\KK$ is non-archimedean. In this case, we only need to partition $[N]$ and not $[N_0]$. Such a partition was given in the proof of Theorem \ref{thm:inversem}.
In fact, choosing $\ell \gg 1$ such that $q^{-\ell} \simeq \delta^M$ and setting $N = q^n$ so that $N' = q^{n-\ell}$, we have
\begin{align*}
[N]  \  = \  B_{q^{n}}(0) \ = \ \bigcup_{y \in {\mathcal F}} B_{q^{n -\ell}}(y),
\end{align*}
which gives a partition of $[N]$ where ${\mathcal F} = \{ y = \sum_{j=0}^{\ell-1} y_j \pi^{-n + j} : y_j \in o_\KK/m_\KK \}$. 
Note $\# {\mathcal F} = q^{\ell}$ so that $\# {\mathcal F} \simeq \delta^{-M}$. 
Hence $\Lambda_{\mathcal P; N}(f_0,\varphi_{N_1}*f_1, \ldots, f_m) =$
$$
\frac{1}{N_0 N} \sum_{y_0 \in {\mathcal F}} \int_{\KK} \int_{B_{q^{n-\ell}}(y_0)}
f_0(x) \varphi_{N_1} * f_1 (x - P_1(y)) \prod_{i=2}^m f_i(x - P_i(y)) d\mu(y) d\mu(x).
$$
We observe that $\varphi_{N_1} * f_1 (x - P_1(y)) = \varphi_{N_1}* f_1(x - P_1(y_0))$ for any $y \in B_{q^{n-\ell}}(y_0)$
by the non-archimedean nature of $\KK$, if $M$ is chosen large enough depending on $P_1$. Hence, by the pigeonhole principle, we can find a
$y_0 \in {\mathcal F}$ such that
$$
\Bigl| \frac{1}{N_0 N'} \int_{\KK} \int_{B_{q^{n-\ell}}(y_0)}
f_0(x) \varphi_{N_1} * f_1 (x - P_1(y_0)) \prod_{i=2}^m f_i(x - P_i(y)) d\mu(y) d\mu(x) \Bigr| \gtrsim \delta.
$$
Changing variables $y \to y_0 + y$ allows us to write the above as 
$$
|\Lambda_{{\mathcal P}', N'}(f_0', f_2',\ldots, f_m')| \ \gtrsim \ \delta \ \ \ {\rm where} 
$$
$$
\Lambda_{{\mathcal P}', N'}(f_0', f_2',\ldots, f_m') \ = \frac{1}{N_0} \int_{\KK^2} f_0'(x) \prod_{j=2}^m f_j'(x - P_j'(y)) d\mu_{[N']}(y) d\mu(x),
$$
with $P_j'(y) = \ P_j(y_0 + y) - P_j(y_0)$, $f_0'(x) \ = \ f_0(x) \varphi_{N_1}*f_1(x - P_1(y_0))$ and 
$f_j'(x) = f_j (x + P_j(y_0))$.
Note that each $f_j'$ is supported in a fix dilate of $I$. In order to apply Theorem \ref{thm:inverse}, we require
$N' \simeq \delta^M N \ge 1$ and here is where the condition $N \gtrsim \delta^{-O_{\mathcal P}(1)}$ is needed.
Therefore Theorem \ref{thm:inverse} implies that
$$
N_0^{-1} \|\mu_{[N_2]} * f_2 \|_{L^1(\KK)} \ = \ N_0^{-1} \|\mu_{[N_2]} * f_2' \|_{L^1(\KK)} \ \gtrsim \ \delta^{O(1)}.
$$ 
The equality of $L^1$ norms follows from the change of variables $x \to x + P_2(y_0)$.
This completes the proof of \eqref{eq:19} for $j=2$ when $\KK$ is non-archimedean since $\mu_{[N_2]} = \varphi_{N_2}$.

We now turn to the archimedian case, when $\KK = {\mathbb R}$ or when $\KK={\mathbb C}$. Here we argue
as in Step 1. and establish the version of \eqref{var-form} for the function $f_2$. More precisely, writing 
$$
\Lambda_{\mathcal P; N}(f_0, \ldots, f_m) =
\Lambda_{\mathcal P;N}(f_0, f_1, \varphi_{N_2} * f_2, \ldots, f_m) + \Lambda_{\mathcal P;N}(f_0, f_1, f_2 - \varphi_{N_2}*f_2, \ldots, f_m),
$$
the argument in Step 1. shows that \eqref{eq:63} implies  
\begin{align}\label{var-form-2}
|\Lambda_{\mathcal P;N}(f_0, f_1, \varphi_{N_2} * f_2, \ldots, f_m)| \ \gtrsim \ \delta.
\end{align}
This inequality allows us to reduce matters to showing that \eqref{eq:63} implies 
$N_0^{-1}\|\mu_{[N_2]}*f_2\|_{L^1(\KK)} \gtrsim \delta^{O(1)}$ since then \eqref{var-form-2} would
imply 
$$
\delta^{O(1)} \lesssim N_0^{-1} \|\mu_{[N_2]}*\varphi_{N_2}*f_2\|_{L^1(\KK)} \le N_0^{-1} \|\varphi_{N_2}*f_2\|_{L^1(\KK)},
$$
establishing \eqref{eq:19-s} for $j=2$.

We give the details when $\KK = {\mathbb C}$ since there are additional
technical difficulties alluded to above. The case ${\mathbb R}$ is easier. Given a large, general interval ${\mathcal I}$ in
${\mathbb C}$ (that is, ${\mathcal I}$ is a disc with large radius $R$), we can clearly find a mesh of $K\simeq \delta^{-M}$ disjoint squares
$(S_k)_{k\in\bra{K}}$ of side length $\delta^{M/2} R$ which sit inside ${\mathcal I}$ such that 
$\mu({\mathcal I}\setminus \bigcup_{k\in\bra{K}} S_k) \lesssim \delta^2 R^2$. We fix such a mesh of squares $(S_k)_{k\in\bra{K}}$ for $[N]$
and a mesh of squares $(T_j)_{j\in\bra{J}}$ for $[N_0]$ so that
$$
\Lambda_{\mathcal P, N}(f_0, \varphi_{N_1} * f_1, \ldots, f_m) \ = 
$$
$$
\frac{1}{N_0 N}
\sum_{j\in\bra{J}} \sum_{k\in\bra{K}} \int_{T_j} \int_{S_k} f_0(x) f_0(x) \varphi_{N_1}*f_1(x - P_1(y)) \prod_{i=2}^m f_i(x - P_i(y)) d\mu(x) d\mu(y)
\ + \ O(\delta^2).
$$
Since $|\Lambda_{\mathcal P; N}(f_0, \varphi_{N_1}*f_1, \ldots, f_m)| \gtrsim \delta$ by \eqref{var-form} and since
the number of terms in each sum above is about $\delta^{-M}$, the pigeonhole principle gives us a square $T_0$ in $[N_0]$
and a square $S_0$ in $[N]$ such that 
$$
\Bigl| \frac{1}{N_0' N'} \int_{T_0} \int_{S_0} f_0(x) f_0(x) \varphi_{N_1}*f_1(x - P_1(y)) \prod_{i=2}^m f_i(x - P_i(y)) d\mu(x) d\mu(y) \Bigr|
\gtrsim \delta.
$$
Write $[N']_{sq} = \{ z \in {\mathbb C}: |z |_{\infty} \le \sqrt{N'}\}$ where $|z|_{\infty} = \max(|x|, |y|)$
for $z = x+iy$. Hence $S_0 = y_0 + [N']_{sq}$ for some $y_0 \in [N]$. For $z \in S_0$, we have $z = y_0 + y$ for some
$y\in [N']_{sq}$ and so by the
mean value theorem and the 1-boundedness of $f_1$,
\begin{align*}
&|\varphi_{N_1}* f_1 (x - P_1(y_0 + y)) - \varphi_{N_1}*f_1(x - P_1(y_0))| \\
&\le \ \sqrt{\frac{(N')^{{\rm Deg}P_1}}{N_1}} 
\int_{{\mathbb C}} \|(\nabla \varphi)_{N_1}(z)\| d\mu(z) \ \lesssim_{\varphi} \ \delta^{(M{\rm Deg}P_1 - C_1)/2}
\end{align*}
where $N_1 = \delta^{C_1} N^{{\rm Deg}P_1}$. Ensuring that $M \deg{P_1} - C_1 \ge 4$, we see that 
$$
 \Bigl| \frac{1}{N_0' N'} \int_{T_0} \int_{[N']_{sq}} f_0(x)  \varphi_{N_1}*f_1(x - P_1(y_0)) \prod_{i=2}^m f_i^t(x - P_i'(y)) d\mu(x) d\mu(y) \Bigr|
\gtrsim \delta,
$$
where $P_i'(y) = P_i(y_0 + y) - P_i(y_0)$ and $f_i^t(x) = f_i(x + P_i(y_0))$. For an appropriate interval
$I'$ containing $T_0$ with measure $\simeq N_0'$, we can write the above inequality as
$|\Lambda_{{\mathcal P}'; N'}(f_0', f_2', \ldots, f_m')| \gtrsim \delta$ where ${\mathcal P}' = \{P_2',\ldots, P_m'\}$,
$f_0'(x) = f_0(x) \varphi_{N_1}*f_1(x-P_1(y_0)) \ind{T_0}(x)$ and  $f_i'(x) = f_i^t(x) \ind{I'}(x)$ for $i\in\bra{m}\setminus\bra{1}$.
Here
$$
\Lambda_{{\mathcal P}'; N'}(f_0',\ldots, f_m') = \frac{1}{N_0'} \iint_{{\mathbb C}^2} f_0'(x) \prod_{i=2}^m f_i'(x-P_i'(y)) d\mu_{[N']_{sq}}(y)
d\mu(x).
$$
Again, in order to apply Theorem \ref{thm:inverse},  we need $N' = \delta^M N \ge 1$ which holds
provided $N \gtrsim \delta^{-O_{\mathcal P}(1)}$. Therefore  
by Theorem \ref{thm:inverse} (see the remark following the statement of Theorem \ref{thm:inverse}), we conclude that
\begin{align*}
(N_0')^{-1}\big\| \mu_{[N_2]_{sq}}*f_2'\big\|_{L^1({\mathbb C})} \gtrsim_{\mathcal P} \delta^{O(1)}
\end{align*}
for some $N_2\simeq \delta^{C_2+M\deg(P_2)}N^{\deg(P_2)}$.
The function $\mu_{[N_2]}*f_2'$ is supported on an interval
$I''\supseteq I'$ such that $\mu(I''\setminus I') \lesssim N_2$. Furthermore we can find an
interval $I'''\subseteq I'$  so that $\mu(I'\setminus I''') \lesssim N_2$ and for 
$x\in I'''$, we have $\ind{I'}(x-u) = 1$ for all $u \in [N_2]_{sq}$. Hence
$$
\delta^{O(1)} \lesssim \frac{1}{N_0'} \int_{I'''} \Bigl| \int_{\mathbb C} f_2(x+P_2(y_0) - u) d\mu_{[N_2]_{sq}}(u) \Bigr| d\mu(x)
 \ + \ O(N_2 (N_0')^{-1})
$$ 
where $N_2/N_0' \lesssim \delta^{M(\deg{P_2} - 1)}$ and $\deg{P_2} -1 \ge 1$. Hence, for $M\gg 1$ sufficiently
large, we conclude that
\begin{align}\label{sq-ineq}
\delta^{O(1)} \lesssim \frac{1}{N_0'} \int_{I'''} \Bigl| \int_{\mathbb C} f_2(x+P_2(y_0) - u) d\mu_{[N_2]_{sq}}(u) \Bigr| d\mu(x)
\lesssim N_0^{-1} \|\mu_{[N_2]_{sq}} * f_2 \|_{L^1({\mathbb C})}.
\end{align}
In the final inequality, we promoted the integration in $x$ to all of ${\mathbb C}$ and changed variables $x \to x + P_2(y_0)$. Hence we have shown that \eqref{eq:63} implies $N_0^{-1}\|\mu_{[N_2]_{sq}}*f_2\|_{L^1({\mathbb C})} \gtrsim \delta^{O(1)}$.
Since \eqref{eq:63} holds with $f_2$ replaced by $\varphi_{N_2}*f_2$ (this is \eqref{var-form-2}), we see that
$$
\delta^{O(1)} \lesssim N_0^{-1} \|\mu_{[N_2]_{sq}}*\varphi_{N_2}*f_2\|_{L^1({\mathbb C})} \le N_0^{-1}
\|\varphi_{N_2}*f_2\|_{L^1({\mathbb C})},
$$
establishing \eqref{eq:19-s} for $j=2$. Now we can proceed inductively and obtain \eqref{eq:19-s} for all $j\in\bra{m}$.
\end{proof}

\subsection{Multilinear functions and their duals} Recall the multilinear form
$$
\Lambda_{\mathcal P; N}(f_0, f_1, \ldots, f_m) = \frac{1}{N_0} \iint_{\KK^2} f_0(x) \prod_{i=1}^m f_i(x - P_i(y)) d\mu_{[N]}(y) d\mu(x).
$$
We define the multilinear function
\begin{align*}
A_N^{\mathcal P}(f_1,\ldots, f_m)(x):=\int_{\KK}\prod_{i=1}^mf_{i}(x-P_i(y))d\mu_{[N]}(y)
\end{align*}
so that $\Lambda_{\mathcal P; N}$
can be written as a pairing of $A_N^{\mathcal P}$ with $f_0$,
\begin{align*}
\langle A_N^{\mathcal P}(f_1,\ldots, f_m), f_0\rangle = N_0 \, \Lambda_{\mathcal P; N, [N_0]}(f_0, f_1,\ldots, f_m)
\end{align*}
where $\langle f, g\rangle = \int_\KK f(x) g(x) d\mu(x)$.
By duality we have
\begin{align*}
\langle A_N^{\mathcal P}(f_1,\ldots, f_m), f_0\rangle=
\langle (A_N^{\mathcal P})^{*j}(f_1,\ldots, f_{j-1}, f_0, f_{j+1}, \ldots, f_m), f_j\rangle,
\end{align*}
where
\begin{align*}
(A_N^{\mathcal P})^{*j}(f_1,\ldots,  f_0,  \ldots, f_m) (x) :=
\int_{\KK}\prod_{\substack{i=1\\i\neq j}}^mf_{i}(x-P_i(y)+P_j(y))f_0(x+P_j(y))d\mu_{[N]}(y).
\end{align*}

\begin{lemma}[Application of Hahn--Banach]\label{hahn} Let $A,B > 0$, let $I\subset \KK$ be an interval
and let $G$ be an element of $L^2(I)$.  Let $\Phi$ be a family of
vectors in $L^2(I)$, and assume the following inverse theorem:
whenever $f \in L^2(I)$ is such that $\|f\|_{L^\infty(I)} \leq 1$ and
$|\langle f, G \rangle| > A$, then $|\langle f, \phi \rangle| > B$ for
some $\phi \in \Phi$.  Then $G$ lies in the closed convex hull of
\begin{equation}\label{dip}
V= \{ \lambda\phi\in L^2(I): \phi \in \Phi, \ |\lambda| \leq A/B\}
\cup \{ h \in L^2(I): \|h\|_{L^1(I)} \leq A \}.
\end{equation}
\end{lemma}

\begin{proof} By way of contradiction, suppose that $G$ does not lie in 
$W = \overline{{\rm conv} V}^{\|\cdot\|_{L^2(I)}}$. From the Hahn-Banach theorem, we can
find a continuous linear functional $\Lambda$ of $L^2(I)$ which seperates $G$ from $W$; that is,
there is a  $C \in {\mathbb R}$ such that $\Rea\Lambda(h) \le C < \Rea \Lambda(G)$ for all $h \in W$.
Scaling $\Lambda$ allows us to change the constant $C$ so we can choose $\Lambda$ such that $C=A$ is in the statement of the lemma. 
Since $W$ is balanced, we see that $|\Lambda(h)| \le A < \Rea \Lambda(G)$ for all $h \in W$. 
By the Riesz representation theorem, there is an $f \in L^2(I)$ which represents $\Lambda$ so that
$|\langle f, h\rangle|\le A < \mathrm{Re} \langle f, G \rangle$ for all $h\in V$. This implies that
\[
|\langle f, \phi\rangle| \leq B
\]
for all $\phi \in \Phi$, and that
\[
\|f\|_{L^\infty(I)}=\sup_{\|h\|_{L^1(I)}\le1} |\langle f, h\rangle|\le 1,
\]
contradicting the hypothesis of the lemma.  This completes the
proof of the lemma.
\end{proof}

\begin{cor}[Structure of dual functions]\label{struct}
Let $N\ge 1$ be a scale, $m\in\Z_+$ and $0<\delta\le 1$ be given. Let $\mathcal P:=\{P_1,\ldots, P_m\}$
be a collection of polynomials such that
$1\le \deg{P_1}<\ldots<\deg{P_m}$.  Let
$f_0, f_1,\ldots, f_m\in L^0(\KK)$ be $1$-bounded functions supported
on an interval of measure $N_0 = N^{\deg(P_m)}$.  Then for every $j\in\bra{m}$, provided $N\gtrsim \delta^{-O_{\mathcal P}(1)}$, there exist a decomposition
\begin{equation}\label{decomp}
(A_N^{\mathcal P})^{*j}(f_1,\ldots, f_0, \ldots, f_m)(x)  = H_j(x) + E_j(x)
\end{equation} 
where  $H_j \in L^2(\KK)$ has Fourier transform supported in $[(N_j)^{-1}]$
where $N_j \simeq \delta^{C_j} N^{\deg{P_j}}$ and $C_j$ is as in Theorem \ref{thm:inverse-s}, 
and obeys the bounds
\begin{equation}\label{faq-bound}
\|H_j\|_{L^\infty(\KK)} \lesssim_m 1,
\quad\text{ and }\quad
\|H_j\|_{L^1(\KK)} \lesssim_m N_0.
\end{equation}
The error term $E_j\in L^1(\KK)$ obeys the bound
\begin{equation}\label{e-bound}
\|E_j\|_{L^1(\KK)} \leq \delta N_0.
\end{equation}
\end{cor}

\begin{proof}
Fix $j\in\bra{m}$, let
$I_0:=\supp{(A_N^{\mathcal P})^{*j}(f_1,\ldots, f_0, \ldots, f_m)}$,
and recall that $N_0= N^{\deg(P_m)}$. By translation invariance
we may assume $\supp{f_j}\subseteq [N_0]$ for all $j\in\bra{m}$, and that $I_0:=[O(N_0)]$. If there exists
$f\in L^{\infty}(I_0)$ with $\|f\|_{L^{\infty}(I_0)}\le 1$ such that
\begin{align}
\label{eq:85}
|\langle f, (A_N^{\mathcal P})^{*j}(f_1,\ldots, f_0, \ldots, f_m)  \rangle|> \delta N_0,
\end{align}
then proceeding as in the proof of Theorem \ref{thm:inverse-s} we may conclude that
\begin{align*}
|\langle \varphi_{N_j}*f, (A_N^{\mathcal P})^{*j}(\varphi_{N_1}*f_1,\ldots, \varphi_{N_{j-1}}* f_{j-1}, f_0, f_{j+1} \ldots, f_m)  \rangle|\ge c_m \, \delta N_0,
\end{align*}
where $N_i\simeq \delta^{C_i} N^{\deg(P_i)}$ for $i\in\bra{j}$.
This implies that there exists a 1-bounded $F\in L^2(\KK)$ with $\|F\|_{L^1(\KK)} \le N_0$  
such that  $\supp{\widehat{F}}\subseteq [N_j^{-1}]$ and
\begin{align}
\label{eq:87}
|\langle f, F  \rangle|\ge c_m \, \delta N_0.
\end{align}
If fact, we can take
$$
F(x) = {\tilde{\varphi}}_{N_j}*(A_N^{\mathcal P})^{*j}(\varphi_{N_1}*f_1,\ldots, \varphi_{N_{j-1}}* f_{j-1}, f_0, f_{j+1} \ldots, f_m)(x)
$$
where ${\tilde{\varphi}}(x) = \varphi(-x)$.
Let $\Psi$ denote the collection of all 1-bounded $F\in L^2(\KK)$ with $\supp{\widehat{F}}\subseteq [N_j^{-1}]$
and $\|F\|_{L^1(\KK)}\le N_0$.
Invoking  Lemma \ref{hahn} with $A=\delta N_0/4$ and $B = c_m \delta N_0$ and the set 
$\Phi = \{ F \ind{I_0} : F \in \Psi\}$,
we obtain a decomposition
\begin{align}
\label{eq:88}
(A_N^{\mathcal P})^{*j}(f_1,\ldots,  f_{j-1}, f_0, f_{j+1}, \ldots, f_m)  = \sum_{l=1}^{\infty}c_l\phi_l+E(1)+E(2),
\end{align}
with the following properties:
\begin{itemize}
\item[(i)]  for each $l\in\Z_+$ we have that $\phi_l=\lambda_lF_l \ind{I_0}$, $F_l \in \Psi$  and $\lambda_l\in\C$ such that  $|\lambda_l|\lesssim_m 1$;
\item[(ii)] the coefficients $c_l$ are non-negative with
$\sum_{l=1}^{\infty}c_l\le 1$, and all but finitely $c_l$ vanish;
\item[(iii)] the error term $E(1)\in L^1(I_0)$  satisfies  $\|E(1)\|_{L^1(I_0)} \leq \delta N_0/2$;
\item[(iv)] the error term $E(2)\in L^2(I_0)$  satisfies  $\|E(2)\|_{L^2(I_0)} \leq \delta$.
\end{itemize}
The latter error term arises as a consequence of the fact that
one is working with the closed convex hull instead of the convex hull. In fact, its $L^2(I_0)$ norm can be made arbitrarily small, but
$\delta$ will suffice for our purposes. 

Grouping together terms in the deomposition \eqref{eq:88}, we have
\[
(A_N^{\mathcal P})^{*j}(f_1,\ldots, f_{j-1}, f_0, f_{j+1}, \ldots, f_m)
=   H_j' + E_j'
\]
where
\begin{align*}
H_j' = \bigl[\sum_{l=1}^{\infty} c_l \lambda_l F_l\bigr] \ind{I_0} \quad \text{satisfies} \quad 
\|H_j'\|_{L^1(\KK)} \le \sum_{l=1}^{\infty} c_l |\lambda_l| \|F_l\|_{L^1(\KK)} \lesssim_m N_0 \ \ \text{and}
\end{align*} 
\begin{align*} 
\|H_j'\|_{L^{\infty}(\KK)} \le \sup_{l\in\NN} \|F_l\|_{L^{\infty}(\KK)} \sum_{l=1}^{\infty} c_l |\lambda_l| \lesssim_m 1.
\end{align*}
Also $E_j' = E(1) + E(2)$ satisfies $\|E_j'\|_{L^1(I_0)} \le \delta N_0$ by (iii) and (iv) above since by the Cauchy--Schwarz inequality, 
we have $\|E(2)\|_{L^1(I_0)} \leq \delta N_0^{1/2}$. 

We note that the function $F(x) = \sum_{i=1}^{\infty} c_i \lambda_i F_i (x)$ is Fourier supported in the interval $[N_j^{-1}]$.

When $\KK$ is non-archimedean, ${\rm supp}({\widehat{\ind{I_0}}}) \subseteq [N_0^{-1}]$ and so the Fourier transform
of $H_j'$ is supported in $[N_j^{-1}]$. 
This verifies \eqref{faq-bound} in this case and completes the proof when $\KK$ is non-archimedean since
the decomposition $H_j' + E_j'$ of $(A_N^{\mathcal P})^{*j}$ satisfies \eqref{faq-bound} and \eqref{e-bound}.

Now suppose $\KK$ is archimedean. Let $\psi$ be a Schwartz function such that $\int_{\KK}\psi(x)d\mu(x)=1$ and  
$\supp{\widehat{\psi}}\subseteq [2]$. 
Let $M\simeq \delta^{O(1)} N_0$ and as usual, set $\psi_M(x) = M^{-1} \psi(M^{-1}x)$ when $\KK = {\mathbb R}$
and $\psi_M(x) = M^{-1}\psi(M^{-1/2} x)$ when $\KK = {\mathbb C}$. From the proof of \eqref{L1-norm}, we have
\begin{align}\label{L1-ind}
\|\ind{I_0}-\ind{I_0}*\psi_M\|_{L^1(\KK)} \lesssim M^{1/4} {N_0}^{3/4}.
\end{align}
We set $H_j(x)  = F(x) \ind{I_0} * \psi_M (x)$ and $E_j = E(1) + E(2) + (\ind{I_0} - \ind{I_0}*\psi_M) F$ so that
$$
(A_N^{\mathcal P})^{*j}(f_1,\ldots, f_{j-1}, f_0, f_{j+1}, \ldots, f_m) (x)
=   H_j(x) + E_j (x).
$$
From \eqref{L1-ind}, we see that $E_j$ satisfies \eqref{e-bound}.
The properties $\|H_j\|_{L^\infty(\KK)} \lesssim_m 1$ and $\|H_j\|_{L^1(\KK)} \lesssim_m N_0$ are 
still preserved.
Moreover, $\supp{\widehat{H}_j}\subseteq [O(N_j^{-1})]$, since 
\[
\widehat{H}_j=(\widehat{\ind{I_0}}\widehat{\varphi}_M)*\widehat{F}.
\]
The shows that \eqref{faq-bound} holds for $H_j$ and this completes the proof of the corollary.
\end{proof}

We will combine Corollary \ref{struct} and the following $L^p$ improving bound for polynomial averages
to establish the key Sobolev inequality.

\begin{lemma}[$L^p$-improving for polynomial averages]\label{poly-improving} Let 
$Q \in \KK[{\rm y}]$ with $\deg(Q) = d$ and let $N\gg_Q 1$ be a large scale. Consider the averaging operator 
\[
M_N^{Q}g(x):=\int_\KK g(x-Q(y))d\mu_{[N]}(y).
\]
For any parameters
$1 < r < s< \infty$ satisfying $1/s = 1/r - 1/d$, the following inequality holds:
\begin{align}
\label{eq:90}
\|M_N^{Q}g\|_{L^s(\KK)}\lesssim_Q N^{d(\frac{1}{s}-\frac{1}{r})}\|g\|_{L^r(\KK)}\quad \text{ for } \quad g\in L^r(\KK).
\end{align}
\end{lemma}

\begin{proof} As our bounds are allowed to depend on $Q$, we may assume that $Q$ is monic.
Let $\alpha\in \KK$ be such that $|\alpha| = N$ and change variables $y \to \alpha y$ to write
$$
M^N_Q g(x) \ = \int_{B_1(0)} g(x - Q(\alpha y)) \, d\mu(y) = \int_{B_1(0)} g_{\alpha}(\alpha^{-d} x - Q_{\alpha}(y)) \, d\mu(y)
$$
where $g_{\alpha}(x) = g(\alpha^d x)$ and 
$Q_{\alpha}(y) = \alpha^{-d} Q(\alpha y) =  y^d + \alpha^{-1} a_{d-1} y^{d-1} + \ldots+\alpha^{-d}a_0$. Hence the right-hand side
above can be written as $M^1_{Q_{\alpha}} g_{\alpha}( \alpha^{-d} x)$. Since $\|g_{\alpha}\|_{L^r(\KK)} = N^{-d/r} \|g\|_{L^r(\KK)}$,
we see that matters are reduced to proving \eqref{eq:90} for $N=1$ and $Q = Q_{\alpha}$ with uniform bounds in $\alpha$.

The mapping $y \to Q_{\alpha}(y)$ is $d$-to-$1$ and we can use a
generalised change of variables formula to see that
\begin{align*}
|M^1_{Q_{\alpha}} g(x)| \lesssim \int_{|s|\le 2} |g(x -s)| |s|^{-(d-1)/d} d\mu(s)
\end{align*}
when $N\gg_{Q} 1$. Hence $M^1_{\alpha}$ is controlled by fractional integration, uniformly in $\alpha$. 
When $\KK$ is archimedean, such a change of variables formula is well-known. Recall that when $\KK = {\mathbb C}$, $|s| = s {\overline{s}}$ is the square of the usual absolute value. 

When $\KK = {\mathbb Q}_p$ is the $p$-adic field, such a formula is given
in \cite{Evans}. The argument in \cite{Evans} generalises to general non-archimedean fields (when the charateristic, if positive,
is larger than $d$).  Alternatively one can use a construction in \cite{W-igusa}, valid in any local field and valid
for any polynomial $Q$ where $Q'(x)$ does not equal to zero mod $m_\KK$ for any nonzero $x$ (we need the condition on the
characteristic of the field for this), in which the unit group $U = \bigcup_{j\in\bra{J}} U_j$ is partitioned into $J={\rm gcd}(d, q-1)$ open sets and analytic isomorphisms $\phi_j :D_j \to \phi_j(D_j)$ are constructed such that $y = \phi_j(x)$
precisely when $Q(y) = x$. For us, $Q_{\alpha}'(x) \not= 0$ mod $m_\KK$ for any nonzero $x$  
if $|\alpha| = N \gg_Q 1$ is sufficiently large.

By the Hardy-Littlewood-Sobolev
inequality (easily seen to be valid over general locally compact topological fields), we have
$$
\|M^1_{Q_{\alpha}} g \|_{L^s(\KK)} \lesssim \|g\|_{L^r(\KK)},
$$
uniformly in $\alpha$ whenever $1/s = 1/r - 1/d$, completing the proof of the lemma.
\end{proof}

We now come to the proof of Theorem \ref{sobolev-informal}.

As in the set up for Theorem \ref{thm:inverse-s}, we fix a smooth function $\varphi$ with
compact Fourier support. When $\KK$ is archimedean, let $\varphi$ be a
Schwartz function on $\KK$ so that
\begin{align*}
\ind{[1]}(\xi)\le \widehat{\varphi}(\xi)\le \ind{[2]}(\xi), \qquad \xi\in \KK.
\end{align*}
When $\KK={\mathbb R}$, we set $\varphi_N(x) = N^{-1} \varphi(N^{-1} x)$ for any
$N>0$ and when $\KK = {\mathbb C}$, we set
$\varphi_{N}(z) = N^{-1} \varphi(N^{-1/2} z)$ for any $N > 0$. When
$\KK$ is non-archimedean, we set $\varphi(x) = \ind{B_1(0)}(x)$ so that $\widehat{\varphi}(\xi) = \ind{B_1(0)}(\xi)$ and 
we set $\varphi_{N}(x) = N^{-1} \ind{[N]}(x)$ for any scale $N$.  
We restate Theorem \ref{sobolev-informal} in a more formal, precise way.

\begin{theorem}[A Sobolev inequality for $A_N^{\mathcal P}$]
\label{sobolev}
Let
$\mathcal P:=\{P_1,\ldots, P_m\}$ be a collection of polynomials such
that $1\le \deg{P_1}<\ldots<\deg{P_m}$. Let $N \gg_{\mathcal P} 1$ be a scale, $m\in\Z_+$ and $0<\delta\le 1$ be given. 
Let $1<p_1,\ldots, p_m<\infty$ satisfying
$\frac{1}{p_1}+\ldots+\frac{1}{p_m}=1$ be given. Suppose $N\gtrsim \delta^{-O_{\mathcal P}(1)}$. Then
for all $f_1\in L^{p_1}(\KK),\ldots, f_m\in L^{p_m}(\KK)$ we have
\begin{align}
\label{eq:97}
\|A_N^{\mathcal P}(f_1,\ldots,f_{j-1}, (\delta_0-\varphi_{N_j})*f_j,f_{j+1}\ldots, f_m)\|_{L^1(\KK)}
\lesssim \delta^{1/8}
\prod_{i=1}^{m}
\|f_i\|_{L^{p_i}(\KK)},
\end{align}
where $N_j \simeq \delta^{C_j} N^{\deg{P_j}}$ and $C_j$ is the parameter from Theorem \ref{thm:inverse-s}. Here 
$\widehat{\delta_0} \equiv 1$.
\end{theorem}

\paragraph{\bf Remark} The proof of Theorem \ref{sobolev} (and its statement) implicitly assumes that $m\ge 2$ but there is a version when $m=1$, which 
will be given in Section \ref{appendix} where it is needed. 
\begin{proof}
We fix $j\in\bra{m-1}$ and recall $N_j\simeq \delta^{O(1)}N^{\deg(P_j)}$. We first prove that for every functions  $f_1,\ldots, f_{j-1}, f_{j+1},\ldots, f_{m-1}\in L^{\infty}(\KK)$ and $f_j, f_m\in L^2(\KK)$,  we have
\begin{gather}
\label{eq:94}
\begin{split}
\|A_N^{\mathcal P}(f_1,\ldots,f_{j-1}, (\delta_0-\varphi_{N_j})*f_j,f_{j+1}\ldots, f_m)\|_{L^1(\KK)}\\
\lesssim \delta^{1/8}
\bigg(\prod_{\substack{i=1\\i\neq j}}^{m-1}
\|f_i\|_{L^\infty(\KK)}\bigg)\|f_j\|_{L^2(\KK)}\|f_m\|_{L^2(\KK)}.\qquad 
\end{split}
\end{gather}
Choose $f_0\in L^{\infty}(\KK)$  so that $\|f_0\|_{L^{\infty}(\KK)}=1$ and 
\begin{align*}
\|A_N^{\mathcal P}(f_1,&\ldots,f_{j-1}, (\delta_0-\varphi_{N_j})*f_j,f_{j+1}\ldots, f_m)\|_{L^1(\KK)}\\
&\simeq |\langle A_N^{\mathcal P}(f_1,\ldots,f_{j-1}, (\delta_0-\varphi_{N_j})*f_j,f_{j+1}\ldots, f_m), f_0\rangle|\\
&=|\langle (\delta_0-\varphi_{N_j})*(A_N^{\mathcal P})^{*j}(f_1,\ldots, f_0, \ldots, f_m), f_j\rangle|
\end{align*}
By the Cauchy--Schwarz inequality it will suffice to prove
\begin{gather}
\label{eq:96}
\begin{split}
\|(\delta_0-\varphi_{N_j})*(A_N^{\mathcal P})^{*j}(f_1,\ldots, f_0, \ldots, f_m)\|_{L^2(\KK)}\\
\lesssim \delta^{1/8}
\|f_0\|_{L^\infty(\KK)}\bigg(\prod_{\substack{i=1\\i\neq j}}^{m-1}
\|f_i\|_{L^\infty(\KK)}\bigg)\|f_m\|_{L^2(\KK)}.
\end{split}
\end{gather}
By multilinear interpolation, the bounds \eqref{eq:94} imply \eqref{eq:97} and so the proof of
Theorem \ref{sobolev} is reduced to establishing \eqref{eq:96}
which will be divided into three steps. In the first two steps, we will assume that $f_m$ is supported
in some interval of measure $N_0$ where $N_0 \simeq N^{\deg(P_m)}$.
\paragraph{\bf Step 1.} In this step, we will establish the bound
\begin{gather}
\label{eq:91}
\begin{split}
\|(\delta_0-\varphi_{N_j})*(A_N^{\mathcal P})^{*j}(f_1,\ldots, f_0, \ldots, f_m)\|_{L^2(\KK)}\quad \\
\lesssim \delta^{1/2}N_0^{1/2}
\|f_0\|_{L^\infty(\KK)}\bigg(\prod_{\substack{i=1\\i\neq j}}^{m-1}
\|f_i\|_{L^\infty(\KK)}\bigg)\|f_m\|_{L^\infty(\KK)}
\end{split}
\end{gather}
under the assumption that $f_m$ is supported in an interval of measure $N_0$ (when $\KK = {\mathbb C}$, this implies 
in particular that $f_m$ is supported in a square with measure about $N_0$, which in Step 3.  will be a helpful observation). 
When $f_m$ has this support condition, 
$$
(A_N^{\mathcal P})^{*j}(f_1,\ldots, f_0, \ldots, f_m)  = (A_N^{\mathcal P})^{*j}(f_1',\ldots, f_0', \ldots, f_m')
$$
where $f_i'(x) = f_i(x) \ind{I_0}(x)$ for some interval $I_0$ of measure $O(N_0)$. To prove \eqref{eq:91}, it suffices
to assume $\|f_i\|_{L^{\infty}(\KK)} = 1$ for $i=0,1,\ldots, j-1, j+1,\ldots, m$ and so \eqref{eq:91} takes the form
\begin{align}
\label{eq:89}
\|(\delta_0-\varphi_{N_j})*(A_N^{\mathcal P})^{*j}(f_1,\ldots, f_0, \ldots, f_m)\|_{L^2(\KK)}\lesssim \delta^{1/2}N_0^{1/2}.
\end{align}
We apply the decomposition \eqref{decomp} to $(A_N^{\mathcal P})^{*j}(f_1',\ldots, f_0', \ldots, f_m')$ to write
$$
(A_N^{\mathcal P})^{*j}(f_1,\ldots, f_0, \ldots, f_m) (x) = H_j(x) + E_j(x)
$$
where $H_j$ satisfies \eqref{faq-bound} and $E_j$ satisfies \eqref{e-bound}.
Using the fact that $\widehat{H}_j\subseteq [(N_j)^{-1}]$ we conclude
that $(\delta_0-\varphi_{N_j})*H_j=0$. Thus 
\[
(\delta_0-\varphi_{N_j})*(A_N^{\mathcal P})^{*j}(f_1,\ldots, f_0, \ldots, f_m)=(\delta_0-\varphi_{N_j})*E_j.
\]
From \eqref{e-bound} and the 1-boundedness of $(A_N^{\mathcal P})^{*j}(f_1,\ldots, f_0, \ldots, f_m)$, we have
\begin{align*}
\|(\delta_0 -\varphi_{N_j})*E_j\|_{L^1(\KK)}\lesssim \delta N_0,
\quad \text{ and } \quad
\|(\delta_0 -\varphi_{N_j})*E_j\|_{L^\infty(\KK)}\lesssim 1,
\end{align*}
respectively. Therefore
\begin{align*}
\|(\delta_0 -\varphi_{N_j})*E_j\|_{L^2(\KK)}\lesssim \delta^{1/2} N_0^{1/2},
\end{align*}
establishing \eqref{eq:89} and hence \eqref{eq:91}. This completes Step 1.
\paragraph{\bf Step 2.}
We continue with our assumption that $f_m$ is supported in an interval of measure $N_0$ but
now we relax the $L^\infty(\KK)$ control on $f_m$ to $L^2(\KK)$
control and show that
\begin{gather}
\label{eq:92}
\begin{split}
\|(\delta_0-\varphi_{N_j})*(A_N^{\mathcal P})^{*j}(f_1,\ldots, f_0, \ldots, f_m)\|_{L^2(\KK)}\\
\lesssim \delta^{1/4}\|f_0\|_{L^\infty(\KK)}\bigg(\prod_{\substack{i=1\\i\neq j}}^{m-1}
\|f_i\|_{L^\infty(\KK)}\bigg)\|f_m\|_{L^2(\KK)}.
\end{split}
\end{gather}
The main tool for this will be the $L^p$-improving
estimate \eqref{eq:90} for the polynomial average $M_N^Q$.
We have a pointwise bound 
\begin{align*}
|(A_N^{\mathcal P})^{*j}(f_1,\ldots, f_0, \ldots, f_m)(x)|\le M_N^{P_m-P_j}|f_m|(x),
\end{align*}
which combined with \eqref{eq:90} (for $Q=P_m-P_j$, $d = \deg(P_m)$, $s=2$ and $r=(d+2)/2d$) yields 
\begin{gather}
\label{eq:93}
\begin{split}
\|(\delta_0-\varphi_{N_j})*(A_N^{\mathcal P})^{*j}(f_1,\ldots, f_0, \ldots, f_m)\|_{L^2(\KK)}\quad\\
\lesssim N_0^{- 1/d}\|f_0\|_{L^\infty(\KK)}\bigg(\prod_{\substack{i=1\\i\neq j}}^{m-1}
\|f_i\|_{L^\infty(\KK)}\bigg)\|f_m\|_{L^r(\KK)}.
\end{split}
\end{gather}
Interpolating \eqref{eq:91} and \eqref{eq:93} we obtain \eqref{eq:92} as desired.
\paragraph{\bf Step 3.} In this final step, we remove the support condition on $f_m$ and establish
\eqref{eq:96}. To prove \eqref{eq:96}, we may assume that $\|f_i\|_{L^{\infty}(\KK)} = 1$ for $i=0, 1, \ldots, j-1, j+1, \ldots, m-1$.
We split $f_m = \sum_{I\in\mathcal I} f_m \ind{I}$ where $I$ ranges over a partition $\mathcal I$ of $\KK$ into intervals $I$ of 
measure $N_0$. We  have seen this is possible when $\KK$ is non-archimedean or when $\KK={\mathbb R}$. This is not possible
when $\KK = {\mathbb C}$ but in this case, we can find a partition $\mathcal I$ of squares.
By Step 1. and Step 2.,  the local dual function 
$D_I \coloneqq (A_N^{\mathcal P})^{*j}(f_1,\ldots, f_0, \ldots, f_m \ind{I})$ obeys the bound
\begin{equation}\label{anghi}
\| (\delta_0 - \varphi_{N_j})* D_I \|_{L^2(\KK)} \lesssim \delta^{1/4} \| f_m \|_{L^2(I)}
\end{equation}
for each interval $I$, and we wish to establish
\[
\Big\| \sum_{I\in\mathcal I} (\delta_0 - \varphi_{N_j})* D_I\Big\|_{L^2(\KK)} \lesssim\delta^{1/8} \| f_m \|_{L^2(\KK)}.
\]
We will square out the sum. To handle the off-diagonal terms, we observe that for finite intervals
$I, J\subset \KK$ (squares when $\KK = {\mathbb C}$) of measure $N_0$ and $M>0$ and $1\le p<\infty$, we have
\begin{align}
\label{eq:95}
\|\varphi_{N_j}*(f\ind{I})\|_{L^p(J)}\lesssim_{M, p} \big(1+N_0^{-1}{\rm dist}(I, J)\big)^{-M}  \|f\|_{L^p(I)}.
\end{align}
By squaring and applying Schur's test, it suffices to obtain the decay bound
\[
\bigl| \langle (\delta_0 - \varphi_{N_j})* D_I, (1 - \varphi_{N_j})* D_J \rangle \bigr| \lesssim
\delta^{1/4} \big(1+N_0^{-1}\mathrm{dist}(I,J) \big)^{-2} \| f_m \|_{L^2(I)} \| f_m \|_{L^2(J)}
\]
for all intervals $I,J$ of measure $N_0$. By Cauchy--Schwarz and \eqref{anghi} we know
\[
\langle (\delta_0 - \varphi_{N_j})* D_I, (1 - \varphi_{N_j})* D_J \rangle \lesssim
\delta^{1/2}  \| f_m \|_{L^2(I)} \| f_m \|_{L^2(J)}.
\]
On the other hand, $D_I$ is supported in a $O(N_0)$-neighborhood of $I$, and similarly for $D_J$.  From \eqref{eq:95} and Cauchy--Schwarz, we thus have
\begin{align*}
\langle (\delta_0 - \varphi_{N_j})* D_I, (1 - \varphi_{N_j})* D_J \rangle &\lesssim
\big(1+N_0^{-1}\mathrm{dist}(I,J) \big)^{-10} \| D_I \|_{L^2(\KK)} \| D_J \|_{L^2(\KK)}\\
&\lesssim \big(1+N_0^{-1}\mathrm{dist}(I,J) \big)^{-10} \| f_m \|_{L^2(I)} \| f_m \|_{L^2(J)}.
\end{align*}
Taking the geometric mean of the two estimates, we obtain the claim in \eqref{eq:96}.
This completes the proof of Theorem \ref{sobolev}.
\end{proof}

\section{The implication Theorem \ref{sobolev-informal} $\Longrightarrow$ Theorem \ref{thm:main}}\label{appendix}

Here we give the details of Bourgain's argument in \cite{B} which allow us to pass from Theorem \ref{sobolev-informal}
to Theorem \ref{thm:main} on polynomial progressions. Let ${\mathcal P} = \{P_1, \ldots, P_m\}$ be a sequence of polynomials in $\KK[{\rm y}]$ with distinct degrees and
no constant terms. Without loss of generality, we may assume
$$
\deg{P_1} < \deg{P_2} < \cdots < \deg{P_m}
$$
and we set $d_{m-j} := \deg{P_j}$ and $d := d_0 = \deg{P_m}$ so that $d_{m-1} < \cdots < d_1 < d$.

Since the argument showing how Theorem \ref{sobolev-informal} implies Theorem \ref{thm:main} has been given
in \cite{B}, \cite{D+}, and \cite{CGL} in the euclidean setting (albeit for shorter polynomial progressions), we will only give the details
for non-archimedean fields $\KK$ where uniform notation can be employed.

We will proceed in several steps.

\paragraph{\bf Step 1} When $\KK$ is non-archimedean, the family $(Q_t)_{t>0}$ of convolution operators defined by
\[
Q_t f (x) \ = \ f * \mu_{[t]}(x)  = \ \frac{1}{t} \int_{|y|\le t} f(x - u) d\mu(u) \quad \text{for scales} \quad  t>0
\]
gives us a natural appoximation of the identity and
form the analogue of the Poisson semigroup in the non-archimedean setting. They also
give us Fourier localization since 
\begin{align}\label{poisson-fourier}
{\widehat{Q_t f}}(\xi) = {\widehat{Q_t}}(\xi) {\widehat{f}}(\xi) = \ind{[t^{-1}]}(\xi) {\widehat{f}}(\xi).
\end{align}
We will need the
following bound for $(Q_t)_{t>0}$ (see Lemma 6 in \cite{B} or  Lemma 2.1 in \cite{D+}): for $f \ge 0$ and
scales $0 < t_1, \ldots, t_m \le 1$,
\begin{align}\label{poisson}
\int_{B_1(0)} f(x) Q_{t_1} f(x) \cdots Q_{t_m} f(x) d\mu(x)  \ \ge \ \Bigl( \int_{B_1(0)} f(x) d\mu(x) \Bigr)^{m+1}.
\end{align}
The proof in the euclidean setting given in \cite{D+} established \eqref{poisson} for general approximations of the identity
but the first step is to show \eqref{poisson} for martingales $(E_k)_{k\in\N}$ defined with respect to dyadic intervals.
However a small scale $t$ in a non-archimedean field $\KK$ is the form $t = q^{-k}$ and
\[
Q_t f(x) = q^k \int_{|y|\le q^{-k}} f(x - y) d\mu(y) = \sum_{{\underline{x}}\in {\mathcal C}_k} A_{k,{\underline{x}}} f
\ \ind{B_{q^{-k}}}({\underline{x}}), \quad \text{ where }
\]
\[
{\mathcal C}_k \ = \ \{ {\underline{x}}=x_0 + x_1 \pi + \cdots + x_{k-1}\pi^{k-1} : x_j \in o_\KK/m_\KK \}
\quad \text{ and } \quad
A_{k,\underline{x}}f = q^{k} \int_{B_{q^{-k}({\underline{x}})}} f(u) d\mu(u). 
\]
Hence $(Q_t)_{t>0}$
is a martingale with respect to the dyadic structure of non-archimedean fields and so the argument in \cite{D+}
extends without change to establish \eqref{poisson}. 

\paragraph{\bf Step 2} Fix $\varepsilon>0$. Our goal is to find  a $\delta(\varepsilon, {\mathcal P}) > 0$ and $N(\varepsilon, {\mathcal P}) \ge 1$ such
that for any scale $N \ge N(\varepsilon, {\mathcal P})$ and $f \in L^0(\KK)$ with $0\le f \le 1$ satisfying $\int_\KK f d\mu \ge \varepsilon N^d$,
we have
\begin{align}\label{mult-recurrence-function}
I := \frac{1}{N^d} \iint_{\KK^2} f(x) f(x+P_1(y)) \cdots f(x+P_m(y)) d\mu_{[N]}(y) d\mu(x) \ \ge \ \delta.
\end{align}
Taking $f = \ind{S}$ with $S\subseteq \KK$ in Theorem \ref{thm:main} implies \eqref{mult-recurrence}, the desired conclusion.
We may assume the $f$ is supported in the interval $[N^d]$. 

Let $\alpha, \beta \in \KK$ satisfy $|\alpha| = N^d$ and $|\beta| = N$ and write
$$
I = \iint_{\KK^2} g(x) g(x + R_1(y)) \cdots g(x +R_m(y)) d\mu_{[1]}(y) d\mu(x),
$$
where $g(x) = f(\alpha x)$ and $R_j(y) = \alpha^{-1} P_j(\beta y)$. In particular, we have
$\int_\KK g \ge \varepsilon$. We note that $g$ is supported in $[1] = B_1(0)$. Fix three small scales $0<t_0 \ll t_1 \ll t \ll 1$
and decompose
\begin{align}\label{I-decomp}
t_1^{-1} I \ge \iint_{\KK^2} g(x) g(x + R_1(y)) \cdots g(x + R_m(y)) d\mu_{[t_1]}(y) d\mu(y) \ =: \ I_1 + I_2 + I_3,
\end{align}
where 
$$
I_1 = \iint_{\KK^2} g(x) \prod_{j=1}^{m-1} g(x + R_j(y))\  Q_t g(x +R_m(y)) d\mu_{[t_1]}(y) d\mu(x),
$$
$$
I_2 = \iint_{\KK^2} g(x) \prod_{j=1}^{m-1} g(x + R_j(y))\  [Q_{t_0} - Q_t ] g(x +R_m(y)) d\mu_{[t_1]}(y) d\mu(x) \ \ {\rm and}
$$
$$
I_3 = \iint_{\KK^2} g(x) \prod_{j=1}^{m-1} g(x + R_j(y))\  [ {\rm Id} - Q_{t_0}] g(x +R_m(y)) d\mu_{[t_1]}(y) d\mu(x).
$$
For $I_1$, we note that for $t_1 \ll_{P_m} t$,
$$
Q_t g(x + R_m(y)) = \frac{1}{t} \int_{|u|\le t} g(x + R_m(y) - u) d\mu(u) = \frac{1}{t} \int_{|u|\le t} g(x-u) d\mu(u) = 
Q_t g(x)
$$
whenever $|y| \le t_1$. For the final equality we made the change of variables $u \to u - R_m(y)$, noting that when $|y| \le t_1$, then
$|R_m(y)| \le C_{P_m} t_1 \le t$. Hence
\begin{align*}
I_1 = \iint_{\KK^2} g(x)  \prod_{j=1}^{m-1} g(x + R_j(y))\   Q_t g(x)\, d\mu_{[t_1]}(y) d\mu(x).
\end{align*}
For $I_2$ we use the Cauchy--Schwarz inequality to see that
\begin{align}\label{I_2}
I_2 \ \le \ \| Q_{t_0} g - Q_t g \|_{L^2(\KK)}.
\end{align}
For $I_3$, we will use the more precise formulation of Theorem \ref{sobolev-informal} given in Theorem \ref{sobolev}.
We rescale $I_3$, moving from $g, R_j$ back to $f, P_j$ and write
\[
I_3 = \frac{1}{N^d} \iint_{\KK^2} f(x) \prod_{j=1}^{m-1} f(x + P_j(y))\   [{\rm Id} - Q_{t_0 N^d}] f (x + P_m(y)) d\mu_{[t_1N]}(y) d\mu(x),
\]
where the function $h(x) = [{\rm Id} - Q_{t_0N^d}] f(x)$ has the property that ${\widehat{h}}(\xi) = 0$ 
whenever $|\xi| \le (t_0 N^d)^{-1}$, see \eqref{poisson-fourier}. Hence
$$
I_3 \le N^{-d} \|A^{\mathcal P}_{t_1N}(f,f, \ldots, f, [{\rm Id} - Q_{t_0N^d}] f) \|_{L^1(\KK)}
$$
and we will want to apply Theorem \ref{sobolev} to the expression on the right
with $N$ replaced by $t_1N$ and $0 < \delta \le 1$ defined by $\delta^{C_m} (N t_1)^d = N^d t_0$
or $\delta = (t_0/ t_1^d)^{1/C_m}$. In order to apply Theorem \ref{sobolev}, we will need to ensure
\begin{align}\label{N-delta}
N \ \ge \ t_1^{-1} (t_1^{d_{m-1}} / t_0)^{C'} \ge\ldots \ge \ t_1^{-1} (t_1^{d} / t_0)^{C'}
\end{align}
for some appropriate large $C' = C'_{\mathcal P}$. If \eqref{N-delta} holds, then Theorem \ref{sobolev} implies 
there exists a constant $b = b_{\mathcal P} > 0$ such that
$$
\|A^{\mathcal P}_{t_1N}(f,f, \ldots, f, h)\|_{L^1(\KK)} \lesssim_{\mathcal P} \bigl(t_0/t_1^d\bigr)^b \prod_{j=1}^m \|f\|_{L^{p_i}(\KK)}
\le \bigl(t_0/t_1^d\bigr)^b N^d
$$
since $1/p_1 + \cdots + 1/p_m = 1$ and $\|f\|_{L^{p_i}(\KK)} \le N^{d/p_i}$ for $i\in\bra{m}$  (which follows since $f$ is 1-bounded and supported in $[N^d]$).
 Hence
\begin{align*}
I_3 \lesssim_{\mathcal P} \bigl(t_0/t_1^d\bigr)^b  \quad \text{ if } \quad\eqref{N-delta}\quad \text{holds}.
\end{align*}

\paragraph{\bf Step 3} Next we decompose $I_1 = I_1^1 + I_{2}^1 + I_3^1$, where
$$
I_1^1 \ = \ \iint_{\KK^2}  g(x)  \prod_{j=1}^{m-2} g(x + R_j(y))\  Q_{t/N^{d-d_1}} g(x + R_{m-1}(y)) Q_t g(x) \, d\mu_{[t_1]}(y) d\mu(x),
$$
$$
I_{2}^1 = \iint_{\KK^2} g(x) \prod_{j=1}^{m-2} g(x + R_j(y))\ [Q_{t_0/N^{d-d_1}} - Q_{t/N^{d-d_1}} ] g(x +R_{m-1}(y)) 
Q_t g(x) d\mu_{[t_1]}(y) d\mu(x) \ \ {\rm and}
$$
$$
I_{3}^1 = \iint_{\KK^2} g(x) \prod_{j=1}^{m-2} g(x + R_j(y))\ [ {\rm Id} - Q_{t_0/N^{d-d_1}}] g(x +R_{m-1}(y)) Q_t g(x) d\mu_{[t_1]}(y) d\mu(x).
$$

For $I_1^1$, we set $s = t/N^{d-d_1}$ and note that for $t_1 \ll_{\mathcal P} t$,
$$
Q_{s} g(x + R_{m-1}(y))
= \frac{1}{s} \int_{|u|\le s} g(x + R_{m-1}(y) - u) d\mu(u)
= \frac{1}{s} \int_{|u|\le s} g(x-u) d\mu(u) = 
Q_s g(x)
$$
whenever $|y| \le t_1$. For the final equality we made the change of variables $u \to u - R_{m-1}(y)$, noting that when $|y| \le t_1$, then
$|R_{m-1}(y)| \le C_{P_{m-1}} N^{-(d-d_1)}t_1 \le s$ since $ t_1 \ll_{\mathcal P} t$. Hence
\begin{align*}
I_1^1 = \iint_{\KK^2} g(x)  \prod_{j=1}^{m-2} g(x + R_j(y))\ Q_{t/N^{d-d_1}} g(x) Q_t g(x) \, d\mu_{[t_1]}(y) d\mu(x).
\end{align*}
As in \eqref{I_2}, we have
$$
I_{2}^1 \le \|Q_{t_0/N^{d-d_1}} g  -  Q_{t/N^{d-d_1}} g \|_{L^2(\KK)}.
$$

For $I_{3}^1$, we will use Theorem \ref{sobolev}.
We rescale $I_{3}^1$, moving from $g, R_j$ back to $f, P_j$ and write
$$
I_{3}^1 = \frac{1}{N^d} \iint_{\KK^2} f(x) \prod_{j=1}^{m-2} f(x + P_j(y))\ [{\rm Id} - Q_{t_0 N^{d_1}}] f (x + P_{m-1}(y))
Q_{t N^d} f(x)  d\mu_{[t_1N]}(y) d\mu(x)
$$
where the function $h'(x) = [{\rm Id} - Q_{t_0 N^{d_1}}] f(x)$ has the property that ${\widehat{h'}}(\xi) = 0$ 
whenever $|\xi| \le (t_0 N^{d_1})^{-1}$. Hence for ${\mathcal P'} = \{P_1, \ldots, P_{m-1}\}$,
\begin{align*}
I_{3}^1 \le N^{-d} \|A^{\mathcal P'}_{ t_1 N}(fQ_{t N^d} f,f, \ldots, f, [{\rm Id} - Q_{t_0 N^{d_1}}] f) \|_{L^1(\KK)}
\end{align*}
and so, as long as \eqref{N-delta} holds, Theorem \ref{sobolev} implies there exists a constant 
$b' = b_{\mathcal P'} > 0$ such that
$$
\|A^{\mathcal P'}_{t_1N}(fQ_{t N^d} f,f, \ldots, f, h')\|_{L^1(\KK)} \lesssim_{\mathcal P'} \bigl(t_0/t_1^d\bigr)^{b'} \prod_{j=1}^m \|f\|_{L^{p_i}(\KK)}
\le \bigl(t_0/t_1^d\bigr)^{b'} N^d
$$
since $1/p_1 + \cdots + 1/p_{m-1} = 1$ and $\|f\|_{L^{p_i}(\KK)} \le N^{d/p_i}$ for $i\in\bra{m-1}$  (which follows since $f$ is 1-bounded and supported in $[N^d]$).
 Hence
\begin{align*}
I_{3}^1 \lesssim_{\mathcal P'} \bigl(t_0/t_1^d\bigr)^{b'} \quad \text{ if } \quad \eqref{N-delta} \quad \text{holds}.
\end{align*}

\paragraph{\bf Step 4} We iterate, decomposing $I_1^1 = I_{1}^2 + I_{2}^2 + I_{3}^2$, followed by decomposing 
$I_1^2 = I_1^3 + I_2^3 + I_3^3$ and so on. For each $0\le j \le m-1$, we have
\begin{align}
\label{I_1^j}
&I_1^j =  \iint_{\KK^2} g(x)  \Big(\prod_{i=1}^{m-j-1} g(x+ R_i(y))\Big) \, \Big(\prod_{i=0}^jQ_{t/N^{d-d_i}} g(x)\Big)\, d\mu_{[t_1]}(y) d\mu(x),\\
\label{I_3^j}
&I_{2}^j \le \|Q_{t_0/N^{d-d_j}} g  -  Q_{t/N^{d-d_j}} g \|_{L^2(\KK)} \quad \text{and} \quad 
I_{3}^j \lesssim_{\mathcal P} \bigl(t_0/t_1^d\bigr)^{b} \quad \text{for some} \quad b = b_{\mathcal P} > 0,
\end{align}
again if \eqref{N-delta} holds.
Strictly speaking, the estimate \eqref{I_3^j} for $I_3^j$ does not follow from Theorem \ref{sobolev} when $j=m-1$
since the proof of Theorem \ref{sobolev} assumed that the collection ${\mathcal P}$ of polynomials consisted
of at least two polynomials. Nevertheless the bound \eqref{I_3^j} holds when $j=m-1$. To see this, we apply
the Cauchy--Schwarz inequality and Plancherel's theorem to see that 
\begin{gather*}
| I_3^{m-1} |^2 \ \le \ \frac{1}{N^{d}} \int_\KK \bigl| \int_\KK  [{\rm Id} - Q_{t_0N^{d_{m-1}}}] f (x + P_1(y))
\, d\mu_{[t_1N ]}(y)\bigr|^2\, d\mu(x)\\
= \frac{1}{N^{d}} \int_{|\xi|\ge (N^{d_{m-1}} t_0)^{-1}} |{\widehat{f}}(\xi)|^2 |m_{N,t_1}(\xi)|^2\,
d\mu(\xi), \quad  \text{ where } \quad  m_{N, t_1}(\xi) := \int_{B_1(0)} {\rm e}(P_1(t_1N y)\xi)\, d\mu(y).
\end{gather*}
The oscillatory integral bound \eqref{osc-int-est} implies that $|m_{N,t_1}(\xi)| \lesssim_{\mathcal P} (t_0/t_1)^b$
whenever $|\xi| \ge (N^{d_{m-1}} t_0)^{-1}$ and so \eqref{I_3^j} for $I_3^j$ follows when $j=m-1$ since
$\|f\|_{L^2(\KK)}^2 \le N^d$.

\paragraph{\bf Step 5} From \eqref{I-decomp} and the iterated decomposition of $I_1$, we see that $t_1^{-1} I \ge A + B + C$, where
$$
 A = \int_{\KK} g(x)  \prod_{j=0}^{m-1}Q_{t/N^{d-d_j}} g(x) \, d\mu(x) \ \ge \ \varepsilon^{m+1}
$$
by \eqref{poisson}, and for some $C_{\mathcal P}>0$ we have 
\[
|B| \le C_{\mathcal P} \sum_{j=0}^{m-1} \|Q_{t_0/N^{d-d_j}} g  -  Q_{t/N^{d-d_j}} g \|_{L^2(\KK)}
\quad \text{and} \quad |C| \ \le C_{\mathcal P} \ 
\bigl(t_0/t_1^d\bigr)^{b} \le \varepsilon^{m+1}/4
\]
if $t_0 \le c_0 \, \varepsilon^{(m+1)/b} \, t_1^d$ and $c_0^b C_{\mathcal P}<1/4$ and \eqref{N-delta} holds.

Finally we claim that we can find a triple $t_0 \ll t_1 \ll t$ of small scales such that $|B| \le \varepsilon^{m+1}/4$.
If we are able to do this, then $I \ge \varepsilon^{m+1} t_1/2$ and the proof is complete.

Define $v:=-C_0\log_q( c_0 \varepsilon^{(m+1)/b})$ for some large constant $C_0\gg d$.
Choose a sequence of small scales $t_0 = q^{-\ell_j}$ and $t_1 = q^{-k_j}$ and $t=q^{-u_j}$ satisfying
\begin{align}
\label{eq:4}
\begin{gathered}
0\le u_1< dk_1+v < \ell_1 <u_2< dk_2+v < \ell_2 < \ldots < u_n<dk_n+v<\ell_n<\ldots \\
\text{ and } \qquad\ell_{n+1}\le \ell_n-C_0\log_q( c_0 \varepsilon^{(m+1)/b}).
\end{gathered}
\end{align}
 Taking $L\in\NN$ such that $L=\lfloor 16C_{\mathcal P}m^2\varepsilon^{-2(m+1)}\rfloor +1$ we claim that there exists $j\in\bra{L}$ such that
\begin{align}
\label{eq:3}
C_{\mathcal P}\sum_{n=0}^{m-1} \|Q_{q^{-\ell_j} N^{-(d-d_n)}} g  -  Q_{q^{-u_j }N^{-(d-d_n)}} g \|_{L^2(\KK)} <\varepsilon^{m+1}/4.
\end{align}
Indeed, suppose for a contradiction that \eqref{eq:3} does not hold. Then for all $j\in\bra{L}$ by the Cauchy--Schwarz inequality we have
$$
\varepsilon^{2(m+1)} \le 16C_{\mathcal P}^2m
\sum_{n=0}^{m-1} \|Q_{q^{-\ell_j} N^{-(d-d_n)}} g  -  Q_{q^{-u_j }N^{-(d-d_n)}} g \|_{L^2(\KK)}^2. 
$$
Then
\begin{gather*}
L \varepsilon^{2(m+1)} \le 16C_{\mathcal P}^2m\sum_{j=1}^L
\sum_{n=0}^{m-1}   \|Q_{q^{-\ell_j} N^{-(d-d_n)}} g  -  Q_{q^{-u_j} N^{-(d-d_n)}} 
g \|_{L^2(\KK)}^2 \\
 = 16C_{\mathcal P}^2m \sum_{n=0}^{m-1}   \int_\KK |{\widehat{g}}(\xi)|^2 \sum_{j=1}^L \bigl| \ind{[q^{\ell_j} N^{d-d_n}]}(\xi) - 
\ind{[q^{u_j} N^{d-d_n}]}(\xi) \bigr|^2 d\mu(\xi) \le 16C_{\mathcal P}^2m^2 \|g\|_{L^2(\KK)}^2
\end{gather*}
and this implies $L \le 16C_{\mathcal P}^2m^2\varepsilon^{-2(m+1)}$ since $\|g\|_{L^2(\KK)} \le 1$, which is impossible by our choice of $L$. 

Therefore there exists $j\in\bra{L}$ and  a corresponding  triple of scales $t_0 = q^{-\ell_j}\ll t_1 = q^{-k_j}\ll t=q^{-u_j}$ satisfying the desired properties for which \eqref{eq:3} is true. In particular, 
$|B| \le \varepsilon^{m+1}/4$ holds.

\paragraph{\bf Step 6} Furthermore, with these scales by \eqref{eq:4},  we have $t_0= q^{-\ell_j} \gtrsim (c_0 \varepsilon^{m+1})^{O_{\mathcal P}(m^2 \varepsilon^{-2(m+1)})}$.
In order to ensure that \eqref{N-delta} holds for every iteration in the decomposition, we set
$$
N(\varepsilon, {\mathcal P}) \ := \ (c_0 \varepsilon^{m+1})^{- O_{\mathcal P}( m^2 \varepsilon^{-2(m+1)})}
$$
so that for every $N\ge N(\varepsilon, {\mathcal P})$ condition \eqref{N-delta} holds.
Hence
$$
I \gtrsim  \varepsilon^{m+1} t_1 \gtrsim \varepsilon^{m+1} t_0 \gtrsim \varepsilon^{m+1} 
(c_0\varepsilon^{m+1})^{O_{\mathcal P}(m^2 \varepsilon^{-2(m+1)})},
$$
establishing the desired bound \eqref{mult-recurrence-function} with $\delta =  \varepsilon^{C_1 \varepsilon^{-2m-2}}$
for some $C_1>0$ depending only on ${\mathcal P}$. 

This completes the proof of Theorem \ref{thm:main}.

\paragraph{\bf Conflict of interests} None.

\paragraph{\bf Financial Support} Mariusz Mirek is supported by the NSF grant DMS-2154712 and by the NSF CAREER grant DMS-2236493. Sarah
Peluse is supported by the NSF grant DMS-2401117 and was supported by
the NSF Mathematical Sciences Postdoctoral Research Fellowship Program
under grant DMS-1903038. James Wright is supported by a Leverhulme
Research Fellowship RF-2023-709$\backslash$9

\subsection*{Acknowledgments}
We thank the referees for
careful reading of the manuscript and useful remarks that led to the
improvement of the presentation.

\end{document}